\documentclass[11pt,oneside,english,reqno]{amsart}

\usepackage[T1]{fontenc}
\usepackage{lmodern}
\usepackage[a4paper]{geometry}
\usepackage{mathrsfs}
\usepackage{amstext}
\usepackage{amsthm}
\usepackage{amssymb}
\usepackage{bm}
\usepackage{bbm}
\usepackage{color}
\usepackage{hyperref}
\usepackage{dsfont}
\usepackage{enumerate}
\usepackage{verbatim} 
\usepackage{stmaryrd}
\usepackage{enumitem}
\usepackage[dvipsnames]{xcolor}
\usepackage{mathtools}
\usepackage{nicefrac}
\usepackage{cleveref}

\makeatletter
\numberwithin{equation}{section}
\numberwithin{figure}{section}
\theoremstyle{plain}
\newtheorem{thm}{\protect\theoremname}
\theoremstyle{definition}
\newtheorem{defn}[thm]{\protect\definitionname}
\theoremstyle{remark}
\newtheorem{rem}[thm]{\protect\remarkname}
\theoremstyle{plain}
\newtheorem{lem}[thm]{\protect\lemmaname}
\theoremstyle{plain}
\newtheorem{prop}[thm]{\protect\propositionname}
\theoremstyle{plain}
\newtheorem{cor}[thm]{\protect\corollaryname}
\theoremstyle{plain}
\newtheorem{ex}[thm]{\protect\examplename}
\theoremstyle{plain}
\newtheorem{ass}[thm]{\protect\assumptionname}
\numberwithin{thm}{section}
\newtheorem{convention}[thm]{\protect\conventionname}
\numberwithin{thm}{section}
\makeatother
\crefname{lem}{Lemma}{Lemmas}
\usepackage{babel}
\providecommand{\corollaryname}{Corollary}
\providecommand{\definitionname}{Definition}
\providecommand{\lemmaname}{Lemma}
\providecommand{\propositionname}{Proposition}
\providecommand{\remarkname}{Remark}
\providecommand{\theoremname}{Theorem}
\providecommand{\examplename}{Example}
\providecommand{\assumptionname}{Assumption}
\providecommand{\conventionname}{Convention}

\newcommand{\bra}[1]{{\llbracket #1 \rrbracket}}

\newcommand{\cA}{\mathcal{A}}
\newcommand{\caa}{\mathcal{A}}

\newcommand{\cC}{\mathcal{C}}
\newcommand{\C}{\mathcal{C}}

\newcommand{\cF}{\mathcal{F}}
\newcommand{\cff}{\mathcal{F}}
\newcommand{\cG}{\mathcal{G}}

\newcommand{\cL}{\mathcal{L}}
\newcommand{\cll}{\mathcal{L}}
\newcommand{\cM}{\mathcal{M}}

\newcommand{\cP}{\mathcal{P}}

\newcommand{\cS}{\mathcal{S}}
\newcommand{\cT}{\mathcal{T}}

\newcommand{\cZ}{\mathcal{Z}}
\newcommand{\czz}{{\cZ}}

\renewcommand{\AA}{\mathbb{A}}

\newcommand{\EE}{\mathbb{E}}
\newcommand{\E}{\EE}
\newcommand{\FF}{\mathbb{F}}
\newcommand{\GG}{\mathbb{G}}

\newcommand{\JJ}{\mathbb{J}}

\newcommand{\NN}{\mathbb{N}}

\newcommand{\PP}{\mathbb{P}}

\newcommand{\RR}{\mathbb{R}}

\newcommand{\WW}{\mathbb{W}}

\newcommand{\ZZ}{\mathbb{Z}}

\newcommand{\mcC}{\mathcal{C}}

\newcommand{\mcF}{\mathcal{F}}

\newcommand{\mcL}{\mathcal{L}}

\newcommand{\mcS}{\mathcal{S}}

\newcommand{\mbA}{\mathbb{A}}

\newcommand{\mbE}{\mathbb{E}}
\newcommand{\mbF}{\mathbb{F}}
\newcommand{\mbG}{\mathbb{G}}

\newcommand{\mbN}{\mathbb{N}}

\newcommand{\mbP}{\mathbb{P}}

\newcommand{\mbR}{\mathbb{R}}

\newcommand{\vertiii}[1]{{\left\vert\kern-0.25ex\left\vert\kern-0.25ex\left\vert #1 
		\right\vert\kern-0.25ex\right\vert\kern-0.25ex\right\vert}}

\newcommand{\RN}[1]{%
	\textup{\uppercase\expandafter{\romannumeral#1}}%
}

\DeclarePairedDelimiter\floor{\lfloor}{\rfloor}

\newcommand{\eps}{\varepsilon}
\newcommand{\dd}{\mathop{}\!\mathrm{d}}

\newcommand{\var}{\mathrm{var}}
\newcommand{\law}{\cll}

\newcommand{\variable}{\,\cdot\,}

\let\div\relax
\DeclareMathOperator{\div}{div}

\DeclareMathOperator{\esssup}{ess sup}

\def \C{\mathbb{C}}

\def \R{\mathbb{R}}

\hypersetup{
	colorlinks   = true,
	citecolor    = blue,
	linkcolor    = blue
}

\title[Quantitative Chaos for singular systems driven by fBm]{Quantitative Propagation of Chaos for Singular Interacting Particle Systems Driven by Fractional Brownian Motion}

\author{Lucio Galeati}
\address{University of L'Aquila, Italy}
\email{lucio.galeati@univaq.it}

\author{Khoa L\^e}
\address{University of Leeds, UK}
\email{k.le@leeds.ac.uk}

\author{Avi Mayorcas}
\address{University of Bath, UK}
\email{am2735@bath.ac.uk}

\date{\today}

\begin{document}

\begin{abstract}
	We consider interacting particle systems driven by i.i.d. fractional Brownian motions, subject to irregular, possibly distributional, pairwise interactions.
	We show propagation of chaos and mean field convergence to the law of the associated McKean--Vlasov equation, as the number of particles $N\to\infty$, with quantitative sharp rates of order $N^{-1/2}$.
	Our results hold for a wide class of possibly time-dependent interactions, which are only assumed to satisfy a Besov-type regularity, related to the Hurst parameter $H\in (0,+\infty)\setminus \NN$ of the driving noises.
	In particular, as $H$ decreases to $0$, interaction kernels of arbitrary singularity can be considered, a phenomenon frequently observed in regularization by noise results.
	Our proofs rely on a combinations of Sznitman's direct comparison argument with stochastic sewing techniques.\\[1ex]
	\textbf{MSC2020 Subject Classification:} 60H10, 82C22, 60H50, 60G22, 60L90.\\[1ex]
\end{abstract}

\maketitle

\section{Introduction}\label{sec:introduction}
Consider a system of $N\geq 2$ stochastic particles $X^{(N)} = \{X^{i;N}\}_{i=1}^N$ on $\mbR^d$, starting from independent initial positions $X^{(N)}_0 = \{X^{i}_0\}_{i=1}^N$,  driven by a collection of independent, additive, idiosyncratic noises $W^{(N)} =\{W^{i}\}_{i=1}^N$, whose evolution is described by the system of coupled SDEs 
\begin{equation}\label{eq:intro_IPS}
	X^{i;N}_t = X_0^i + \frac{1}{N-1}\sum_{\substack{j=1\\ j\neq i}}^N\int_0^t b_s(X^{i;N}_s,X^{j;N}_s) \dd s + W^{i}_t,\qquad \text{for all } i=1,\ldots,N,
\end{equation}
for some time-dependent function $b:[0,+\infty)\times\RR^d\times\RR^d\to \RR^d$, $b=b_t(x,y)$.
We will often refer to $b$ as an interaction kernel, or simply an interaction, or drift; notice that in \eqref{eq:intro_IPS} there is no self-interaction, since the $i=j$ term has been removed from the sum.

The study of such systems is of widespread importance in mathematics and applied sciences,
for example in relation to galactic dynamics \cite{jeans_15_star}, plasma physics \cite{vlasov_38}, fluid dynamics \cite{onsager_49_stat}, bacterial chemotaxis in biology, as well as approximation theory, random matrix theory and crystallisation \cite{serfaty_18_systems} and social sciences \cite[Sec.~5]{chaintron_22_propagation_II}.
We refer to the reviews \cite{chaintron_22_propagation_II,jabin2017mean} for more details and further examples.

As system \eqref{eq:intro_IPS} is $Nd$-dimensional, with very large $N$, it is often impractical or almost impossible to solve it directly, even numerically.
Besides, in relevant applications one is often not interested in obtaining a full description of the system $X^{(N)}_t$, but rather understanding the evolution of macroscopic quantities, or the statistical behaviour of a typical particle $X^{1,N}$ in the system.
In particular, one often seeks an asymptotic \emph{effective description} of \eqref{eq:intro_IPS} as $N\to\infty$, which hopefully can be captured by a much lower-dimensional system.

To this end, it is convenient to introduce the \emph{empirical measure} associated to \eqref{eq:intro_IPS}, given by
%
\begin{equation*}
	\mu^N_t \coloneqq \frac{1}{N}\sum_{i=1}^N \delta_{X^{i;N}_t}.
\end{equation*} 
Notice that, for fixed $N$, $\mu^N_t$ is a random probability measure on $\RR^d$, describing in an averaged way the whole ensemble $X^{(N)}$, rather than single particles.
The \emph{mean field formalism} asserts that, under suitable assumptions on $X^{(N)}_0$ and $b$, the sequence $\{\mu^N_t\}$ should converge as $N\to\infty$ to a deterministic measure $\bar{\mu}_t$, in a law of large number fashion; this is often referred to as a mean field limit or mean field convergence.
Furthermore, $\bar{\mu}_t$ should be given by the law $\cL(\bar X_t)$ of the \emph{non-linear process} $\bar{X}$, solving the McKean--Vlasov equation
\begin{equation}\label{eq:intro_MKV}
	\bar{X}_t = X_0 + \int_0^t \int_{\mbR^d} b_s(\bar{X}_s,y) \dd \bar{\mu}_s(y)\dd s + W_t,\qquad \bar{\mu}_t = \mcL(\bar{X}_t),
\end{equation}
where $(X_0,W)$ is sampled as $(X^1_0,W^1)$.
Observe that the SDE \eqref{eq:intro_MKV} is set on $\RR^d$, instead of $\RR^{Nd}$ like \eqref{eq:intro_IPS}, but it is non-local in the probability space; in other words, we are trading the higher-dimensional, simpler description given by \eqref{eq:intro_IPS} for the lower-dimensional, more complex one of \eqref{eq:intro_MKV}.

An alternative asymptotic description, similarly capturing the idea that the particles $X^{i;N}$ become ``asymptotically independent'' and behave like $\bar{X}$ as $N\to \infty$, is captured by the concept of \emph{propagation of chaos};
loosely speaking, it asserts that if $\{X^i_0\}$ are sampled i.i.d., then the same holds for $t>0$ asymptotically in the system size.
More precisely, for any fixed $k\in\NN$, one requires that
\begin{equation}\label{eq:intro_poc_definition}
	\cL(X^{1;N}_t,\ldots, X^{k;N}_t)\to \bar\mu_t^{\otimes k}\quad \text{as } N\to\infty,
\end{equation}
with convergence of measures being understood in an appropriate weak sense.
In many cases, mean field convergence and propagation of chaos can be shown to be equivalent, see \cite{sznitman1991topics,jabin2017mean,chaintron_diez_22_propagation_I,chaintron_22_propagation_II}. 

In the specific case where either the evolution is deterministic or the noises are Brownian, namely for $W^{i}=\sqrt{2\eps} B^i$, under mild assumptions, one can typically provide an alternative self-contained description of $\bar\mu_t$ by solving the \emph{non-linear Fokker--Planck equation}
\begin{equation}\label{eq:intro_FP}
	\partial_t \bar{\mu}_t + \nabla \cdot \big(\bar{\mu}_t\, b_t(\,\cdot\,,\bar{\mu}_t)\big) = \eps \Delta \bar{\mu}_t,\quad
	b_t(x,\bar{\mu}_t) \coloneqq \int_{\mbR^d} b_t(x,y) \dd \bar{\mu}_t(y)
\end{equation}
with prescribed initial condition $$\bar\mu_0\coloneq\cL(X_0).$$
In this case, under appropriate regularity assumptions, \eqref{eq:intro_MKV} and \eqref{eq:intro_FP} are in one-to-one correspondence: a solution to \eqref{eq:intro_MKV} can be used to construct one to \eqref{eq:intro_FP} and vice versa, see for example \cite{barbu_rockner_20_nonlinear}.

Early works in the literature established, in a fully rigorous manner, the connection between \eqref{eq:intro_IPS}, \eqref{eq:intro_MKV} and \eqref{eq:intro_FP}, showing both mean field convergence and propagation of chaos, in the case where the interaction $b$ is sufficiently regular, say Lipschitz; see among others \cite{mckean1966class,braun_hepp_77_vlasov,sznitman1991topics}.
In more recent years, there has been tremendous advance in the case of \emph{singular interactions}, which arise naturally in the aforementioned applications. Most of the relevant works exploit crucially the hypothesis that the noise is Brownian, either by leveraging It\^o calculus, or the well-posedness and regularity of the PDE \eqref{eq:intro_FP}; we postpone a deeper discussion on the topic to Section \ref{sec:literature}.

Interestingly, it was pointed out recently in \cite{coghi2020pathwise}, revisiting old ideas by Tanaka \cite{tanaka1984Limits}, that for regular $b$ the connection between \eqref{eq:intro_IPS} and \eqref{eq:intro_MKV} is valid for almost \emph{any} choice of the driving noises $W^i$.
In other words, mean field convergence is a statistically universal phenomenon, intrinsic to the mean field structure of the interaction and the imput data $\{X^{i;N}_0,W^i\}$ satisfying a law of large numbers, but not specific to the noise being Brownian, or an underlying PDE connection being present.
In fact, one can allow many choices of non-Markovian, non-martingale noises $W^i$, for which any link to \eqref{eq:intro_FP} and It\^o calculus are immediately lost; a prominent example, the one we will focus on in this paper, is given by fractional Brownian motion (fBm) $W^H$ of Hurst parameter $H\in (0,+\infty)\setminus\NN$. See Section \ref{subsec:fbm} for its definition.

The work \cite{coghi2020pathwise} leaves open the question of what happens for less regular interactions, the most relevant ones for applications.
In this case, one can hope to exploit the regularizing effect of fBm trajectories on SDEs, a phenomenon often referred to as \emph{regularization by noise}, to recover well-posedness of \eqref{eq:intro_IPS}-\eqref{eq:intro_MKV} and subsequently prove mean field convergence and propagation of chaos.
We started this program in \cite{GaHaMa2023}, which established the well-posedness of the McKean--Vlasov equation \eqref{eq:intro_MKV} for a large class of singular interactions $b$; but the question of establishing \eqref{eq:intro_poc_definition} was left open therein.
The aim of the current paper is to close this gap.

Our results are summarised by the next theorem, which for the sake of presentation is stated somewhat loosely.
Therein, $B^\alpha_{\infty,\infty}(\RR^{2d};\RR^d)$ denote inhomogeneous Besov--H\"older spaces; we refer to Section~\ref{sec:notation} for a thorough explanation of the main relevant function spaces, notations and conventions adopted in this paper.

\begin{thm}\label{thm:main}
	Let $T\in (0,\infty)$, $(\Omega,\mcF,\mbP)$ be a probability space carrying $\{X^{i}_0\}_{i=1}^\infty, X_0$, a collection of i.i.d random variables in $\mbR^d$, and $\{W^{i,H}\}_{i=1}^\infty, W^H$, a collection of i.i.d fractional Brownian motions of Hurst parameter $H\in (0,+\infty)\setminus \NN$. Further assume that $b\in L^q([0,T];B^\alpha_{\infty,\infty}(\mbR^{2d};\mbR^d))$, for some $\alpha,\,q$ satisfying
	\begin{align}\label{eq:intro_assumption}
		\alpha\in (-\infty,1),\quad  q\in (1,2],\quad \alpha>1-\frac{1}{H {q^\prime}},\quad \frac{1}{q'}\coloneqq  1-\frac{1}{q}.
	\end{align}
	Then the following statements hold:
	\begin{enumerate}[label=\roman*)]
		\item\label{it:intro_thm_i} For any $N\geq 2$, there exists a unique strong solution $X^{(N)} =\{X^{i;N}\}_{i=1}^N$ to the interacting particle system \eqref{eq:intro_IPS}.
		\item\label{it:intro_thm_ii} There exists a unique strong solution $\bar{X}$ to the McKean--Vlasov equation \eqref{eq:intro_MKV}.
		\item\label{it:intro_thm_iii} Let $\bar{X}^1$ denote the unique solution to \eqref{eq:intro_MKV}, with $X_0,\, W^H$ replaced by $X_0^1$, $W^{1,H}$, coming from Point~\ref{it:intro_thm_ii}; then for any $m\in [1,\infty)$ and $N\geq 2$ it holds that
		\begin{align}\label{eq:intro_poc}
			\mbE\Bigg[ \,\sup_{t\in [0,T]} |X^{1;N}_t-\bar{X}^{1}_t|^m \Bigg]^{\tfrac{1}{m}} \lesssim N^{-1/2}.
		\end{align}
		\item\label{it:intro_thm_iv} For any $m\in [1,\infty)$, $N\geq 2$ and $\varphi \in W^{1,\infty}(\mbR^d;\mbR)$ it holds that
		\begin{align}\label{eq:intro_mfc}
			\mbE\Bigg[\,\sup_{t\in [0,T]} |\langle \varphi,\mu_t^N\rangle-\langle\varphi,\bar \mu_t\rangle|^m \Bigg]^{\tfrac{1}{m}}\,
			\lesssim \|\varphi\|_{W^{1,\infty}_x} N^{-1/2}
		\end{align}
		where $$\langle \varphi,\nu\rangle\coloneqq  \int_{\RR^d} \varphi(x) \nu(\dd x)$$ for any finite signed measure $\nu$.
	\end{enumerate}
\end{thm}

While Theorem~\ref{thm:main} faithfully captures the spirit of our main results, its formulation is intentionally imprecise.
Firstly, since we can allow for distributional drifts whenever $\alpha<0$, pointwise evaluations of the form $b_s(X^{i;n}_s,X^{j;N}_s)$ are not meaningful, so even the notion of solution to \eqref{eq:intro_IPS}-\eqref{eq:intro_MKV} is non-obvious.
Secondly, we have not yet specified exactly what we mean by uniqueness of solutions.
Both of these issues will be properly addressed in Section~\ref{sec:main_proofs}, which contains more rigorous and precise statements, overall comprising Theorem~\ref{thm:main}.

We will illustrate the main ideas of the proof in Section~\ref{subsec:ideas_proof}; let us first collect a number of remarks and examples discussing the relevance of the result.

\begin{rem}\label{rem:meaning_assumption_intro}
	Condition~\eqref{eq:intro_assumption} first appeared in the study of standard SDEs driven by fBm in~\cite{galeati2022subcritical}.
	Arguing as therein, it can be shown that $\alpha>1-1/(Hq')$ identifies a \emph{subcritical regime}, for both~\eqref{eq:intro_IPS} and~\eqref{eq:intro_MKV}, obtained by rescaling the interaction $b$ in accordance with the self-similarity of the driving noise $W^H$.
	In practical terms, it gives a quantitative version of the known principle ``the rougher the noise, the better the regularization'' from the regularization by noise literature: at fixed $q$, as $H\downarrow 0$, \eqref{eq:intro_assumption} allows for more negative $\alpha$, namely more singular interactions $b$. In contrast, as $H\uparrow +\infty$, the condition degenerates to $\alpha>1$, which gives back the standard Lipschitz regime where the result holds classically. In particular, the roughness of the noise $W^H$ (quantified by $H$) is inversely proportional to its regularising effect on systems \eqref{eq:intro_IPS}-\eqref{eq:intro_MKV}.
	The restriction $q\in (1,2]$ instead does not come from scaling considerations; we refer to Section~\ref{subsec:open_problems} for a deeper discussion on this point.
\end{rem}

\begin{rem}\label{rem:large_q_embedding}
	Since we always work on a given, finite time interval $[0,T]$, for $q>2$ one can use the embedding $L^q([0,T];B^\alpha_{\infty,\infty}(\mbR^{2d};\mbR^d)) \hookrightarrow L^2([0,T];B^\alpha_{\infty,\infty}(\mbR^{2d};\mbR^d))$ to modify \eqref{eq:intro_assumption} as follows:
	\begin{align*}
		\alpha\in (-\infty,1),\quad  q\in (1,\infty],\quad \alpha>1-\frac{1}{H (2\vee {q^\prime})}.
	\end{align*}	  
	In particular, for autonomous interactions $b\in B^\alpha_{\infty,\infty}(\RR^{2d};\RR^d)$, one should take
	\begin{equation}\label{eq:intro_condition_autonomous}
		\alpha>1-\frac{1}{2H}\quad \iff \quad H<\frac{1}{2(1-\alpha)}; 
	\end{equation}
	this is in line with the condition first obtained in \cite{Catellier2016} for standard SDEs driven by fBm.
\end{rem}
\begin{rem}\label{rem:intro_MKV}
	Although the main focus of this paper are the mean field convergence and propagation of chaos underlying systems \eqref{eq:intro_IPS}-\eqref{eq:intro_MKV}, even the well-posedness of \eqref{eq:intro_MKV} for our class of interactions $b$ is a partially novel result, not covered by our previous works~\cite{GaHaMa2023,galeati2022subcritical}.	In fact, one cannot develop the same fixed point argument as therein, which first requires to ``freeze'' the measures $\{\bar\mu_t\}_t$ and consider the drift $$\tilde b_t(x)\coloneqq  \langle b_t(x,\cdot), \bar\mu_t\rangle$$; this is because, without any information on $\bar{\mu}_t$, one cannot make sense of the duality pairing $\langle b_t(x,\cdot), \bar\mu_t\rangle$, so that it is not a priori clear if $\tilde b$ is well-defined, even in the sense of distributions.
\end{rem}
\begin{rem}\label{rem:convergence_mode}
	We formulated our mean field convergence rate in \eqref{eq:intro_mfc} by fixing a Lipschitz observable $\varphi$, and then estimating the $m$-th moment of the $\RR$-valued random variable $|\langle \varphi,\mu^N_t-\bar \mu_t\rangle|$.
	This is not the only option, as similar arguments allow for the use of other quantities, e.g. by considering $\| \mu^N_t-\bar\mu_t\|_{H^{-s}}$ for suitable negative Sobolev spaces $H^{-s}(\RR^d)$; see Corollary~\ref{cor:mean_field_negative_sobolev} for more details.
\end{rem}

\begin{rem}\label{rem:rate}
	The convergence rate $N^{-1/2}$ appearing in \eqref{eq:intro_mfc} is optimal.
	Indeed, even when the interaction $b$ is smooth, it was shown by Tanaka \cite{tanaka1984Limits} in the Brownian setting that that the sequence $N^{-1/2}(\mu^N-\bar{\mu})$ converges in law to a non-trivial Gaussian measure; the result was then extended to arbitrary additive noise (including fBm as considered here) in \cite[Sec.~5]{coghi2020pathwise}.
	Let us point out that this is not in contradiction with the possibility to achieve better rates for other quantities of interest; indeed, Lacker \cite{Lacker2023} recently showed that, in the Brownian setting, for suitable interactions $b$ it holds that
	\begin{equation}\label{eq:intro_lacker}
		\Big\| \cL(X^{1;N}_t,\ldots, X^{k;N}_t)- \bar\mu_t^{\otimes k} \Big\|_{TV} \lesssim \frac{k}{N}.
	\end{equation}
\end{rem}

\begin{rem}
	Entropy methods were recently employed by Han \cite{han2022entropic} to show propagation of chaos for fBm driven SDEs with rough drifts; therein, a relative entropy decay of order $N^{-1}$ (implying one of order $N^{-1/2}$ in total variation) was shown, for either locally bounded drifts of at most linear growth when $H\in (0,\nicefrac{1}{2}]$, or suitable space-time H\"older drifts when $H\in (\nicefrac{1}{2},1)$.
	However, \cite{han2022entropic} requires subGaussian initial conditions $X_0^i$, and the estimate only holds on some interval $[0,T^\ast]$, where $T^\ast$ depends on the distribution of $X_0$; moreover, differently from our setting, \cite{han2022entropic} does not cover drifts of negative regularity, nor with poor time integrability.
	
	Since completion of this manuscript, \cite{hu2024mimicking} obtained a similar result to \eqref{eq:intro_lacker} in the case of interacting particle systems driven by fBm, see \cite[Thm.~5.5]{hu2024mimicking}.
	The assumptions imposed on $b$ are quite abstract, but it seems plausible for them to be satisfied by the same class of drifts considered in \cite{han2022entropic}, cf. \cite[Ex.~2.6]{hu2024mimicking}, and to work on any finite time interval $[0,T]$.
\end{rem}

\begin{rem}\label{rem:intro_poc}
	Estimate \eqref{eq:intro_poc} encodes a quantitative propagation of chaos convergence.
	Indeed, since system \eqref{eq:intro_IPS} is exchangeable, the same estimate holds with $X^{1;N}$, $\bar{X}^1$ replaced by $X^{i;N}$, $\bar{X}^i$ respectively, for each $i=1,\ldots, N$.
	Since the $\bar{X}^i$ solve the same equation \eqref{eq:intro_MKV} but associated to independent data $(X_0^i,W^{i,H})$, we have $\cL(\bar{X}^1_t,\ldots,\bar{X}^k_t)=\bar\mu^{\otimes k}_t$; combined with \eqref{eq:intro_poc} and the triangle inequality, this yields that
	\begin{align*}
		\WW_m\Big( \cL(X^{1;N}_t,\ldots,X^{k;N}_t),\bar\mu^{\otimes k}_t\Big)
		\leq \Big\| (X^{1;N}_t,\ldots,X^{k;N}_t) - (\bar{X}^1_t,\ldots,\bar{X}^k_t)\Big\|_{L^m(\Omega)}
		\lesssim \frac{k}{\sqrt N}
	\end{align*}
	where $\WW_m$ denotes the $m$-th Wasserstein distance. Note however that, in light of \eqref{eq:intro_lacker}, we do not expect this rate to be sharp.
\end{rem}

\begin{rem}
	Theorem \ref{thm:main} allows for a large class of initial data and quite general interactions.
	Indeed, we do not require anything on $\{X^i_0\}_i$ apart from being i.i.d. (an assumption which we expect it to be possible to relax, cf. Section \ref{subsec:open_problems}). In particular, $\cL(X^i_0)$ does not need to be regular, nor possess finite moments or finite entropy.
	Similarly, as soon as $b$ satisfies Condition \eqref{eq:intro_assumption}, no further structural assumptions are needed.
	For instance, we do not impose conditions on the divergence of $b$, nor require that the kernel be attractive or repulsive, etc.
	Finally, let us point out that estimates \eqref{eq:intro_poc}-\eqref{eq:intro_mfc} provide \emph{functional} mean field limits and propagation of chaos rates, since suprema over $t\in [0,T]$ always appear inside the expectation. 
\end{rem}

We next discuss the most relevant classes of interactions $b$ satisfying Condition~\eqref{eq:intro_assumption}.

\begin{ex}[Convolutional, Pointwise and Statistical Interactions]\label{ex:convolutional}
	An important class of interactions $b\in L^q([0,T];B^\alpha_{\infty,\infty}(\RR^{2d};\RR^d))$ covered by Theorem \ref{thm:main} are those which formally decompose as
	\begin{equation}\label{eq:intro_drift_decomposition}
		b_t(x,y)\coloneqq f_t(x) + g_t(y)+h_t(x-y) \vspace{0.5em}
	\end{equation}
	for $f,\,g,\,h \in L^q([0,T]; B^\alpha_{\infty,\infty}(\mbR^{d};\mbR^d))$, with $(q,\alpha)$ satisfying \eqref{eq:intro_assumption}; see Lemma~\ref{lem:besov_dimension_reduction} in Appendix~\ref{app:holder_besov} for a self-contained proof of this fact.
	For $b$ decomposing as in \eqref{eq:intro_drift_decomposition}, we may formally write 
	\begin{align*}
		b_t(x,\mu) = f_t(x) + \langle g_t,\mu\rangle + (h_t\ast \mu)(x).
	\end{align*}
	We see that $f$ takes the role of a potential in which particles are immersed, $g$ is a statistic of the measure $\mu$, and most importantly $h$ is the convolutional kernel associated to an homogeneous pairwise interaction.
\end{ex}

\begin{ex}[H\"older and Bounded Interactions]\label{ex:bounded}
	We recall that, for $\alpha\in (0,1)$, $B^\alpha_{\infty,\infty}(\RR^{2d};\RR^d)$ coincides with the space $C^\alpha_b(\RR^{2d};\RR^d)$ of bounded, $\alpha$-H\"older continuous functions.
	For instance, in the Brownian case $H=1/2$, Theorem~\ref{thm:main} allows for interactions $b\in L^2([0,T]; C^\alpha_b)$ for arbitrary $\alpha>0$. We point out that this result, in the Brownian case, was already essentially obtained by \cite[Thm.~2.1]{holding_16_propagation}. On the other hand, since $L^\infty(\RR^{2d};\RR^d)\hookrightarrow B^{-\delta}_{\infty,\infty}(\RR^{2d};\RR^d)$ for all $\delta>0$, we can allow for spatially bounded $b\in L^2([0,T]; L^\infty_x)$ as soon as $H<1/2$.
\end{ex}

\begin{ex}[Integrable Riesz Potentials]\label{ex:coloumb}
	An important family of interactions, in the context of Example~\ref{ex:convolutional}, are those associated to Riesz potentials. Consider
	\begin{equation}
		b(x,y) = (\mbA \nabla  h)(x-y) \quad \forall \,\, x\neq y  \in \mbR^d,
	\end{equation}
	where $\mbA$ is any unitary matrix and
	\begin{equation*}
		|h(x)| \leq  \frac{1}{|x|^{s}}\quad \text{with}\quad s \in (0,d) \quad \text{or} \quad |h(x)| \leq  1\wedge |\log(|x|)|,\quad \forall\, x \neq 0.
	\end{equation*}
	Since $h$ is locally integrable, bounded at infinity and is bounded by an $(-s)$-homogeneous distribution (where we regard the $log$ case as $s=0$), it can be readily checked that it defines an element of $B^{-s}_{\infty,\infty}(\mbR^d;\mbR^d)$.
	Interpreting $\nabla h$ in the sense of distributions, one then has $b \in B^{-s-1}_{\infty,\infty}(\mbR^d;\mbR^d)$.
	In light of condition~\eqref{eq:intro_condition_autonomous} and Example~\ref{ex:convolutional}, the assumptions of Theorem~\ref{thm:main} are satisfied with $q =2$, provided that
	\begin{equation}\label{eq:intro_coulomb_condition}
		H<\frac{1}{2(2+s)}.
	\end{equation}
	The particular case of Coulomb interactions, related to gravitational attraction and electro-magnetic repulsion, corresponds to taking $\AA=\pm I_d$ and $s= d-2$ for $d\geq 3$, or the logarithmic potential when $d=2$; condition \eqref{eq:intro_coulomb_condition} then becomes
	\begin{equation}\label{eq:intro_coulomb_specific}
		H<\frac{1}{2d}.
	\end{equation}
	Similarly, in $d=2$ if we take $h\sim \log(|x|)$ and $\AA=\JJ$, where $\JJ$ denotes the symplectic matrix, then $b$ corresponds to the Biot--Savart kernel, related to dynamics of point vortices in $2$-dimensional fluids; condition \eqref{eq:intro_coulomb_specific} becomes $H<1/4$.
\end{ex}
\begin{ex}[Dirac Interactions]\label{ex:dirac}
	Another class of convolutional type kernel are those formally given by
	\begin{equation*}
		b(x,y) = v \delta_0(x-y)\quad \text{for some } \, v\in \mbR^d.
	\end{equation*}
	Since the Dirac delta lies in $B^{-d}_{\infty,\infty}(\mbR^d;\mbR^d)$, by eq.~\eqref{eq:intro_condition_autonomous} Theorem~\ref{thm:main} applies as soon as
	\begin{equation*}
		H <\frac{1}{2(1+d)}.
	\end{equation*}
	We refer the reader to \cite{butkovsky2023stochastic} for a discussion on SDEs with Dirac delta drifts in relation to \emph{skew processes}, and to \cite[Ch.~II]{sznitman1991topics} for their relevance in particle models of one dimensional fluid dynamics, see also \cite{bossy_talay_97_stochastic}.
	Similarly, one may consider distributional derivatives of the Dirac delta $\nabla^k \delta \in B^{-d-k}_{\infty,\infty}(\mbR^d;\mbR^d)$, for all $k\in \mbN$.  
\end{ex}

\begin{ex}[Non-Integrable Riesz Interactions]\label{ex:singular_riesz}
	One can extend the class of Riesz-type interactions allowing potentials which are non-integrable at the origin.
	As above, let us formally write 
	\begin{equation*}
		b(x,y) = (\mbA \nabla  h)(x-y) \quad \forall \,\, x\neq y  \in \mbR^d,
	\end{equation*}
	for any unitary matrix $\mbA$ and
	\begin{equation*}
		h(x) = \frac{1}{|x|^s} \quad \text{with}\quad  s \in (d,\infty) \setminus \mbN, \quad \forall \, x\neq 0.
	\end{equation*}
	In this case, there exists a unique extension of $\nabla  h$ to a genuine distribution on all of $\mbR^d$ which respects the natural scaling of $\nabla h$. Furthermore, this extension defines an element of $B^{-s-1}_{\infty,\infty}(\mbR^d;\mbR^d)$, see Lemma~\ref{lem:non_int_homogeneous} of Appendix~\ref{app:holder_besov}.
	As a consequence, Theorem~\ref{thm:main} still applies under condition~\eqref{eq:intro_coulomb_condition}, in this extended range of values $s \in (d,\infty) \setminus \mbN$.
	The case of integer values is somewhat more subtle, see Remark~\ref{rem:int_homogeneous} for its discussion.
	
	For relevance of these more singular interactions in the context of approximation theory and best packings, we refer the reader to \cite[Sec.~2.2]{serfaty_18_systems} and \cite{saff_kuijlaars_97_distributing,brauchart_hardin_saff_12_next}.
\end{ex}
Since Examples~\ref{ex:dirac}-\ref{ex:singular_riesz} are all interactions of convolution type, one could expect to bootstrap on the regularity of the McKean--Vlasov law to relax the regularity assumptions on $b$, as discussed in the Brownian case for instance in \cite{deraynal_jabir_manozzi_25_multidimensional}. In the fBm case, after the completion of this manuscript this was achieved in \cite[Thm.~5.2 \& Rem.~5.3]{anzeletti2025density}, establishing novel results for the well-posedness of the McKean--Vlasov equation; however, establishing propagation of chaos in the regularity regime considered in \cite{anzeletti2025density} is currently open.
\subsection{Main Ideas of the Proof}\label{subsec:ideas_proof}

The core idea of our proof strategy is to follow Sznitman's direct coupling argument, c.f. \cite[Thm.~1.4]{sznitman1991topics}.
That is, we would like to couple a solution $X^{(N)}$ of the particle system \eqref{eq:intro_IPS} to a collection of $N$ i.i.d. solutions $\bar{X}^{(N)}=\{\bar X^i\}_{i=1}^N$, obtained by solving for each $i$ the McKean--Vlasov equation \eqref{eq:intro_MKV} associated to $(X^i_0,W^{i,H})$.
For simplicity, here we fix $N$ and drop it from the superscript.
As the noises $W^{i,H}$ are additive, they cancel out whenever taking differences $X^i-\bar X^i$, leaving us with
\begin{equation*}\begin{split}	
		X^i_t-\bar{X}^i_t
		&=  \frac{1}{N-1} \sum_{j\neq i } \int_0^t [b(X^i_s,X^j_s) - b(\bar{X}^i_s,\bar{X}^j_s)] \dd s 
		+ \frac{1}{N-1} \sum_{j\neq i} \int_0^t [b(\bar{X}^i_s,\bar{X}^j_s) - b(\bar{X}^i_s,\bar{\mu}_s)] \dd s\\
		& =: R^{1,N}_t + R^{2,N}_t
\end{split}\end{equation*}
Sznitman's original argument would then exploit two crucial passages:
\begin{enumerate}
	\item\label{it:sznit_1} Since $b$ is Lipschitz, taking modulus on both sides, one can bring it inside the integrals to find a pointwise estimate
	\begin{align*}
		\Big| \int_0^t [b(X^i_s,X^j_s) - b(\bar{X}^i_s,\bar{X}^j_s)] \dd s\Big|
		\lesssim \| b\|_{W^{1,\infty}} \sup_{j=1,\ldots,N} \sup_{s\leq t} | X^j_s-\bar{X}^j_s|	
	\end{align*}
	which eventually allows to control the term $R^{1,N}$ in a Gr\"onwall-like fashion, and reduce the problem to estimating $R^{2,N}$.
	\item\label{it:sznit_2} Since $\bar{X}^j$ are i.i.d. and $\cL(\bar{X}^j_t)=\bar\mu_t$, $R^{2,N}$ can be regarded as a sum of mean-zero i.i.d. random variables, for which classical concentration inequalities yield a bound of the form
	\begin{align*}
		\EE\bigg[\, \sup_{t\in [0,T]} |R^{2,N}_t|^2\, \bigg]^{\frac{1}{2}} \lesssim \| b\|_{L^\infty} N^{-1/2}.
	\end{align*}
\end{enumerate}
This strategy is not available in our setting, for multiple reasons.
Firstly, we never work with Lipschitz coefficients (nor e.g. Sobolev ones), which precludes us from performing a passage like~\eqref{it:sznit_1}. In the regime $\alpha<0$, the situation is even worse, because in general for a distribution $b$ we cannot define $|b|$; in particular, we can never bring moduli inside integrals when performing estimates.
Similarly, it is unclear a priori how to perform a step like~\eqref{it:sznit_2}. What is even worse, it is not clear if we can guarantee the well-posedness of solutions $\bar X^i$ to \eqref{eq:intro_MKV}, as explained in Remark~\ref{rem:intro_MKV}, which is necessary in setting up the direct coupling to begin with.

Our strategy overcomes both issues by exploiting the stochasticity of the system, in a regularization by noise fashion.
Indeed, our reference objects $X^i$, $\bar X^i$ are not just any paths, but rather candidate solutions to \eqref{eq:intro_IPS}-\eqref{eq:intro_MKV}.
As argued in \cite{Catellier2016,le2020stochastic,galeati2022subcritical}, this should imply a decomposition of solutions $X^i=\theta^i+W^{i,H}$, where $\theta^i$ are some ``slowly varying remainders''. If such a decomposition holds, $X^i$ should inherit the rough, oscillatory behaviour or $W^{i,H}$, allowing to make sense of stochastic processes of the form
\begin{equation}\label{eq:intro_integrals}
	\int_0^\cdot b_s(X^i_s,X^j_s) \dd s\quad \text{with }i\neq j,
\end{equation}
and show their stability with respect to the arguments $(b, X^i,X^j)$, even in situations where $b$ is no longer a classical function.
Observe in particular that we do not make sense of the pointwise evaluation $b_s(X^i_s,X^j_s)$, but only its integral over time; this is in analogy to well-studied stochastic processes like local times, formally corresponding to $b=\delta_0$.

A second key aspect of our strategy is that we work from the ground up; namely, our primary object of study is the particle system \eqref{eq:intro_IPS} itself, and all relevant estimates will be first performed at this level.
In turn, this informs us of the correct solution Ansatz needed to make sense of the McKean--Vlasov equation \eqref{eq:intro_MKV}, and all properties of solutions to \eqref{eq:intro_MKV} will be derived from those enjoyed by solutions to \eqref{eq:intro_IPS}, through a limiting procedure.
This inverts the standard paradigm, where usually \eqref{eq:intro_MKV} can be solved ex nihilo (as done e.g. in \cite{GaHaMa2023,galeati2022subcritical}); indeed, the analysis for \eqref{eq:intro_MKV} is typically expected to be easier than the one for \eqref{eq:intro_IPS}, and solutions often enjoy increased regularity by bootstrap arguments (possibly leveraging connections to the parabolic PDE \eqref{eq:intro_FP} as in the Brownian case).
This is in stark contrast with our setting, where without the correct solution Ansatz, it is not even clear how to define solutions to \eqref{eq:intro_MKV}.

In practical terms, our procedure to obtain a priori estimates for \eqref{eq:intro_IPS} is inspired by \cite{galeati2022subcritical}, which considered the standard SDE case.
A natural quantification of slowly varying property of $\theta^i$ is by means of conditional expectations and conditional moments, measuring how well $\theta^i$ can be predicted at time $t$ given the history up to $s$.
This can be combined with the \emph{local nondeterminism} of fBm (recalled in Section~\ref{subsec:fbm}) and most importantly \emph{stochastic sewing arguments}, to achieve estimates for integrals of the form \eqref{eq:intro_integrals} which are suited to our task.

The Stochastic Sewing Lemma (SSL) was first introduced by one of the authors in \cite{le2020stochastic} as a generalization of the classical sewing lemma by Gubinelli and Feyel--La Pradelle \cite{Gubinelli2004,FeyLaP2006}.
Despite its recent introduction, it has already found numerous applications in regularization by noise \cite{le2020stochastic,butkovsky2023stochastic,ButGal2023}, numerical schemes for SDEs \cite{BuDaGe2021,le2021taming,DaGeLe2023}, singular SPDEs \cite{ABLM,BuDaGe2023}, convergence of slow-fast systems \cite{hairer2019averaging,LiPaSi2023} and the development of a hybrid theory of rough SDEs \cite{friz2021rough}.

Armed with suitable versions of SSL (recalled in Section~\ref{subsec:SSL}), we will define and provide stability estimates for general integral processes of the form \eqref{eq:intro_integrals} in Section~\ref{subsec:integral_estimates}.

After that, our strategy will consist of the following steps:
\begin{enumerate}
	\item[i)] We first work with smooth and bounded approximations $\{b^n\}_{n}$ of a given irregular drift $b$.
	For each $b^n$, classical results ensure strong existence and pathwise uniqueness for~\eqref{eq:intro_IPS} and~\eqref{eq:intro_MKV}, as well as propagation of chaos and mean field convergence. Such results are recalled for convenience in Appendix~\ref{app:lipschitz_p_system}.
	\item[ii)] Performing the analysis at the level of smooth interactions $b$, we obtain sufficiently strong a priori and stability estimates for solutions to \eqref{eq:intro_IPS}, which crucially depend only on $\|b\|_{L^q_T B^\alpha_{\infty,\infty}}$ and are uniform in $N$. This is accomplished in Section~\ref{subsec:pairwise_a_priori}, where we obtain a far reaching stability estimate in Proposition~\ref{prop:generic_pairwise_stabillity}; this will be our substitute for Step~\eqref{it:sznit_1} in Sznitman's argument.
	\item[iii)] Since in the regular case we know that propagation of chaos holds, we can take the limit $N\to\infty$ in the particle systems, giving the same estimates at the level of McKean--Vlasov equations \eqref{eq:intro_MKV}.
	See Section~\ref{sub:mkv}.
	\item[iv)] As the estimates only depend on $\|b\|_{L^q_T B^\alpha_{\infty,\infty}}$, we can eventually pass to the limit along approximation sequences $\{b^n\}_n$ to treat the case of truly singular drifts. This is first shown for particle systems, and then in a similar vein for McKean--Vlasov equations. In particular, we not only obtain well-posedness of both equations, but also show that they enjoy the same stability estimates as their regular counterparts.
	See Sections~\ref{sec:sing_drifts_sols} and~\ref{sec:reg_drifts_sols}. 
	\item[v)] Finally, we piece everything together in Section~\ref{sec:poc_proofs}, where we prove mean field convergence and propagation of chaos with quantitative rates for singular interactions.
	Here we employ a concentration inequality in martingale type $2$ spaces to readapt Step~\eqref{it:sznit_2} from Sznitman's argument; see Theorem~\ref{th:pairwise_mean_field}.
\end{enumerate}

Let us finally mention a small technical detail, related to step iv) above.
As the spaces $B^\alpha_{\infty,\infty}$ do not coincide with the closure of smooth functions under the $\|\,\cdot\,\|_{B^\alpha_{\infty,\infty}}$-norm, it will be convenient throughout the paper to work with another scale of closely related spaces, which we denote by $\cC^\alpha_x$ and are rigorously defined by eq.~\eqref{eq:defn_cC_alpha} in Section~\ref{sec:notation}. This comes without loss of generality, as will be shown in the proof of Theorem~\ref{thm:main} at the very end of Section \ref{sec:main_proofs}.
\subsection{Context and Wider Literature}\label{sec:literature}

The mathematical study of interacting particle systems dates back to Kac~\cite{kac_56_foundations}, in relation to mathematical foundations of kinetic theory, formalising earlier work by Boltzmann.
As mentioned, relevant examples in the physics literature already appeared in earlier works by Jeans~\cite{jeans_15_star}, Vlasov~\cite{vlasov_38} and Onsager~\cite{onsager_49_stat}.
We refer to \cite{sznitman1991topics,jabin2017mean} for a more detailed discussion of this history. 

First rigorous results addressing Kac's program are usually attributed to McKean \cite{mckean1966class}.
For Lipschitz interactions, similar results were then achieved with and without (Brownian) noise in a number of works, see \cite{braun_hepp_77_vlasov,dobrushin1979vlasov,tanaka1984Limits}.
In this regular regime, G\"artner's results \cite{gartner1988mckean} remain among the most general, based on monotonicity and Lyapunov-type conditions; among recent extensions in this direction, see \cite{hammersley2018mckean}.
As already mentioned, \cite{coghi2020pathwise} revisited the argument from \cite{tanaka1984Limits} to show that in the Lipschitz case, the connection between \eqref{eq:intro_IPS} and \eqref{eq:intro_MKV} holds rigorously for general, continuous or càdlàg noises $W^i$.
This approach was built upon by \cite{GaHaMa2022}, obtaining mean field convergence under a number of different conditions, in the style of \cite{gartner1988mckean}.

Subsequently, more challenging systems with singular interactions $b$ have started receiving more attention; among relevant precursors let us mention \cite{osada_85_stoch,osada_87_propagation} for stochastic $2$D point vortices (in relation to $2$D Navier--Stokes) and \cite{sznitman1991topics} for Dirac delta interactions (in relation to $1$D Burgers).
See \cite{hauray2007n,fournier2014propagation,fournier_jourdain_17_stochastic} for a small selection of later works and \cite{golse_16_dynamics} for a survey of the extensive literature available up to 2016.

In the last ten years, the field has witnessed tremendous progress, thanks to many newly introduced PDE-based arguments.
Among them, let us mention:
i) the modulated free energy method by Serfaty \cite{serfaty2017mean,serfaty2020mean}, covering deterministic particles interacting through singular, repulsive or conservative kernels of Coulomb and singular Riesz-type \cite{serfaty2020mean,nguyen_rosenzweig_serfaty_22_mean}, then extended to $2$D point vortices in \cite{Rosenzweig2022};
ii) the relative entropy methods by Jabin--Wang \cite{Jabin2018} allowing for Brownian particles with Biot--Savart or bounded interactions;
iii) the works \cite{bresch_jabin_wang_19_application,bresch_jabin_wang_23_mean} which combined the two approaches, treating both deterministic and Brownian particles and a wide range of singular interactions.
Such results have subsequently been upgraded to allow for uniform-in-time propagation of chaos estimates, see \cite{Guillin2022,Guillin2024,decourcel_rosenzweig_serfaty_23_attractive}; 
other recent readaptations include \cite{nikolaev2024quantitative,chen2023well}.

Another more stochastic approach to Brownian particles, based on the BBGKY hierarchy along with entropy estimates, was recently proposed by Lacker \cite{Lacker2023}.
As mentioned, for H\"older continuous coefficients, the author therein is able to show the improved total variation rate \eqref{eq:intro_lacker}, disproving the widely held belief that the previously obtained rate $N^{-1/2}$ was optimal.
For yet another approach, also leveraging a hierarchical description, see \cite{bresch_duerinckxx_jabin_24_duality}.

In a different direction, several works in the stochastic analysis community have devoted their attention to the McKean--Vlasov equation \eqref{eq:intro_MKV}, in relation to regularisation by noise.
This terminology refers to the frequently observed phenomenon where SDEs driven by Brownian motion are strongly well-posed even when the same does not holds for their deterministic counterparts, due to poor regularity of the drift.
Early contributions in the field are due to Zvonkin \cite{Zvonkin1974} and Veretennikov \cite{Veretennikov1981}; we refer to \cite{Flandoli2011} for a survey.
Leveraging on stochastic tools like It\^o calculus, Girsanov transform and Malliavin calculus, as well as the connection between \eqref{eq:intro_MKV} and \eqref{eq:intro_FP}, similar well-posedness results for McKean--Vlasov equations driven by non-degenerate Brownian noise have been established e.g. in \cite{rockner2018well,bauer_meyerBrandis_proske_18_strong,huang2019distribution,mishura_veretennikov_20_existence}. 
Only fewer works have been able to use the same techniques to achieve propagation of chaos results for general classes of interactions $b$; among them, let us single out the particularly noteworthy contributions \cite{Lacker2018,hao2022strong}, as well as the underlying large deviations derived in \cite{hoeksema2020large}. For specific choices of $b$, several works exploited It\^o calculus and Girsanov's theorem to obtain improved results, see \cite{fournier_jourdain_17_stochastic,tomasevic2020propagation,fournier_tomasevic_23_particle,tardy_23_weak}. 
Analysis of this type has also led to detailed results on precise behaviour of the interacting particles themselves, for instance categorising the probability of collision and collapse in the case of singular, attractive kernels, \cite{fournier_tardy_23_collisions,kinzebulatov_24_critical}.

Of course, the separation of recent literature presented above, into a more PDE-based community on one side, focusing predominantly on the nonlinear Fokker--Planck equation \eqref{eq:intro_FP}, the Liouville equation and BBGKY hierarchies, and another more stochastic analytic community on the other, focusing on the non-linear SDE \eqref{eq:intro_MKV} and the regularizing properties of Brownian motion, is partially artificial.
There are several works not fitting this classification; let us mention a few.
Rather than a collection of idiosyncratic noises $\{W^i\}_i$, \cite{CogFla2016} considered particles undergoing a Lipschitz interaction, subject to a common, environmental noise (which then converts \eqref{eq:intro_FP} into an SPDE with transport noise); for singular interactions, the same type of model was covered in \cite{nguyen_rosenzweig_serfaty_22_mean}.
In the context of rough path theory, the works \cite{CasLyo14,BaCaDe2020,bailleul2021propagation} obtained well-posedness of particle systems and McKean--Vlasov equations, as well as propagation of chaos, in the case of smooth interactions and so called second-order rough drivers.
These results have recently been extended to general rough paths, see \cite{delarue_salkeld_23_example} and the references therein.

Since the work of Davie \cite{davie2007uniqueness}, it has become increasingly clear that regularization by noise phenomena are related to the pathwise properties of the process $W$, and can likely be extended beyond the Brownian or Markovian cases.
The focus on fractional Brownian motion, besides being a natural test field for the mathematical theory, is motivated by its ubiquitous appearance when modelling signals with correlated increments; applications include statistical models of turbulence, hydrology, anomalous polymer dynamics, tele-communication networks and stochastic volatility, see e.g. the discussion in \cite{galeati2022subcritical} and the references therein.
Since the standard stochastic analysis and PDE tools are not available for $H\neq 1/2$, one must develop new methods to capture and exploit the regularising effect of fBm.

Early works in this direction are due to \cite{nualart2002regularization}, based on Malliavin calculus techniques and Girsanov transform.
A robust pathwise approach to regularisation of SDEs by fBm was then developed in \cite{Catellier2016}, based on the formalism of nonlinear Young integrals.
Following on \cite{Catellier2016}, there has been a significant growth of research in this area; for a small selection, we refer to \cite{galeati_harang_22_multiplicative,dareiotis_gerenscer_22_multiplicative,le_22_quantitative,butkovsky2023stochastic}.
As mentioned in Section \ref{subsec:ideas_proof}, a major boost in the field was given by the introduction of stochastic sewing in \cite{le2020stochastic,le2023banach}.

We are only aware of few other results in the literature concerning the well-posedness of singular McKean--Vlasov equation driven by fBm.
The approach from \cite{bauer_meyerBrandis_proske_18_strong} has been readapted in \cite{bauer2019mckean}, to establish  weak existence and uniqueness for \eqref{eq:intro_MKV} possibly on infinite-dimensional Hilbert spaces, for suitable H\"older regular coefficients.
Our previous work \cite{GaHaMa2023} obtained strong well-posedness and stability estimates for \eqref{eq:intro_MKV} for general, measure-dependent non-linearities, roughly speaking extending the results of \cite{Catellier2016} to the McKean--Vlasov case.
These results were further extended in \cite{han2022solving}, based on the entropy methods from \cite{Lacker2018}, and in \cite[Sec.~7]{galeati2022subcritical}, which refined the stability estimates from \cite{GaHaMa2023} to improve the contraction argument performed therein.
To the best of our knowledge, the only other work addressing the question of propagation of chaos in this setting is \cite{han2022entropic}, which readapts the hierarchical approach from \cite{Lacker2023}, to obtain again a convergence rate of the form \eqref{eq:intro_lacker}.
In this case however the result only holds on a sufficiently small time interval $[0,T^\ast]$, which depends on the data of the problem.

Since the first instance of this manuscript, \cite{amorino_nourdin_shevchenko_25_fractional} studied the problem of parameter estimation in interacting particle systems driven by fractional Brownian motion. This was motivated in part by their relevance in modelling turbulence in geophysical fluid flows. Notably, without access to PDE representation results, their approach relies on a propagation of chaos result for the Malliavian derivative of the interacting particle system. Whilst they only treat the case of regular interactions, many physically relevant models involve singular interaction kernels and a potential bridge between this work and  \cite{amorino_nourdin_shevchenko_25_fractional}  may allow for the treatment of this practically relevant class.

\subsection{Further Directions and Open Problems}\label{subsec:open_problems}

We collect here some interesting questions and potential future research topics stemming from our main result, Theorem~\ref{thm:main}, and the techniques employed in its proof.

As mentioned in Remark~\ref{rem:rate}, for regular interactions it is well-known (cf. \cite{tanaka1984Limits,FerMel1997,coghi2020pathwise}) that the random measures $$\eta^N\coloneqq  \sqrt{N}(\mu^N-\bar\mu)$$ admit a non-trivial, Gaussian limit $\eta$ as $N\to\infty$, describing the fluctuations of $\mu^N$ around $\bar{\mu}$. It is thus natural to wonder whether the same result can be obtained in our case. For Brownian noise, Gaussian fluctuations of $\mu^N$ around $\bar{\mu}$ were recently extended to a suitable class of singular interaction kernels (most notably those covered in \cite{Jabin2018}) in \cite{WaZhZH2023}; the proof therein heavily exploits the fact that $\eta$ is the solution to a linear SPDE, thanks to It\^o calculus.
For general fBm, a similar description for candidate limit points of $\{\eta^N\}_N$ is not available; the alternative one provided by \cite[Cor.~43]{coghi2020pathwise} for regular $b$ is quite abstract and seems harder to readapt to the singular case.
Nevertheless, the estimates from Corollary~\ref{cor:mean_field_negative_sobolev} already imply tightness of the sequence $\{\eta^N\}_N$ in suitable negative Sobolev spaces; we leave the problem of characterizing its limit for future investigations.

A closely related question concerns large deviations associated to the convergence $X^{1,N}\to \bar{X}^1$ or $\mu^N\to\bar\mu$; for regular drifts and general driving noise $W$, both are classical \cite{tanaka1984Limits,coghi2020pathwise}.
We currently lack the right tools to address this problem in our setting; let us mention however that, depending on the value of $H$ and the (ir)regularity of $b$, Girsanov transform can be available (see e.g. \cite[App.~C]{galeati2022subcritical} for a deeper discussion), which in turn may allow one to develop entropy and concentration estimates in the style of \cite[Thm.~2.6]{Lacker2018}, or to perform a Gibbs-type analysis directly yielding an LDP, as done for the Brownian case in \cite{hoeksema2020large}. 

In a different direction, it would be interesting to understand if our results, which currently cover arbitrarily large but finite time intervals $[0,T]$, may be improved to uniform estimates on $[0,+\infty)$.
Typically, this requires one to first understand the long-time behaviour of the McKean--Vlasov equation \eqref{eq:intro_MKV} and in particular to show convergence of $\bar\mu_t$ to a stationary equilibrium measure $\bar\mu$, as $t\to\infty$.
In the non-Markovian context of fBm driven SDEs, establishing this fact already requires major efforts, see \cite{Hairer2005,DePaTi2019,LiPaSi2023} and the references therein; these works only treat standard SDEs, and we are not aware of any result concerning singular McKean--Vlasov SDEs instead.

As mentioned in Remark~\ref{rem:meaning_assumption_intro}, while the regularity constraint $\alpha>1-1/(Hq')$ from \eqref{eq:intro_assumption} is natural from a scaling point of view, the restriction $q\in (1,2]$ is not.
Even for standard SDEs, a major question is to go beyond the threshold $\alpha>1-1/(2H)$ first obtained in \cite{Catellier2016}; for autonomous kernels $b$ (amounting to $q=\infty$), scaling suggests $\alpha>1-1/H$ to be the correct subcritical condition, under which one may expect to develop a robust solution theory for \eqref{eq:intro_IPS}-\eqref{eq:intro_MKV}.
This problem is open even in the Brownian case ($H=1/2$), where the condition becomes $b\in B^\alpha_{\infty,\infty}$ with $\alpha>-1$; see however the recent advances from \cite{hao2023sdes}, under additional regularity requirements on $\div b$ and possibly $\cL(X_0)$. For $\alpha<0$, improved results may be expected upon replacing $B^\alpha_{\infty,\infty}(\mbR^{d};\mbR^d)$ with better spaces of distributions (or even functions) with the same scaling behaviour, such as $B^{\alpha+d/p}_{p,\infty}(\RR^d;\RR^d)$ or $L^{-d/\alpha}(\RR^d;\RR^d)$. In this direction, let us mention the works \cite{ABLM,AnRiTa2023}, where well-posedness for SDEs is established for roughly $\alpha\geq 1-1/(2H)$; under their hypothesis, it seems reasonable to expect well-posedness of \eqref{eq:intro_IPS}-\eqref{eq:intro_MKV} as well, as well a qualitative propagation of chaos statement.
However, achieving a quantitative rate may prove challenging, depending on whether one can establish therein sufficiently strong stability estimates as in our case (cf. Proposition~\ref{prop:generic_pairwise_stabillity}).

Concerning existence results, we expect the a priori estimates and compactness arguments developed in \cite{butkovsky2023stochastic,ButGal2023} to carry over to our setting as well, thus providing weak solutions to \eqref{eq:intro_IPS}-\eqref{eq:intro_MKV} for drifts $b\in L^q([0,T];L^p(\RR^d;\RR^d))$ with
\begin{align*}
	\frac{1-H}{q}+\frac{H d}{p}<1-H, \quad H\in (0,1).
\end{align*}
Note however that weak solutions are not enough to implement Sznitman's comparison approach, which is crucial for us in obtaining estimates of the form \eqref{eq:intro_poc}-\eqref{eq:intro_mfc}.
Concerning the structure of noise, in this paper we only considered the case of additive $W^{i,H}$ in \eqref{eq:intro_IPS}; based on the recent progress made in \cite{dareiotis_gerenscer_22_multiplicative,catellier2022regularization}, it would be interesting to cover non-degenerate multiplicative fBm noise as well, e.g. systems of the form
\begin{equation*}
	\dd X^{i;N}_t = \frac{1}{N-1}\sum_{\substack{j=1,\, j\neq i}}^N b_t(X^{i;N}_t,X^{j;N}_t) \dd t + \sigma(X^{i;N}_t) \dd W^{i,H}_t,\qquad \text{for all } i=1,\ldots,N,
\end{equation*}
similarly for the McKean--Vlasov counterpart of \eqref{eq:intro_MKV}; we leave all these questions for future research.

Finally, let us mention other mean field models, related to \eqref{eq:intro_IPS}, which we expect to be amenable to our techniques.
First of all, although in this paper we focus solely on first order systems \eqref{eq:intro_IPS} (namely, with particles only described by their position $X^i$), we expect a similar analysis to hold for second order, kinetic ones (with particles $Z^i=(X^i,V^i)$ described by position and velocity).
To explain what we mean, without making the notation too burdensome, let us focus on the toy model of an SDE
\begin{equation}\label{eq:intro_toy_kinetic}
	\begin{cases}
		\dd X_t = V_t,\\
		\dd V_t = b_t(X_t) \dd t + \dd W^H_t.
	\end{cases}
\end{equation}
The second order system \eqref{eq:intro_toy_kinetic} can be rewritten in an integral formulation, where only the $X$ variable appears, as
\begin{equation}\label{eq:intro_toy_kinetic_2}
	X_t = X_0 + t V_0 +\int_0^t \int_0^s b_r(X_r) \dd r \dd s + \int_0^t W^H_s \dd s
\end{equation}
where we see that the noise driving the system is now of the form $W^{H+1}_t$. Note in particular that, even if we start with $H\in (0,1)$, the reformulation \eqref{eq:intro_toy_kinetic_2} naturally yields an fBm of parameter $H>1$, providing further reason to study SDEs with general values $H\in (0,+\infty)\setminus \NN$.
Due to the reformulation \eqref{eq:intro_toy_kinetic_2}, we expect our techniques to readapt in the case of similarly defined interacting particle systems and McKean--Vlasov equations, yielding well-posedness and stability estimates, as well as mean field and propagation of chaos estimates, as soon as
\begin{align*}
	b\in L^q([0,T];B^\alpha_{\infty,\infty}(\RR^{2d};\RR^d)), \quad \alpha>1-\frac{1}{(1+H)q'}, \quad q\in (1,2];
\end{align*}
this is nothing else than condition \eqref{eq:intro_assumption} with $H$ replaced by $H+1$.
We do not expect however this result to be optimal: due to the double integral in time appearing in \eqref{eq:intro_toy_kinetic_2}, this system should be better behaved than the corresponding first-order one driven by $W^{H+1}_t$. In any case, in the Brownian case $H=1/2$, this regularity counting suggests that our methods allow to deduce propagation of chaos for kinetic equations with $b\in B^\alpha_{\infty,\infty}$ with $\alpha>2/3$, which is on par with the results from \cite{holding_16_propagation}.

Besides second-order systems, one may consider other variants of first-order ones, by allowing different interactions $b$.
For instance, although in the regime $\alpha>0$ we always worked with globally bounded drifts (due to the definition of the inhomogeneous Besov-H\"older spaces $B^\alpha_{\infty,\infty}$), this assumption is not crucial, and it could likely be replaced by a linear growth condition, combined with a bound on the H\"older seminorm of $b$.
Similarly, we restricted ourselves to the study of pairwise interactions, but the same arguments can be applied to $k$-wise interactions of the form
\begin{align*}
	\frac{1}{(N-1)\cdot\ldots\cdot (N-k+1)} \sum_{i\neq j_1\neq \ldots\neq j_{k-1}} b(X^{i;N}_t,X^{j_1;N}_t,\ldots,X^{j_{k-1};N}_t)
\end{align*}
with $b\in L^q([0,T];B^\alpha_{\infty,\infty}(\RR^{kd};\RR^d)$ still satisfying \eqref{eq:intro_assumption}.
Finally, we expect it to be possible to cover functionals $B(x,\mu)$ with a more general dependence on $\mu$, e.g. through some of its statistics, namely $B(x,\mu)=F(x,\langle \phi,\mu\rangle)$; for some well-posedness results of \eqref{eq:intro_MKV} in this case, we refer to \cite{GaHaMa2023,galeati2022subcritical}.

Last but not least, let us point out that although we take i.i.d. initial conditions, leading to exchangeability of the system \eqref{eq:intro_IPS}, this property does not play a key role in our computations; for instance, it is not used in the derivation of the key Proposition \ref{prop:generic_pairwise_stabillity}.
For this reason, it would be interesting to extend these results to the case of non-exchangeable systems, for instance those of the form
\begin{equation}\label{eq:intro_IPS_weights}
	X^{i;N}_t = X_0^i + \frac{1}{N-1}\sum_{\substack{j=1\\ j\neq i}}^N w^N_j \int_0^t b(X^{i;N}_s,X^{j;N}_s) \dd s + W^{i,H}_t
\end{equation}
where $\{w^N_j\}$ are deterministic weights satisfying suitable assumptions; \eqref{eq:intro_IPS} corresponds to the case $w^N_j\equiv 1$.
For regular interactions $b$, system \eqref{eq:intro_IPS_weights} has been studied in \cite{Crevat2019} in the absence of noise, recovering the classical Dobrushin's estimate, and for general additive noise in \cite[Sec.~3.4]{coghi2020pathwise}.
In the Brownian case, singular interactions have been recently treated in \cite{wang2022mean}.
A major point in all of these works is to regard the weights $\{w^N_j\}_{j,N}$ as part of the dynamic itself and therefore consider an extended phase space; interestingly, this allows to obtain mean field limits for a large class of weighted empirical measures of the form
\begin{equation}\label{eq:intro_weighted_empirical}
	\tilde{\mu}^N_t \coloneqq  \frac{1}{N} \sum_{i=1}^N \tilde w^N_i \delta_{X^{i;N}_t}
\end{equation}
where $X^{i;N}$ evolve according to \eqref{eq:intro_IPS_weights}, but $\{\tilde w^N_j\}_{j,N}$ are different from $\{w^N_j\}_{j,N}$. We leave the analysis of \eqref{eq:intro_IPS_weights} and the mean field limit of \eqref{eq:intro_weighted_empirical} for future works.

\subsection{Structure of the Paper}

Section~\ref{sec:notation} gathers notations and conventions that are consistently used in the paper as well as some basic setup.
All necessary preliminaries are recalled in Section~\ref{sec:preliminaries}; these include properties of fBm, in particular its local nondeterminism, 
spaces of $\kappa$-variation, controls and stochastic sewing techniques.
They are combined together in Section~\ref{subsec:integral_estimates}, which provides key estimates on integrals of the form $\int_0^{\,\cdot\,}b_s(Y_s)\dd s$, for paths $Y$ which decompose into an fBm and slowly varying remainder.
Section~\ref{sec:reg.drift} contains the main analytic work of the paper, obtaining necessary a priori estimates and convergence rates for \eqref{eq:intro_IPS}-\eqref{eq:intro_MKV} in the case of smooth interactions.
All our main results, ranging from well-posedness of \eqref{eq:intro_IPS}-\eqref{eq:intro_MKV} to mean field convergence, are proved in Section~\ref{sec:main_proofs}, along with some important corollaries.
Finally, we gather some useful ancillary results in the appendices. Appendix~\ref{app:lipschitz_p_system} recaps the classical propagation of chaos argument for regular interactions;
Appendices~\ref{app:useful} and \ref{app:holder_besov} contain useful results which 
included for the sake of completeness.

\subsection{Notation, Conventions and Setup}\label{sec:notation}
\begin{itemize}[leftmargin=6mm]
	\item From now on we always work on a fixed time interval $[0,T]$, where $T\in (0,+\infty)$ can be taken arbitrarily large but finite.
	\item Whenever the specification of the domain and range are not important or clear from the context, we denote by $C^0_x$ the Banach space $C_b(\RR^\ell;\RR^j)$ of continuous, bounded functions $f:\RR^\ell\to\RR^j$, endowed with the supremum norm $\| f\|_{C^0_x}=\sup_{x\in\RR^\ell} |f(x)|$. Similarly, for $k\geq 1$, we denote by $C^k_x$ the Banach space of continuous functions $f$, with continuous and bounded derivatives up to order $k$, with norm
	\begin{equation*}
		\|f\|_{C^k_x} \coloneqq \max_{l\leq k}\|D^l_x f\|_{C^0_x}.
	\end{equation*}
	\item We denote by $C^\infty_b=C^\infty_b(\RR^\ell;\RR^j)$ the space of infinitely differentiable functions $f:\RR^\ell\to\RR^j$, with all bounded derivatives; similarly for $C^\infty_b([0,T]\times \RR^\ell;\RR^j)$.
	\item We denote by $BUC_x=BUC(\RR^\ell;\RR^j)$ the space of uniformly continuous, bounded functions $f:\RR^\ell\to\RR^j$; arguing as in \cite[Thm.~6.8]{Heinonen}, it is easy to check that $BUC_x$ corresponds to the closure of $C^\infty_b$ with respect to the supremum norm $\| \cdot\|_{C^0_x}$. By definition, $BUC_x\hookrightarrow C^0_x$.
	\item For $\alpha\in\RR$, we denote by $B^\alpha_{\infty,\infty}=B^\alpha_{\infty,\infty}(\RR^\ell;\RR^k)$ the inhomogeneous Besov--H\"older space, as defined for instance in \cite[Def.~2.6.7]{BahCheDan}. For $\alpha\in (0,+\infty)\setminus\NN$, $B^\alpha_{\infty,\infty}$ corresponds to the more classical H\"older space $C^\alpha_x=C^\alpha_b(\RR^\ell;\RR^j)$ of bounded functions $f$ whose derivatives of order $|l|\leq \lfloor\alpha\rfloor$ are bounded and $\{\alpha\}$-H\"older continuous; here $\lfloor\cdot\rfloor$ and $\{\cdot\}$ respectively stand for integer and fractional parts.
	However, for $\alpha=k\in\NN_{\geq 0}$, only the strict embedding $C^k_x \subsetneq B^k_{\infty,\infty}$ holds (cf. \cite[p.~99]{BahCheDan}). 
	\item For $\alpha\in (-\infty,1)$, we define the Banach space $\cC^\alpha_x=\cC^\alpha(\RR^\ell;\RR^j)$ as
	\begin{equation}\label{eq:defn_cC_alpha}
		\cC^\alpha_x =	
		\begin{cases}
			BUC_x \quad & \text{if } \alpha=0,\\
			\text{the closure of $C^\infty_b$ under the norm } \|\cdot\|_{B^\alpha_{\infty,\infty}} \ \ & \text{otherwise};
		\end{cases}
	\end{equation}
	as a consequence of the definition and Besov embeddings, for any $\alpha\in\RR$ and $\delta>0$ one has 
	\begin{equation}\label{eq:holder_complete_embed}
		\cC^\alpha_x\hookrightarrow B^\alpha_{\infty,\infty} \hookrightarrow \cC^{\alpha-\delta}_x.
	\end{equation}
	\item We let $L^\infty_x = L^\infty(\mbR^\ell;\mbR^j)$ denote the space of essentially bounded functions $f:\mbR^\ell \to \mbR^j$, equipped with the usual norm $\|f\|_{L^\infty_x} \coloneqq \esssup_{x\in \mbR^{\ell}}|f(x)|$. $W^{1,\infty}_x = W^{1,\infty}(\mbR^\ell;\mbR^j)$ denotes the Sobolev space of essentially bounded functions with essentially bounded first, weak derivatives, equipped with the norm
	\begin{equation*}
		\|f\|_{W^{1,\infty}_x} = \|f\|_{L^\infty_x} + \|\nabla f\|_{L^\infty_x} \eqqcolon \|f\|_{L^\infty_x} + \|f\|_{\dot W^{1,\infty}_x}.
	\end{equation*}
	Recall that $W^{1,\infty}_x$ is equivalent to the space of bounded, Lipschitz continuous functions, and the optimal Lipscihtz constant of $f$ is given by $\|f\|_{\dot W^{1,\infty}_x}$.
	\item In all of the above points, the subscript $x$ is used in order to stress that regularity/integrability is measured in the space variable, as opposed to being measured in the time variable (subscript $T$) or the underlying probability space (subscript $\omega$), see below. In particular, even when considering domains $\R^\ell=\R^d\times\R^m$ with generic elements $z=(x,y)$ for $x\in\R^d$, $y\in \R^m$, in order not to burden the notation we will keep writing $f\in \cC^\alpha_x$ instead of $\cC^\alpha_z$ or $\cC^\alpha_{x,y}$.
	\item Given a Banach space $E$ and $q\in [1,\infty)$, we use the shorthand notation $L^q_T E=L^q([0,T];E)$ to denote the Bochner--Lebesgue space comprising all strongly measurable functions $f:[0,T]\to E$ such that
	\begin{equation*}
		\|f\|_{L^q_T E}\coloneqq \bigg(\int_0^T \|f_t\|^q_{E}\dd t\bigg)^{\frac{1}{q}}<\infty.
	\end{equation*}
	Similarly, we define $L^\infty_T E$ by taking the essential supremum in $t\in [0,T]$.
	We use $C_T E$ to denote the space of all continuous maps $f:[0,T]\to E$, equipped with the usual supremum norm. When $E = \mbR^\ell$ and there is no cause of confusion, we suppress it and simply write $L^q_T,\, C_T$ in place of $L^q_T  \mbR^\ell,\, C_T  \mbR^\ell$.
	\item A case of particular importance for us will be $E=\cC^\alpha_x$ as defined above.	
	The motivation to use $\cC^\alpha_x$ is that by construction, $C^\infty_b$ is dense in $\cC^\alpha_x$, for any $\alpha\in (-\infty,1)$; as a consequence, for any $q\in [1,\infty)$, $\alpha\in (-\infty,1)$ and $f\in L^q_T \cC^\alpha_x \coloneqq  L^q([0,T];\cC^\alpha_x)$, there exists a sequence $\{f^n\}_n\subset C^\infty_b([0,T]\times \RR^\ell;\RR^j)$ such that $f^n\to f$ in $L^q_T \cC^\alpha_x$.
	In fact, throughout the majority of the paper (especially Sections~\ref{sec:reg.drift} and~\ref{sec:main_proofs}) we will work with $L^q_T \cC^\alpha_x$ rather than $L^q_T B^\alpha_{\infty,\infty}$.
	This comes essentially without loss of generality, thanks to the embedding \eqref{eq:holder_complete_embed}, as will be shown at the end of the paper in the proof of Theorem~\ref{thm:main} in Section~\ref{sec:poc_proofs}.
	\item Given $\beta \in (0,1)$, $p\in [1,\infty)$ and $f:[0,T]\to \mbR^{\ell}$, we define the Gagliardo seminorm
	\begin{equation}\label{eq:gagliardo_seminorm}
		\bra{f}_{W^{\beta,p}_{T}}\coloneqq  \left( \iint_{[0,T]^2} \frac{\|f_t -f_s\|_E^p}{|t-s|^{1+\beta p}} \dd t\dd s\right)^{\frac{1}{p}}.
	\end{equation}
	The fractional Sobolev space $W^{\beta,p}([0,T];\RR^\ell)=W^{\beta,p}_T$ is defined as the collection of all $f\in L^p_T$ such that $\bra{f}_{W^{\beta,p}_{T}}<\infty$, with norm $\| f\|_{W^{\beta,q}_{T}} \coloneqq  \| f\|_{L^p_T} + \bra{f}_{W^{\beta,p}_{T}}$. For a comprehensive overview on such spaces, we refer to \cite{DiPaVa2012}.
	\item 
	Given $q\in [1,\infty]$, we denote by $q' \in [1,\infty]$ its H\"older conjugate, namely such that
	\begin{equation*}
		\frac{1}{q}+\frac{1}{q'}=1,
	\end{equation*}
	with the convention that $q'=\infty$ if $q=1$ and vice versa.
	\item Given a symmetric nonnegative matrix $\Sigma\in \RR^{\ell\times\ell}$, $G_\Sigma$ denotes the convolution with the Gaussian density $g_\Sigma$ associated to $\mathcal{N}(0,\Sigma)$.
	For simplicity, we will write $G_t$ in place of $G_{t I_d}$ to denote the standard heat kernel semigroup on $\RR^\ell$. Namely,
	\begin{align*}
		G_t f = g_t\ast f, \quad \text{where} \quad g_t(x)=(2\pi t)^{-\frac{\ell}{2}} e^{-\frac{|x|^2}{2t}}.
	\end{align*}
	\item For $0\leq \cS\leq \cT\leq T$, we define the two and three dimensional simplexes by
	\begin{equation}\label{eq:simplex}
		[\cS,\cT]^2_{\leq} := \{ (s,t) \in [\cS,\cT]^2 \,:\, s\leq t\},\quad
		[\cS,\cT]^3_{\leq} := \{ (s,u,t) \in [\cS,\cT]^3\,:\, s\leq u\leq t\}.
	\end{equation}
	In addition, we define the restricted simplexes by
	\begin{align}
		& \overline{[\cS,\cT]}^2_{\leq} := \{(s,t)\in [\cS,\cT]^2_{\leq} \,:\, s-(t-s) \geq \cS\},\label{eq:simplex_restricted_2} \\
		& \overline{[\cS,\cT]}^3_{\leq} := \left\{(s,u,t)\in [\cS,\cT]^3_{\leq} \,:\, (u-s)\wedge (t-u) \geq (t-s)/3,\,s-(t-s) \geq \cS  \right\}. \label{eq:simplex_restricted_3}
	\end{align}
	\item For $s,t \in [\cS,\cT]_\leq$ and $Y:[\cS,\cT]\to E$, we write $Y_{s,t}\coloneqq  Y_t-Y_s$.
	\item Given $A:[\cS,\cT]^2_{\leq}\to E$, we define a map $\delta A:[\cS,\cT]^3_{\leq}\to E$ by
	$$\delta A_{s,u,t} = A_{s,t}-A_{s,u}-A_{u,t}.$$
	\item Whenever working with a filtered probability space $(\Omega,\cF,\FF=\{\cF_t\}_{t\in [0,T]},\PP)$, we will tacitly assume that $\cF_0$ is complete and that $\FF$ is right-continuous. 
	\item For an $E$-valued random variable $Y$ on a probability space $(\Omega,\cF,\PP)$, we write $\|Y\|_{L^m_\omega E} := \EE[\|Y\|_E^m]^{1/m}$  for $m\in [1,\infty)$; we set  $L^m_\omega E\coloneqq L^m(\Omega,\cF,\PP;E)$. The definitions extend as usual to $m=\infty$. 
	Moreover, we denote the law of $Y$ on $E$ by $\mcL(Y)\coloneqq Y_{\#} \mbP = \mbP(Y^{-1})$.
	\item Given $Y\in L^1_\omega E$ and a filtration $\FF$, whenever clear we will use the shorthand $\EE_sY := \EE[Y|\cF_s]$ to denote its conditional expectation, for any $s\in [0,T]$.
	\item Given a measure $\mu \in \cM(\RR^d)$ and a function $b\in C^\infty_b(\RR^{2d};\RR^d)$, we adopt the notation
	\begin{equation}\label{eq:measure_functional}
		b(x,\mu):= \int_{\RR^d} b(x,y)\dd \mu(y).
	\end{equation}
	\item 
	We will often write $A \lesssim _{\alpha,\,q\,H,\,\ldots} B$ to indicate that the inequality holds up to a constant depending on the parameters $\alpha,\,q,\,H,\,\ldots$, namely that there exists a constant $C\coloneqq C(\alpha,q,H,\ldots)>0$ such that $A\leq CB$.
	If the constant depends on some norms (typically of the form $\|b\|_{L^q_T \mcC^{\alpha}_x}$), we always mean that the constant is non-decreasing in these quantities. We simply write $A\lesssim B$ if the constant depends on unimportant parameters, which we do not track.
	\item Given collections of random variables such as $\{X^{i;N}\}_{i=1}^N$, or $\{W^{i;N}\}_{i=1}^N$, we will often denote them for short by $X^{(N)}$, $W^{(N)}$.
\end{itemize}

\section{Preliminaries}\label{sec:preliminaries}

\subsection{Fractional Brownian motion, local nondeterminism and conditional norms}\label{subsec:fbm}

We recall here several facts concerning fractional Brownian motion (fBm); we refer to \cite{nualart2006malliavin,picard2011representation} for some standard references.

An $\RR^\ell$-valued fractional Brownian motion (fBm) $W^H$ of Hurst parameter $H\in (0,1)$ is defined as the unique centred Gaussian process with covariance
\begin{equation*}
	\EE(W^H_t\otimes W^H_s)=\tfrac{1}{2}\big(|t|^{2H}+|s|^{2H}-|t-s|^{2H}\big) I_\ell
\end{equation*}
where $I_\ell$ denotes the $\ell\times \ell$ identity matrix; in other words, its components are i.i.d. one dimensional fBms. FBm paths are known to be $\PP$-a.s. $(H-\eps)$-H\"older continuous for all $\eps>0$, but nowhere $H$-H\"older continuous, and $H$-self-similar, that is $\cL(\lambda^{-H} W^H_{\lambda \,\cdot\,})=\cL(W^H)$.

FBm admits a canonical Volterra representation: given an fBm $W^H$ defined on a probability space $(\Omega,\cF,\PP)$, one may construct a standard Brownian motion $B$, defined on the same space, such that 
\begin{equation}\label{eq:volterra}
	W^H_t=\int_0^t K_H(t,r)\dd B_r\quad\forall\, t \geq 0,
\end{equation}
for a suitable family of deterministic kernels $\{K_H(t,\cdot)\}_{t\in [0,T]} \subset L^2([0,T])$; we omit the exact formula for $K_H$, which can be found e.g. in \cite{NUALART2002}. It is known that $W^H$ and $B$ generate the same filtration.
Given a filtration $\FF$, we say that $W^H$ is a $\FF$-fBm if the associated $B$ given by \eqref{eq:volterra} is a $\FF$-Brownian motion.

As a consequence of the Volterra representation \eqref{eq:volterra}, if $W^H$ is an $\mbF$-fBm it holds that
\begin{equation*}
	 \sigma^H_{s,t}\coloneqq  {\rm Cov} \big(W^H_t|\cF_s)
	= {\rm Cov} \big(W^H_t - \EE_s W^H_t \big)
	= \int_s^t K^H(t,r)^2 \dd r.
\end{equation*}
It  follows that there exists a constant $c_H >0$ such that
\begin{equation}\label{eq:LND_fBm}
  c_H |t-s|^{2H}I_\ell \leq {\rm Cov} \big(W^H_t|\cF_s) \leq |t-s|^{2H}I_\ell
\end{equation}
To wit, the lower bound in \eqref{eq:LND_fBm} follows from the pointwise inequality
\begin{equation}\label{eq:volterra_kernel_lower_bound}
	K^H(t,s)\gtrsim_H (t-s)^{H-1/2},
\end{equation}
cf. \cite[Prop.~B.2-(ii)]{butkovsky2023stochastic}. The upper bound is a  consequence of the conditional expectation being a  projection in $L^2_\omega$; i.e. $\text{Cov}(W^H_t -\mbE_s W^H_t) \leq \mbE[|W^H_t - W^H_s|^2]$.

The lower bound of \eqref{eq:LND_fBm} is a stronger version of the property usually referred to as \emph{local nondeterminism} (LND). We refer to \cite{galeati2022subcritical} and the references therein for a deeper discussion on the connection between \eqref{eq:LND_fBm} and regularization by noise.

Following \cite{gerencser2023}, we inductively extend the definition of fBm to all values $H\in (0,+\infty)\setminus\NN$ by setting\footnote{FBm of parameter $H=1$ is somewhat trivial or ill-defined, see \cite{picard2011representation}; this is why we exclude all integer values of $H$.} $$W^{H+1}_t:=\int_0^t W^H_s\dd s.$$
This definition is consistent with most aforementioned properties: it is still a centred, Gaussian process, with trajectories  $\mbP$-a.s. in $C^{H-\eps}_t$ but nowhere $C^H_t$, satisfying the scaling relation $\cL( \lambda^{-H} W^H_{\lambda \cdot}) = \cL(W^H)$; using a stochastic version of Fubini's theorem, one can easily derive a similar representation as \eqref{eq:volterra} and prove \eqref{eq:volterra_kernel_lower_bound}.
Most importantly, the LND property \eqref{eq:LND_fBm} is also still valid, cf. \cite[Prop.~2.1]{gerencser2023} (up to defining the correct constant $c_H$ appearing in \eqref{eq:LND_fBm} also for $H>1$).

In order to exploit the regularizing effect of the noise, let us recall the following standard heat kernel estimates. For any $\beta\in \RR$, any $\ell\geq 1$ and any $k\in\NN_{\geq 0}$ such that $\beta\leq k$, there exists a constant $C\coloneqq C(\ell,\beta,k)>0$ such that for all $r>0$ and all smooth $f:\R^\ell\to \R$, it holds
\begin{equation}\label{eq:heat_kernel_estimates}
	\|D^k_x G_r f\|_{L^\infty_x} \leq C (1\wedge r)^{\frac{\beta-k}{2}} \|f\|_{\cC^\beta_x};
\end{equation}
see e.g. \cite[Lem.~A.10]{galeati2020noiseless} or \cite[Lem.~A.3]{ABLM} for the proof of similar statements.

In our setting, the relevant information in \eqref{eq:heat_kernel_estimates} is related to its explosive behaviour as $r \ll 1$, therefore with a small abuse of notation we will systematically write $r$ in place of $1\wedge r$.
Since conditional expectations are $L^2_\omega$-projections, the Gaussian random variables $W^H_t-\EE_s W^H_t$ and $\EE_s W^H_t$ are $L^2_\omega$-orthogonal, thus independent; more generally, $W^H_t-\EE_sW^H_t$ is independent of $\cF_s$.  Therefore for any $s\leq t$, any bounded measurable function $f:\RR^\ell \to\RR$, any other $\cF_s$-measurable random variable $Z$ and any $\beta<0$ it holds that
\begin{equation}\label{eq:LND+heat}
	\begin{aligned}
		\EE_s f(W^H_t+Z) &= G_{ \text{Cov}(W^H_t -\mbE_s W^H_t)} f(\EE_s W^H_t+Z) = G_{\sigma^H_{s,t}} f (\EE_s W^H_t + Z), \\
		\| D^k_x G_{\sigma^H_{s,t}} f\|_{L^\infty_x}
        & \lesssim (\sigma^H_{s,t})^{\frac{\beta-k}{2}} \| f\|_{\cC^\beta_x}
        \lesssim |t-s|^{(\beta-k) H} \| f\|_{\cC^\beta_x}.
	\end{aligned}
\end{equation}
where in the second line we used \eqref{eq:LND_fBm}, \eqref{eq:heat_kernel_estimates} and the aforementioned notational conventions.

Estimates in the style of \eqref{eq:LND+heat} will play a crucial role when applying sewing techniques, which will be presented in Section~\ref{subsec:SSL}; we will use them in combination with conditional norms, which we recall here briefly as well, referring to \cite{gerencser2023,friz2021rough} for more details.

Given a Banach space $E$, a filtered probability space $(\Omega,\cF,\FF,\PP)$ and $Y\in L^m_\omega E$, we define
\begin{align*}
	\|Y\|_{L^m|\cF_s} = \EE\left[\|Y\|_E^m \big|\cF_s\right]^{1/m}
\end{align*}
with the usual convention for $m=\infty$.
Given $Y,\,Z\in L^m_\omega E$ such that $Z$ is $\cF_s$-measurable, one has
\begin{equation}\label{eq:pseudo_projection}
	\|Y-\EE_s Y\|_{L^m|\cF_s} \leq 2 \|Y-Z\|_{L^m|\cF_s}.
\end{equation}
We often use moment and conditional moment norms in tandem; for $m$, $n\in [1,\infty]$, we write
\begin{equation*}
	\left\| \|Y\|_{L^m|\cF_s}\right\|_{L^n_\omega E} := \EE\left[ \EE\left[ \|Y\|_E^m\big| \cF_s\right]^{n/m}\right]^{1/n}.
\end{equation*}
For $m\leq n$, by Jensen's inequality for conditional expectations, we have
\begin{equation}\label{eq:moment_cond_moment_equiv}
	\|Y\|_{L^m_\omega E}\leq \left\|\|Y\|_{L^m|\cF_s}\right\|_{L^n_\omega E}\leq \|Y\|_{L^n_\omega E},
\end{equation}
with equality trivially attained if $m=n$. Furthermore, for $s \leq t$ it holds that
\begin{align*}
	\left\|\|Y\|_{L^m|\cF_s}\right\|_{L^n_\omega E} \leq \left\|\|Y\|_{L^m|\cF_t}\right\|_{L^n_\omega E}.
\end{align*}
Finally, for $m\leq n$ the mixed norms $\|\|\,\cdot\,\|_{L^m|\cF_s}\|_{L^n_\omega E}$ still satisfy natural analogues of classical inequalities like Jensen's, H\"older's and Minkowski's, as can be verified using properties of conditional expectation.

\subsection{Controls and spaces of bounded variation}\label{subsec:variation}

We start by recalling the concepts of controls and paths of finite $\kappa$-variation; see the monograph \cite{friz2010multidimensional} for a general overview.

\begin{defn}\label{defn:control}
	A function $w :[0,T]^2_\leq  \rightarrow \RR_+$ is a \emph{control} on $[0,T]$ if it is continuous, zero on the diagonal and superadditive, namely $w(s,s)=0$ and $w(s,u) + w(u,t)\leq w(s,t)$ whenever $s\leq u \leq t$.
\end{defn}

\begin{rem}\label{rem:properties_controls} The following properties of controls will be often used throughout this work.
	If $w_1$, $w_2$ are controls, then so are $a w_1 + b w_2$ for all $a,\, b\geq 0$ and $w_1 w_2$.
	More generally, if $w$ is a control and $\gamma$ is an increasing function, in the sense that $\gamma(s',t')\leq \gamma(s,t)$ whenever $[s',t']\subset [s,t]$, then $\gamma\, w$ is a control.
	It is slightly less obvious that, for any $a,\,b>0$, there exists another control $w_3$ such that
	\begin{equation*}
		w_1(s,t)^a w_2(s,t)^b = w_3(s,t)^{a+b} \quad \forall\, [s,t]\subset [0,T]
	\end{equation*}
	which is precisely given by $w_3=w_1^{a/a+b} w_2^{b/a+b}$; see \cite[Ex.~1.9]{friz2010multidimensional}. Moreover if $w$ is a control, then so is $w^p$ for any $p\geq 1$, cf. \cite[Ex.~1.8]{friz2010multidimensional}.
\end{rem}

\begin{defn}\label{defn:kappa_variation}
	Let $E$ be a Banach space, $\kappa\in [1,\infty)$; we say that $Y:[0,T]\to E$ is a continuous $E$-valued path of bounded $\kappa$-variation, $Y\in C^{\kappa-\var}_T E$ for short, if $Y\in C_T E$ and
	\begin{equation*}
		\llbracket Y\rrbracket_{C^{\kappa-\var}_{[s,t]} E} := \Bigg(\sup_{\pi\in \Pi([s,t])} \sum_{(t_i,t_{i+1})\in \pi} \| Y_{t_i,t_{i+1}}\|_E^\kappa \Bigg)^{\frac{1}{\kappa}} <\infty \quad \forall\, [s,t]\subset [0,T]
	\end{equation*}
	where $\Pi([s,t])$ denotes the set of all finite partitions of $[s,t]$.
	
	We denote by $\dot C^{\kappa-\var}_T E$ the space of paths $Y$ such that $Y-Y_0\in C^{\kappa-\var}_T E$, or equivalently such that $Y-Y_0\in C_T E$ and $\bra{Y}_{C^{\kappa-\var}_T}$ is finite. Observe that this does not require $Y_0$ to belong to $E$.
\end{defn}

Given $Y\in C^{\kappa-\var}_T E$, we set
\begin{equation*}
	\|Y\|_{C^{\kappa-\var}_{[s,t]} E} := \| Y_s\|_E + \llbracket Y\rrbracket_{C^{\kappa-\var}_{[s,t]}E}
\end{equation*}
and make use of the shorthands,
\begin{equation*}
	\llbracket Y\rrbracket_{C^{\kappa-\var}_T E} := \llbracket Y\rrbracket_{C^{\kappa-\var}_{[0,T]} E},
	\quad \| Y\|_{C^{\kappa-\var}_T E} := \|Y\|_{C^{\kappa-\var}_{[0,T]} E}.
\end{equation*}
It is easy to check that the space $C^{\kappa-\var}_T E$ is Banach with norm $ \|\cdot\|_{C^{\kappa-\var}_T E}$ and that
\begin{equation}\label{eq:inequalities_variation_control}
	\| Y_{s,t}\|_E \leq \llbracket Y\rrbracket_{C^{\kappa-\var}_{[s,t]} E}, \quad\sup_{r\in [s,t]} \| Y_r\|_E \leq \|Y\|_{C^{\kappa-\var}_{[s,t]} E}\quad \forall\, [s,t]\subset [0,T].
\end{equation}
Further observe that, if $\kappa<\tilde{\kappa}$, then $C^{\kappa-\var}_T E\hookrightarrow C^{\tilde\kappa-\var}_T E$ and
\begin{equation}\label{eq:inequality_variations}
	\llbracket Y\rrbracket_{C^{\tilde{\kappa}-\var}_{[s,t]}E} \leq \llbracket Y\rrbracket_{C^{\kappa-\var}_{[s,t]}E}, \quad
	\| Y\|_{C^{\tilde{\kappa}-\var}_{[s,t]}E} \leq \| Y\|_{C^{\kappa-\var}_{[s,t]}E} \quad
	\forall\, [s,t]\subset [0,T].
\end{equation}
\begin{rem}\label{rem:link_control_variation}
	It follows from \cite[Prop.~5.8]{friz2010multidimensional} that, if $Y\in \dot C^{\kappa-\var}_T E$, then the function
	\begin{equation}\label{eq:def_control_Y}
		(s,t)\mapsto w_Y(s,t):= \llbracket Y\rrbracket_{C^{\kappa-\var}_{[s,t]} E}^\kappa
	\end{equation}
	defines a control on $[0,T]$. By \cite[Prop.~5.10]{friz2010multidimensional}, $w_Y$ is optimal, in the sense that if $Y-Y_0\in C_T E$ and there exists a control $w$ such that $\| Y_{s,t}\|_E \leq w(s,t)^{1/\kappa}$, then necessarily $Y\in \dot C^{\kappa-\var}_T E$ and
	\begin{equation*}
		\llbracket Y\rrbracket_{C^{\kappa-\var}_{[s,t]} E} = w_Y(s,t)^{\frac{1}{\kappa}} \leq w(s,t)^{\frac{1}{\kappa}} \quad \forall\, [s,t]\subset [0,T].
	\end{equation*}
\end{rem}

Throughout the paper, we will mostly work with $C^{\kappa-\var}_T E$ where either $E=L^m(\Omega,\mathcal{F},\PP;\RR^\ell)$ or $E=\RR^\ell$, for some $m\in [1,\infty]$ and $\ell\in\NN$.
We repeatedly use the respective shorthands, $C^{\kappa-\var}_T L^m_\omega$ and $C^{\kappa-\var}_T$, similarly for their (semi)norms and for $\dot C^{\kappa-\var}_T L^m_\omega$.
In the case of stochastic processes, unless specified otherwise, we will implicitly assume $Y$ is adapted with respect to the reference filtration $\FF$; observe that such processes form a closed linear subspace of $C^{\kappa-\var}_T L^m_\omega$. 

A connection between the process $Y$ belonging to $C^{\kappa-\var}_T L^m_\omega$ and its realisations being of finite $\tilde{\kappa}$-variation is given by the following variant of Kolmogorov's lemma.

\begin{lem}\label{lem:kolmogorov_control}
	Let $Y\in C^{\kappa-\var}_T L^m_\omega$ be a process with continuous paths; denote by $w_Y$ the associated control as defined in \eqref{eq:def_control_Y} and assume that $m >\kappa$. Then for any $\tilde{\kappa}\in [1,\infty)$ satisfying
	\begin{equation}\label{eq:kolmogorov_parameters}
		\frac{1}{\tilde \kappa} < \frac{1}{\kappa}-\frac{1}{m},
	\end{equation}  
	it holds that $\PP$-a.s. $Y\in C^{\tilde\kappa-\var}_T$. Moreover, there exist a constant $C\coloneqq C(\kappa,\tilde\kappa)>0$ and a random variable $K$ such that $\PP$-a.s.
	\begin{equation}\label{eq:kolmogorov_control_eq1}
		|Y_{s,t}| \leq K w_Y(s,t)^{\frac{1}{\tilde \kappa}} \quad\forall\, [s,t]\subset [0,T]
		\quad\text{and}\quad \| K\|_{L^m_\omega} \leq C w_Y(0,T)^{\frac{1}{\kappa}-\frac{1}{\tilde\kappa}}.
	\end{equation}
	In particular, for any $[s,t]\subset [0,T]$, it holds that
	\begin{equation}\label{eq:kolmogorov_control_eq2}
		\Big\| \llbracket Y\rrbracket_{C^{\tilde\kappa-\var}_{[s,t]}}\Big\|_{L^m_\omega} \leq C\, \llbracket Y\rrbracket_{C^{\kappa-\var}_{[s,t]} L^m_\omega}.
	\end{equation}
	A similar estimate to \eqref{eq:kolmogorov_control_eq2} holds when $\llbracket Y\rrbracket_{C^{\tilde\kappa-\var}_{[s,t]}}$ and $\llbracket Y\rrbracket_{C^{\kappa-\var}_{[s,t]}}$ are replaced by $\| Y\|_{C^{\tilde\kappa-\var}_{[s,t]}}$ and $\| Y\|_{C^{\kappa-\var}_{[s,t]}}$ respectively.
\end{lem}

\begin{proof}
	Observe that, upon shifting and scaling, we only need to prove \eqref{eq:kolmogorov_control_eq2} for $[s,t]=[0,T]$.
	
	First assume that $w_Y$ is strictly increasing, in the sense that $w_Y(u,v)<w_Y(s,t)$ whenever $[u,v]\subsetneq [s,t]$.
	We can then apply \cite[Prop.~A.1]{le2023banach} for $\alpha=1/\kappa$, $\beta=1/\tilde\kappa$ to deduce the existence of a random variable $K$ such that \eqref{eq:kolmogorov_control_eq1} holds, exactly given by
	\begin{equation}\label{eq:kolmogorov_kappa}
		K := \sup_{0\leq s<t\leq T} \frac{|Y_{s,t}|}{w(s,t)^{\frac{1}{\tilde \kappa}}};
	\end{equation}
	since we are already assuming $Y$ to be continuous, we do not need to pass to a modification. 
	From the bounds in \eqref{eq:kolmogorov_control_eq1} and Remark~\ref{rem:link_control_variation}, we can deduce that
	\begin{align*}
		\Big\| \llbracket Y\rrbracket_{C^{\tilde\kappa-\var}_T} \Big\|_{L^m_\omega}
		\leq \Big\| K\, w_Y(0,T)^{\frac{1}{\tilde \kappa}}\Big\|_{L^m_\omega}
		= w_Y(0,T)^{\frac{1}{\tilde \kappa}} \| K\|_{L^m_\omega}
		\leq C w_Y(0,T)^{\frac{1}{\kappa}} = C \llbracket Y\rrbracket_{C^\kappa_T L^m_\omega}
	\end{align*}
	overall proving \eqref{eq:kolmogorov_control_eq2} in this case.
	
	If $w_Y$ is not strictly increasing, then for any $\eps>0$ consider $w^\eps_Y(s,t)=w_Y(s,t) + \eps (t-s)$; since $w_Y\leq w^\eps_Y$ and the latter is strictly increasing, \cite[Prop.~A.1]{le2023banach} can be applied and so we can obtain analogues of \eqref{eq:kolmogorov_control_eq1}, \eqref{eq:kolmogorov_control_eq2} with $K$ and $w_Y$ replaced by $K^\eps$, $w_Y^\eps$. Observe that since $w^\eps_Y\searrow w_Y$, the variables $K^\eps$ as defined by \eqref{eq:kolmogorov_kappa} can be taken increasing in $\eps$, so that $K=\lim_{\eps\to 0} K^\eps$ exists as a pointwise limit.
	By Fatou's lemma it then holds
	\begin{align*}
		\| K\|_{L^m_\omega} \leq \liminf_{\eps\to 0} \| K^\eps\|_{L^m_\omega}
		\leq C \liminf_{\eps\to 0} \big(w_Y(0,T) + \eps T\big)^{\frac{1}{\kappa}-\frac{1}{\tilde \kappa}}
		= C w_Y(0,T)^{\frac{1}{\kappa}-\frac{1}{\tilde \kappa}}
	\end{align*}
	yielding \eqref{eq:kolmogorov_control_eq1} also in this case.
	
	Finally, the claimed bound relating the $L^m_\omega C^{\tilde\kappa-\var}_{[s,t]}$ and $C^{\kappa-\var}_{[s,t]}L^m_\omega$-norms follows from the ones for the associated seminorms and Minkowski's inequality.
\end{proof}

\begin{rem}
	For $w_Y(s,t)=t-s$, Lemma~\ref{lem:kolmogorov_control} recovers the standard Kolmogorov continuity theorem.
\end{rem}

Interestingly, one can also relate regularity conditions involving controls and stochastic processes to fractional Sobolev spaces $W^{\beta,p}_T$. For $\beta\in (0,1)$ and $q\in [1,\infty)$, we recall that the Gagliardo seminorm $\bra{Y}_{W^{\beta,q}_T}$ was defined in \eqref{eq:gagliardo_seminorm}.

\begin{prop}\label{prop:kolmogorov_sobolev}
	Let $Y$ be a process with measurable paths; assume that there exist $\beta\in (0,1)$, $q\in [1,\infty)$, $m\in[q,\infty)$ and a control $w$ such that 
	\begin{equation}\label{eq:kolmogorov_sobolev_assumption}
		\| Y_{s,t}\|_{L^m_\omega} \leq w(s,t)^{\frac{1}{q}}  |t-s|^{\beta-\frac{1}{q}}\, \quad \forall\, [s,t]\subset [0,T].
	\end{equation}
	Then for any $\tilde\beta \in(0,\beta)$ it holds that $Y\in W^{\tilde\beta,q}_T$ $\PP$-a.s. and there exists $C\coloneqq C(\beta,\tilde \beta,q)>0$  such that
	\begin{equation}\label{eq:kolmogorov_sobolev_conclusion}
		\big\| \llbracket Y \rrbracket_{W^{\tilde\beta ,q}_T}\big\|_{L^m_\omega}
		\leq C T^{\beta-\tilde\beta} w(0,T)^{\frac{1}{q}}.
	\end{equation}
\end{prop}

\begin{proof}
	Since $m\geq q$, we can apply Minkowski's inequality to find
	\begin{align*}
		\Big\| \llbracket Y \rrbracket_{W^{\tilde\beta,q}_T}\Big\|_{L^m_\omega}
		= \left\| \bigg\| \frac{|Y_{s,t}|}{|t-s|^{1/q+\tilde\beta}}\bigg\|_{L^q([0,T]^2)} \right\|_{L^m_\omega}
		\leq \left\| \frac{\|Y_{s,t}\|_{L^m_\omega}}{|t-s|^{1/q+\tilde\beta}} \right\|_{L^q([0,T]^2)}
		\eqqcolon I^1.
	\end{align*}
	By symmetry and assumption \eqref{eq:kolmogorov_sobolev_assumption}, we have
	\begin{align*}
		I^1
		\lesssim \left( \iint_{0<s<t<T} \frac{\|Y_{s,t}\|_{L^m_\omega}^q}{|t-s|^{1+ \tilde\beta q}} \dd s \dd t \right)^{1/q}
		& \leq \left( \iint_{0<s<t<T} |t-s|^{q(\beta-\tilde\beta)-2} w(s,t) \dd s \dd t \right)^{1/q}\\
		& \leq \left( \int_0^T u^{q(\beta-\tilde\beta)-2} \int_0^{T-u} w(s,s+u) \dd s\dd u\right)^{1/q}\\
		&\eqqcolon I^2.
	\end{align*}
	Let $u\in [0,T]$ be fixed and set $k=\lfloor (T-u)/u\rfloor$; then it holds
	\begin{align*}
		\int_0^T w(s,s+u) \dd s
		& = \sum_{i=0}^{k-1} \int_{[iu, (i+1)u]} w(s,s+u) \dd s + \int_{[k u,T-u]} w(s,s+u) \dd s\\
		& = \int_{[0, u]} \sum_{i=0}^{k-1} w(s + iu,s+(i+1)u) \dd s + \int_{[k u,T-u]} w(s,s+u) \dd s\\
		& \leq \int_{[0, u]} w(s,s+ku) + \int_{[k u,T-u]} w(0,T) \dd s
		\leq 2\, u\, w(0,T),
	\end{align*}	
	where we used the superadditivity of the control $w$. Inserting this estimate in the above, we obtain
	\begin{align*}
		I^2
		\lesssim w(0,T)^{1/q} \Bigg( \int_0^T u^{q(\beta-\tilde\beta)-1} \dd u\Bigg)^{1/q}
		\lesssim w(0,T)^{1/q}\, T^{\beta-\tilde\beta},
	\end{align*}
	which yields \eqref{eq:kolmogorov_sobolev_conclusion}.
\end{proof}

\begin{rem}
	The relations between spaces like $C^{\kappa-\var}_T$ and $W^{\beta,q}(0,T)$, or even Besov spaces $B^\beta_{q_1,q_2}$, is well explored in the literature, see for instance \cite{FriVic2006,LiPrTe2020} and the references therein.
	To the best of our knowledge, estimates based on mixed terms $|t-s|^{\beta-1/q} w^{1/q}$ have received much less attention; the associated space $D^{\beta,q}(0,T)$, based on a deterministic variant of \eqref{eq:kolmogorov_sobolev_assumption}, has been introduced in \cite[App.~C]{galeati2022subcritical}. Proposition~\ref{prop:kolmogorov_sobolev} can be regarded as a (partial) stochastic counterpart of \cite[Prop.~C.1]{galeati2022subcritical}.
\end{rem}

We conclude this section with a version of the rough Gr\"onwall lemma.
The statement below differs from the standard one (cf. \cite[Lem.~2.12]{DGHT2019}) in that the function $\gamma$ appearing in \eqref{eq:rough-gronwall-hypothesis} is not assumed to be a control; the result is taken from \cite[Lem.~A.2]{HLN2021}.

\begin{lem}\label{lem:rough-gronwall}
	Let $G:[0,T]\to [0,+\infty)$ be a path such that for some constants $C$, $L>0$, $\kappa\geq 1$, some control $w$ and some increasing function $\gamma$ on $[0,T]^2_{\leq}$ (as defined in Remark~\ref{rem:properties_controls}), one has 
	\begin{equation}\label{eq:rough-gronwall-hypothesis}
		G_t-G_s \leq C \, \sup_{r\leq t}\, G_r \, w(s,t)^{\frac{1}{\kappa}} + \gamma(s,t)
	\end{equation}
	for every $s<t$ satisfying $w(s,t)\leq L$. Then, it holds that
	\begin{equation}\label{eq:rough-gronwall-conclusion}
		\sup_{t\in [0,T]} G_t \leq 2 \exp\bigg( \frac{w(0,T)}{\lambda L}\bigg) \Big( G_0 + \gamma(0,T) \Big)
	\end{equation}
	where $\lambda$ is defined as
	\begin{equation*}
		\lambda = 1 \wedge \big[L (2C e^2)^{\kappa}\big]^{-1}.
	\end{equation*}
\end{lem}

\subsection{Stochastic sewing}\label{subsec:SSL}
We will need two versions of the Stochastic Sewing Lemma (SSL), a fundamental tool first introduced in \cite{le2020stochastic}.
We will frequently use the two and three-dimensional simplex $[\cS,\cT]^2_\leq$, $[\cS,\cT]^3_\leq$ as defined in \eqref{eq:simplex}. The following version of SSL with controls and conditional norms is a particular subcase of \cite[Thm.~3.1]{le2023banach} for $\mathfrak{p}=2$ therein.

\begin{lem}\label{lem:SSL}
	Let $w_1,\,w_2$ be controls, $m,\,n$ be real numbers such that $2\leq m\leq n \leq \infty$ with $m<\infty$ and $0\leq \cS<\cT<\infty$.
	Assume that $\{A_{s,t} : (s,t)\in [\cS,\cT]^2_{\leq}\}$ is a continuous mapping from $[\cS,\cT]^2_{\leq}$ to $L^m_\omega\RR^\ell$ such that $A_{s,t}$ is $\cF_t$-measurable for all $(s,t) \in [\cS,\cT]^2_{\leq}$.
	Suppose that there exist $\eps_1,\,\eps_2>0$ such that for all $(s,u,t)\in [\cS,\cT]^3_{\leq}$ one has the bounds
	\begin{align}
		\left\| \|A_{s,t}\|_{L^m|\cF_s} \right\|_{L^n_\omega} &\leq w_1(s,t)^{\frac{1}{2} + \eps_1},\label{eq:base_SSL_cond_1} \\
		\|\EE_s \delta A_{s,u,t}\|_{L^n_\omega}&\leq w_2(s,t)^{1+\eps_2}. \label{eq:base_SSL_cond_2}
	\end{align}
	Then there exists a unique adapted stochastic process $\{\cA_t: t\in [\cS,\cT]\}$, continuous as a mapping from $[\cS,\cT]$ to $L^m_\omega$, such that $\cA_{\cS} =0$ and 
	\begin{equation}\label{eq:SSL_convergence}
		\lim_{k\to\infty} \bigg\| \cA_t-\cA_s - \sum_{j=0}^{2^k-1} A_{s+j2^{-k}(t-s),s+(j+1)2^{-k}(t-s)} \bigg\|_{L^m_\omega}=0 \quad \forall\, (s,t)\in [\cS,\cT]^2_\leq.
	\end{equation}
	Moreover, there exists a constant $C\coloneqq C(\eps_1,\eps_2,m,n,\ell)>0$ such that
	\begin{align}
		\left\| \|\cA_t -\cA_s \|_{L^m|\cF_s}\right\|_{L^n_\omega} &\leq C w_1(s,t)^{\frac{1}{2}+\eps_1} + C w_2(s,t)^{1+\eps_2}, \label{eq:base:SSL_bnd_1}\\
		\|\EE_{s}(\cA_t-\cA_s -A_{s,t}) \|_{L^n_\omega} &\leq C w_2(s,t)^{1+\eps_2}. \label{eq:base_SSL_bnd_2}
	\end{align}
\end{lem}

Another main tool we will use throughout the paper is a more refined version of Lemma~\ref{lem:SSL}, accounting for the presence of shifts.
To this end, we will frequently use the notation for the restricted simplexes $\overline{[\cS,\cT]}^2_\leq$, $\overline{[\cS,\cT]}^3_\leq$ given by \eqref{eq:simplex_restricted_2}-\eqref{eq:simplex_restricted_3}; in order to save space, we will write $s_{-(t-s)}$ in place of $s-(t-s)$.
The next statement is taken from \cite[Lem.~2.5]{galeati2022subcritical}; we invite the reader to compare the assumptions of Lemmas~\ref{lem:SSL} and \ref{lem:SSL-shift}, none of which implies the other.

\begin{lem}\label{lem:SSL-shift}
	Let $w_1,\,w_2$ be controls, $m,\,n$ be real numbers such that $2\leq m\leq n \leq \infty$, $m<\infty$, and $0\leq \cS<\cT<\infty$. Assume that $\{A_{s,t} : (s,t)\in \overline{[\cS,\cT]}^2_{\leq}\}$ is a continuous mapping from $\overline{[\cS,\cT]}^2_{\leq}$ to $L^m_\omega\RR^\ell$ such that $A_{s,t}$ is $\cF_t$-measurable for all $(s,t) \in \overline{[\cS,\cT]}^2_{\leq}$.
	Suppose that there exist constants $\eps_1,\,\eps_2>0$ such that for all $(s,u,t)\in \overline{[\cS,\cT]}^3_{\leq}$ one has the bounds,
	\begin{align}
		\left\| \|A_{s,t}\|_{L^m|\cF_s} \right\|_{L^n_\omega} &\leq w_1(s_{-(t-s)},t)^{1/2}|t-s|^{\eps_1},\label{eq:stoch_sewing_cond_1} \\
		\|\EE_{s_{-(t-s)}} \delta A_{s,u,t}\|_{L^n_\omega}&\leq w_2(s_{-(t-s)},t)|t-s|^{\eps_2}. \label{eq:stoch_sewing_cond_2}
	\end{align}
	Then there exists a unique adapted stochastic process $\{\cA_t: t\in [\cS,\cT]\}$, continuous as a mapping from $[\cS,\cT]$ to $L^m_\omega$, such that $\cA_{\cS} =0$ and \eqref{eq:SSL_convergence} holds.
	Moreover, $\cA$ is the unique process for which there exist constants $C_1,\,C_2>0$ such that for all $(s,t)\in \overline{[\cS,\cT]}^2_{\leq}$,
	\begin{equation}\label{eq:stoch_sewing_bnd_1}
		\begin{aligned}
			\left\| \|\cA_t -\cA_s -A_{s,t} \|_{L^m|\cF_s}\right\|_{L^n_\omega} \leq& C_1 w_1(s_{-(t-s)},t)^{1/2} |t-s|^{\eps_1} \\
			&+ C_2 w_2(s_{-(t-s)},t)|t-s|^{\eps_2}, 
		\end{aligned}
	\end{equation}
	\begin{equation}\label{eq:stoch_sewing_bnd_2}
		\begin{aligned}
			\|\EE_{s_{-(t-s)}}(\cA_t-\cA_s -A_{s,t}) \|_{L^n_\omega} \leq &C_2 w_2(s_{-( t-s)},t)|t-s|^{\eps_2}. 
		\end{aligned}
	\end{equation}
	Finally, there exists a constant $C\coloneqq C(\eps_1,\eps_2,m,n,\ell)>0$ such that the bounds \eqref{eq:stoch_sewing_bnd_1} and \eqref{eq:stoch_sewing_bnd_2} hold for $C_1=C_2=C$ and such that for all $(s,t)\in [\cS,\cT]^2_{\leq}$,
	\begin{equation}\label{eq:stoch_sewing_bnd_3}
		\left\|\|\cA_t-\cA_s\|_{L^m|\cF_s} \right\|_{L^n_\omega} \leq C\left(w_1(s,t)^{1/2}|t-s|^{\eps_1} + w_2(s,t)|t-s|^{\eps_2}\right). 
	\end{equation}
\end{lem}

\subsection{Integral estimates}\label{subsec:integral_estimates}

Throughout this section we will assume we are given a filtered probability space $(\Omega,\cF,\FF,\PP)$, on which an $\FF$-fBm $W^H$ of parameter $H\in (0,+\infty)\setminus\NN$, taking values in $\RR^\ell$ for suitable $\ell\in\NN$, is defined; let us stress that $\ell$ may differ from $d$, the dimension of the state space $\RR^d$ for the particle dynamics \eqref{eq:intro_IPS}  we are concerned with.

\begin{convention}
	From now on, throughout the paper, whenever we are given a filtered probability space $(\Omega,\cF,\FF,\PP)$, all the processes defined on it will be assumed to be $\FF$\mbox{-}adapted. For instance, when writing $Y\in C^{\kappa-\var}_T L^m_\omega$, we will be implicitly enforcing $\FF$-adaptedness of $Y$ as well.
\end{convention}

Throughout the paper, we will be mostly working with processes of the form $Y=\varphi+W^H$.
 We are interested in the situation where the perturbation $\varphi$, informally speaking, oscillates at a larger time scale than $W^H$, so that we can approximate the increments of $Y$ over short time intervals by those  of $W^H$ without much cost.
There are several different ways to formulate this condition, for instance by requiring $\varphi$ to be of finite $\kappa$-variation for some $\kappa<2$. To formulate them precisely, we need to specify some parameters first.
\begin{ass}\label{ass:drift_assumption}
	We consider parameters $(\alpha,q,H)$ such that
	\begin{equation}\label{eq:drift_assumption}
		H\in (0,+\infty)\setminus\NN,\quad \alpha\in (-\infty,1),\quad  q\in (1,2],\quad \alpha>1-\frac{1}{H {q^\prime}}.
	\end{equation}
\end{ass}
Clearly, Assumption \ref{ass:drift_assumption} is related to condition \eqref{eq:intro_assumption} from the introduction (and the regularity of the forthcoming drift $b$), but at this stage $(\alpha,q)$  are simply convenient parameters for which certain estimates hold.
Having specified $H$, whenever clear from now on we will just write $W$ in place of $W^H$.

For $(\alpha,q,H)$ satisfying Assumption \ref{ass:drift_assumption}, let us define
\begin{equation}\label{eq:defn_kappa_eps}
	\kappa:= \frac{1}{(\alpha-1)H+1}, \quad \eps:= (\alpha-1)H+\frac{1}{{q^\prime}};
\end{equation}
observe that by \eqref{eq:drift_assumption}, it holds that
\begin{equation}\label{eq:kappa_eps_basic}
	\kappa\in (1,q)\subset (1,2), \quad \eps\in (0,1), \quad \eps+\frac{1}{q}=\frac{1}{\kappa}.
\end{equation}
As the next lemma shows, if $\varphi\in \dot C^{\kappa-\var}_T L^m_\omega$, we can give meaning to processes of the form $\int_0^\cdot h_r(\varphi_r+ W_r)\dd r$ even when $h$ is merely a distribution.
\begin{lem}\label{lem:estimate_integrals_v0}
	Let $(\alpha,q,H)$ satisfy Assumption \ref{ass:drift_assumption}, $\kappa,\,\eps$ be defined by \eqref{eq:defn_kappa_eps}; additionally assume that $\alpha<0$ and let $\varphi\in \dot C^{\kappa-\var}_T L^m_\omega$ for some $m\in [2,\infty)$.
	Then there exists a constant $C\coloneqq C(\ell,m,\alpha,q,H)>0$ such that, for any $(s,t)\in [0,T]^2_\leq$ and any function $h\in C^\infty_b([0,T]\times\RR^\ell;\RR)$, it holds that
	\begin{equation}\label{eq:estimate_integrals_v0}\begin{split}
		\bigg\| \int_0^\cdot h_r(\varphi_r& +W_r) \dd r \,\bigg\|_{C^{\kappa-\var}_{[s,t]} L^m_\omega}\\
		&\leq C \bigg( \int_s^t \| h_r\|_{\cC^{\alpha}_x}^q\, \dd r\bigg)^\frac{1}{q} |t-s|^\eps \,\Big[ |t-s|^H + \bra{\varphi}_{C^{\kappa-\var}_{[s,t]} L^m_\omega}\Big].
	\end{split}\end{equation}
	By linearity and density, the map $h\mapsto \int_0^\cdot h_r(\varphi_r+W_r)\dd r$ extends uniquely to a bounded linear operator from $L^q_T \cC^\alpha_x$ to $C^{\kappa-\var}_T L^m_\omega$, for any $m\in[1,\infty)$.
\end{lem}
\begin{proof}
	We plan to apply SSL in the form of Lemma~\ref{lem:SSL} with $m=n$; 
	to this end, define the germ $A_{s,t}:= \EE_s \int_s^t h_r(\varphi_s+W_r)\dd r$.
	We first claim that $\mathcal{A}_\cdot=\int_0^\cdot h_r(\varphi_r+W_r)\dd r$ is the sewing of $A_{s,t}$. Indeed, for any $[s,t]\subset [0,T]$, it holds
	\begin{align*}
		\big\| \mathcal{A}_{s,t}& -A_{s,t} \big\|_{L^m_\omega}
		\leq \big\| \mathcal{A}_{s,t}-\EE_s\mathcal{A}_{s,t}\|_{L^m_\omega} + \|\EE_s\mathcal{A}_{s,t} -A_{s,t} \big\|_{L^m_\omega}\\
		& \leq 2 \Big\| \int_s^t [h_r(\varphi_r+W_r)- h_r(\varphi_s+W_s)] \dd r\Big\|_{L^m_\omega} +  \Big\| \int_s^t [h_r(\varphi_r+W_r)- h_r(\varphi_s+W_r)] \dd r\Big\|_{L^m_\omega}\\
		& \lesssim \| h\|_{L^\infty_t C^1_x} \int_s^t \left(\|\varphi_{s,r}\|_{L^m_\omega}+\|W_{s,r}\|_{L^m_\omega} \right) \dd r
		\lesssim |t-s| \bra{\varphi}_{C^{\kappa-\var}_{[s,t]} L^m_\omega} + |t-s|^{1+H};
	\end{align*}
	in the above we used several times various properties of conditional expectation, including \eqref{eq:pseudo_projection}.
	By the properties of controls, this estimate implies the validity of \eqref{eq:SSL_convergence} and thus the claim. 
	As a consequence, once we verify estimates \eqref{eq:base_SSL_cond_1}-\eqref{eq:base_SSL_cond_2} on $A_{s,t}$, the  bound \eqref{eq:base:SSL_bnd_1} for $\mathcal{A}_{s,t}$ will follow.
	
	Define the controls $$w_h(s,t):=\int_s^t \| h_r\|_{\cC^\alpha_x}^q \dd r, \quad w_\varphi:= \bra{\varphi}_{C^{\kappa-\var}_{[s,t]}L^m_\omega}^{\kappa}.$$
	By the adaptedness of $\varphi$, the LND property of the fBm $W$ and the heat kernel estimates as combined in \eqref{eq:LND+heat}, it holds
	\begin{align*}
		|A_{s,t}|
		& = \Big|\int_s^t G_{\sigma^H_{s,r}} h_r(\varphi_s+\EE_s W_r)\dd r\Big|
		\leq \int_s^t \| G_{\sigma^H_{s,r}} h_r\|_{L^\infty_x} \dd r\\
		& \lesssim \int_s^t |r-s|^{\alpha H} \| h_r\|_{\cC^\alpha_x} \dd r
		\lesssim w_h(s,t)^{\frac{1}{q}} |t-s|^{\alpha H + \frac{1}{q'}}\\
		& =: w_1(s,t)^{1+\alpha H} = w_1(s,t)^{\frac{1}{\kappa} + H},
	\end{align*}
	where in the middle passage we applied H\"older's inequality and in the last one we used Remark~\ref{rem:properties_controls} to define the control $w_1$. As the estimate is pointwise, it yields a bound on $\| A_{s,t}\|_{L^\infty_\omega}$ (thus also on $\| A_{s,t}\|_{L^m_\omega}$).
	Similarly, for $\EE_s \delta A_{s,u,t}$ we have 
	\begin{align*}
		\| \EE_s \delta A_{s,u,t}\|_{L^m_\omega}
		& = \Big\| \EE_s \EE_u \int_u^t [h_r(\varphi_u+W_r)-h_r(\varphi_s + W_r)] \dd r \Big\|_{L^m_\omega}\\
		& \leq \Big\| \int_u^t G_{\sigma^H_{u,r}}  h_r(\varphi_u+\EE_u W_r)- G_{\sigma^H_{u,r}} h_r(\varphi_s + \EE_u W_r)] \dd r \Big\|_{L^m_\omega}\\
		& \lesssim \int_u^t \| D_x  G_{\sigma^H_{u,r}} h_r\|_{L^\infty_x} \| \varphi_{s,u}\|_{L^m_\omega} \dd r\\
		& \lesssim w_\varphi(s,t)^{\frac{1}{\kappa}} \int_u^t |r-u|^{(\alpha-1) H} \| h_r\|_{\cC^\alpha_x} \dd r\\
		& \lesssim w_\varphi(s,t)^{\frac{1}{\kappa}}\,w_h(s,t)^{\frac{1}{q}}\, |t-s|^{(\alpha-1) H + \frac{1}{q'}}
		=: w_2(s,t)^{\frac{2}{\kappa}};
	\end{align*}
	observe that by Assumption \ref{ass:drift_assumption}, $(\alpha-1)H+1/q'>0$, so that all the integrals appearing above are finite. 
	By \eqref{eq:kappa_eps_basic}, it holds that $1 + \alpha H>1/2$ and $2/\kappa>1$, so we can apply Lemma~\ref{lem:SSL} (with $n=m$) to deduce that
	\begin{equation}\label{eq:integr_estim_v0_eq1}\begin{split}
			\Big\| \int_s^t h_r(\varphi_r+W_r)\dd r\Big\|_{L^m_\omega}
			& \lesssim w_1(s,t)^{\frac{1}{\kappa} + H} + w_2(s,t)^{\frac{2}{\kappa}}\\
			& = w_h(s,t)^{\frac{1}{q}}\, |t-s|^{\eps} \big(|t-s|^H + w_\varphi(s,t)^{\frac{1}{\kappa}}\big)
	\end{split}\end{equation}
	for all $[s,t]\subset [0,T]$. Finally, by~\eqref{eq:integr_estim_v0_eq1} and Remark~\ref{rem:link_control_variation}, we conclude that \eqref{eq:estimate_integrals_v0} holds as well.
	
	The last statement, concerning the unique extendability of the map $h\mapsto \int_0^\cdot h_r(\varphi_r+W_r)\dd r$ to all $h\in L^q_T \cC^\alpha_x$, is a consequence of the linearity of this map, the fact that $C^\infty_b([0,T]\times \RR^d)$ is dense in $L^q_T \cC^\alpha_x$, and standard arguments.
\end{proof}

\begin{rem}\label{rem:integral_estimates_v0}
	Let $h\in L^q_T \cC^{\alpha}_x$ with $\alpha<0$ and $Z^h$ denote the process obtained by the extension procedure described in Lemma~\ref{lem:estimate_integrals_v0}, i.e. characterized by the property that
	\begin{align*}
		\lim_{k\to\infty} \Big\| Z^h_t- \int_0^t h^k_r(\varphi_r+W_r) \dd r \Big\|_{L^m_\omega} = 0
	\end{align*}
	for all smooth sequences $\{h_k\}_k$ such that $h_k\to h \in L^q_T \cC^{\alpha}_x$.
	With this limit in mind, with an abuse of notation we will directly write $Z^h_\cdot = \int_0^\cdot h_r(\varphi_r+W_r)\dd r$; we warn the reader that this writing is only formal: when $h$ is a distribution of negative regularity, $Z^h$ is no longer meaningful as a Lebesgue integral.
	A similar comment applies to the results contained in the rest of this section.
\end{rem}

Although powerful, Lemma~\ref{lem:estimate_integrals_v0} is not always enough for our purposes. As Lemma~\ref{lem:estimate_integrals} below shows, we can obtain stronger estimates on integral processes by imposing an additional assumption on the process $\varphi$, quantitatively capturing the property of being ``well-predicted given the past''.
This property interacts nicely with the LND property \eqref{eq:LND_fBm} of $W$, as it implies that $Y$ retains the same kind of unpredictability and intrinsic stochasticity at small scales.
Following \cite{gerencser2023,galeati2022subcritical}, we require the existence of a control $w$ such that
\begin{equation}\label{eq:structure_assumption_processes}
	\big\| \EE_s[|\varphi_t - \EE_s \varphi_t|] \big\|_{L^\infty_\omega} \leq w(s,t)^{\frac{1}{q}} |t-s|^{\alpha H + \frac{1}{{q^\prime}}} \quad\forall\, [s,t]\subset [0,T],
\end{equation}
where $(\alpha,q)$ are suitable parameters as above.
\begin{lem}\label{lem:estimate_integrals}
	Let $(\alpha,q,H)$ satisfy Assumption \ref{ass:drift_assumption}, $\kappa,\,\eps$ be defined by \eqref{eq:defn_kappa_eps}.
	Let $\varphi$ be a processes satisfying condition \eqref{eq:structure_assumption_processes} for some control $w$. 
	Then for any $m\in [1,\infty)$ there exists a constant $C\coloneqq C(\ell,m,\alpha,q,H)>0$ such that for any $(s,t)\in [0,T]^2_\leq$ and any smooth function $h:[0,T]\times\RR^\ell\to \RR$ it holds that
	\begin{equation}\label{eq:estimate_integrals}
		\bigg\llbracket \int_0^\cdot h_r(\varphi_r+W_r) \dd r \bigg\rrbracket_{C^{\kappa-\var}_{[s,t]} L^m_\omega}
		\leq C \bigg( \int_s^t \| h_r\|_{\cC^{\alpha-1}_x}^q\, \dd r\bigg)^\frac{1}{q} |t-s|^\eps \big[ 1 + w(s,t)^{\frac{1}{q}} |t-s|^\eps\big].
	\end{equation}
	By linearity and density, the map $h\mapsto \int_0^\cdot h_r(\varphi_r+W_r)\dd r$ extends uniquely to a bounded linear operator from $L^q_T \cC^{\alpha-1}_x$ to $C^{\kappa-\var}_T L^m_\omega$, for any $m\in[1,\infty)$.
\end{lem}

\begin{proof}
	The statement is a consequence of the results from \cite{galeati2022subcritical}; since we will shortly provide the full proof of a similar estimate in Proposition~\ref{prop:comparison_integrals} below, let us only mention some key steps here.
	Setting $Z_t=\int_0^t h_r(\varphi_r + W_r)\dd r$ and applying the second part of \cite[Lem.~3.1]{galeati2022subcritical} (where $w_{h,\alpha-1,q}(s,t):=\int_s^t \| h_r\|_{\cC^{\alpha-1}_x}^q \dd r$), we deduce that
	\begin{equation*}
		\big\| \| Z_{s,t} \|_{L^m\vert \mathcal{F}_s} \big\|_{L^\infty_\omega}
		\lesssim \bigg( \int_s^t \| h_r\|_{\cC^{\alpha-1}_x}^q\, \dd r\bigg)^\frac{1}{q} |t-s|^\eps\, \Big[ 1 + w(s,t)^{\frac{1}{q}} |t-s|^\eps\Big];
	\end{equation*}
	observing that by \eqref{eq:moment_cond_moment_equiv} one has the inequality $\| \cdot\|_{L^m_\omega} \leq \| \| \cdot\|_{L^m\vert\mathcal{F}_s} \|_{L^\infty_\omega}$, we find
	\begin{equation}\label{eq:estimate-integrals-proof}
		\| Z_{s,t} \|_{L^m_\omega}
		\lesssim \bigg( \int_s^t \| h_r\|_{\cC^{\alpha-1}_x}^q\, \dd r\bigg)^\frac{1}{q} |t-s|^\eps\, \Big[ 1 + w(s,t)^{\frac{1}{q}} |t-s|^\eps\Big].
	\end{equation}
	By Remark~\ref{rem:properties_controls} and the relation $1/q+\eps=1/\kappa$, the right hand side of \eqref{eq:estimate-integrals-proof} is of the form $\tilde{w}^{1/\kappa}$, for a suitable control $\tilde{w}$; claim \eqref{eq:estimate_integrals} then follows from Remark~\ref{rem:link_control_variation}.
\end{proof}
\begin{rem}
	Given $h$, $\varphi$ as in Lemma~\ref{lem:estimate_integrals}, let $Z^h=\int_0^\cdot h_r(\varphi_r+W_r)\dd r$ (with the convention from Remark~\ref{rem:integral_estimates_v0}).	
	Since Lemma~\ref{lem:estimate_integrals} applies for any $m\in [1,\infty)$, by Lemma~\ref{lem:kolmogorov_control} we deduce that $Z^h\in L^m_\omega C^{\tilde\kappa-\var}_T$ for any $\tilde\kappa>\kappa$; similarly, by Proposition~\ref{prop:kolmogorov_sobolev}, $Z^h\in L^m_\omega W^{\beta,q}_T$ for any $\beta<1/\kappa$.
\end{rem}
\begin{rem}\label{rem:estimate_integrals_mixed}
	It is easy to see that the techniques from Lemmas~\ref{lem:estimate_integrals_v0} and~\ref{lem:estimate_integrals} can be combined as follows: Let $(\alpha,q,H)$ satisfy Assumption \ref{ass:drift_assumption} with $\alpha<0$, $\kappa,\,\eps$ defined by \eqref{eq:defn_kappa_eps}, $\varphi$ satisfying condition \eqref{eq:structure_assumption_processes} for some control $w$ and $\psi\in \dot C^{\kappa-\var}_T L^m_\omega$ for some $m\in [2,\infty)$. Then for any $(s,t)\in [0,T]^2_\leq$ and any function $h\in C^\infty_b([0,T]\times\RR^\ell;\RR)$, it holds that
	\begin{align*}
		\bigg\| \int_0^\cdot h_r(\varphi_r+\psi_r+W_r) \dd r \,\bigg\|_{C^{\kappa-\var}_{[s,t]} L^m_\omega}
		\lesssim \bigg( & \int_s^t \| h_r\|_{\cC^{\alpha}_x}^q\, \dd r\bigg)^\frac{1}{q} |t-s|^\eps\times\\
		& \times \Big[1 + w(s,t)^{\frac{1}{q}} |t-s|^\eps + |t-s|^H + \bra{\psi}_{C^{\kappa-\var}_{[s,t]} L^m_\omega}\Big];
	\end{align*}
	as before, by standard arguments the bound extends to all $h\in L^q_T \cC^\alpha_x$.
\end{rem}
Having defined the map $\varphi\mapsto \int_0^\cdot h_r(\varphi_r+W_r) \dd r$ in the previous lemmas, we can now establish its local Lipschitz continuity in suitable norms.

\begin{prop}\label{prop:comparison_integrals}
	Let $(\alpha,q,H)$ satisfy Assumption \ref{ass:drift_assumption}, $\kappa,\,\eps$ be defined by \eqref{eq:defn_kappa_eps},  $m\in [2,\infty)$ and $h:[0,T]\times\RR^\ell\to \RR$ be a smooth function; define the control
	\begin{align*}
		w_h(s,t):=\int_s^t \| h_r\|^q_{\cC^{\alpha}_x} \dd r.
	\end{align*}
	Then, given a process $\varphi^1$ satisfying condition \eqref{eq:structure_assumption_processes} for some control $w$, and another process $\varphi^2$ such that $\varphi^1-\varphi^2\in C^{\kappa-\var}_T L^m_\omega$, there exists a constant $C\coloneqq C(l,m,n,\alpha,q,H)>0$ such that, for any $(s,t)\in [0,T]^2_\leq$
	\begin{equation}\label{eq:comparison_integrals}\begin{split}
			\bigg\llbracket \int_0^\cdot [h_r(\varphi^1_r+W_r) & - h_r(\varphi^2_r + W_r)] \dd r \bigg\rrbracket_{C^{\kappa-\var}_{[s,t]} L^m_\omega}\\
			& \leq C w_h(s,t)^{\frac{1}{q}} |t-s|^{\eps} \| \varphi^1-\varphi^2\|_{C^{\kappa-\var}_{[s,t]} L^m_\omega} \big(1 + w(s,t)^{\frac{1}{q}} |t-s|^{\eps} \big).
	\end{split}\end{equation}
	By linearity and density, the inequality \eqref{eq:comparison_integrals} extends to any $h\in L^q_T \cC^\alpha_x$.
\end{prop}

\begin{rem}\label{rem:asymmetric}
	Estimate \eqref{eq:comparison_integrals} is asymmetric, in the fact that $\varphi^1$ is required to satisfy \eqref{eq:structure_assumption_processes}, and only the control $w$ associated to $\varphi^1$ appears, while less information is needed on $\varphi^2$.
	This will be crucial to achieve stability estimates in the style of Proposition~\ref{prop:generic_pairwise_stabillity}, as well as the uniqueness statements given in Section~\ref{sec:main_proofs}.
\end{rem}

\begin{proof}[Proof of Proposition~\ref{prop:comparison_integrals}]
	Throughout the proof we set $\psi\coloneqq  \varphi^2-\varphi^1$.
	We only need to consider smooth $h$, as the second part of the statement follows by passing to the limit and making use of either Lemma~\ref{lem:estimate_integrals} or Remark~\ref{rem:estimate_integrals_mixed} (with $\varphi=\varphi^1$, $\psi$ as above).
	As usual, we make use of stochastic sewing techniques, by introducing a well chosen germ $A_{s,t}$. For $i=1,2$, set
	\begin{align*}
		A^i_{s,t} &:= \EE_{s_{-(t-s)}} \int_s^t h_r(\EE_{s_{-(t-s)}} \varphi^i_r + W_r) \dd r,
		\quad A_{s,t}:=A^1_{s,t}-A^2_{s,t},\\
		\cA^i_{s,t} & := \int_s^t h_r(\varphi^i_r + W_r) \dd r, \qquad \qquad \qquad \quad \ \ \cA_{s,t}:=\cA^1_{s,t}-\cA^2_{s,t}.
	\end{align*}
	We claim that $\cA^i$ (resp. $\cA$) is the stochastic sewing of $A^i$ (resp. $A$); we check it for the slightly more challenging case $i=2$, the other one being similar.
	Let us fix $(s,t)\in [0,T]_\leq^2$ and set $s_-=s-(t-s)$. We have the trivial estimate
	\begin{equation}\label{eq:proof-stability-trivial}
		\| A^2_{s,t}\|_{L^\infty_\omega} \leq \| h\|_{L^\infty_t C^0_x} |t-s|,\quad
		| \cA^2_{s,t}\|_{L^\infty_\omega} \leq \| h\|_{L^\infty_t C^0_x} |t-s|.
	\end{equation}		
	Moreover, by addition and subtraction, the triangle inequality and recalling that $\psi\coloneqq  \varphi^2-\varphi^1$
	\begin{align*}
		|\EE_{s_-}\cA_{s,t}^2-\EE_{s_-} A^2_{s,t}|
		& \leq \EE_{s_-} \int_s^t |h_r(\varphi^2_r+W_r)-h_r(\EE_{s_-} \varphi^2_r+W_r)| \dd r\\
		& \leq \| h\|_{L^\infty_t C^1_b} \int_s^t \big[\EE_{s_-} |\varphi^1_r-\EE_{s_-} \varphi^1_r| + \EE_{s_-} |\psi_r-\EE_{s_-} \psi_r|\big] \dd r\\
		& \lesssim |t-s|^{1+\alpha H + \frac{1}{q'}} w(s_-,t)^{\frac{1}{q}} + \int_s^t \EE_{s_-} |\psi_r-\EE_{s_-} \psi_r| \dd r
	\end{align*}
	where we used \eqref{eq:structure_assumption_processes}.
	Taking the $L^m_\omega$-norm on both sides, using Minkowski's inequality and the fact that $\| \EE_{s_-}|\psi_r-\EE_{s_-} \psi_r|\|_{L^m_\omega}\leq \| \psi_r-\EE_{s_-} \psi_r\|_{L^m_\omega} \lesssim \| \psi_r- \psi_{s_-}\|_{L^m_\omega}$ thanks to \eqref{eq:pseudo_projection}, we arrive at
	\begin{equation}\label{eq:proof-stability-trivial2}
		\|\EE_{s_-}\cA_{s,t}^2-\EE_{s_-} A^2_{s,t}\|_{L^m_\omega} \lesssim |t-s|^{1+\alpha H + \frac{1}{q'}}\, w(s_-,t)^{\frac{1}{q}} + |t-s|\, w_{\psi}(s_-,t)^{\frac{1}{\kappa}}
	\end{equation}
	where we set $w_{\psi}(s',t')=\bra{\psi}_{C^{\kappa-\var}_{[s',t']} L^m_\omega}^\kappa$.
	Estimates \eqref{eq:proof-stability-trivial}-\eqref{eq:proof-stability-trivial2} imply the validity of conditions \eqref{eq:stoch_sewing_bnd_1}-\eqref{eq:stoch_sewing_bnd_2} from Lemma~\ref{lem:SSL-shift}, for $m=n$, and so conclude the identification of $\mathcal{A}_{s,t}^2$ as a stochastic sewing of $A_{s,t}^2$ (similarly for $\mathcal{A}_{s,t}$ and $A_{s,t}$).
	
	Our aim is now to obtain suitable estimates of the form \eqref{eq:stoch_sewing_cond_1}-\eqref{eq:stoch_sewing_bnd_2} on $A_{s,t}$, in order to invoke again Lemma~\ref{lem:SSL-shift} for $m=n$ and in particular deduce \eqref{eq:stoch_sewing_bnd_3}.
	
	First of all, it holds
	\begin{equation}\label{eq:proof-stability-0}\begin{split}
			\| A_{s,t}\|_{L^m_\omega}
			& = \bigg\|  \int_s^t [  G_{\sigma^H_{s_-,r}}  h_r(\EE_{s_-} \varphi^1_r + W_r) - G_{\sigma^H_{s_-,r}}  h_r(\EE_{s_-}\varphi^2_r + W_r)] \dd r \bigg\|_{L^m_\omega}\\
			& \leq \int_s^t \| G_{\sigma^H_{s_-,r}} h_r\|_{C^1_x}\, \| \EE_{s_-}\varphi^1_r - \EE_{s_-} \varphi^2_r\|_{L^m_\omega} \dd r\\
			& \lesssim \int_s^t |r-s_- |^{(\alpha-1) H}  \| h_r\|_{\cC^\alpha_x} \| \psi_r\|_{L^m_\omega} \dd r\\
			& \lesssim |t-s|^{(\alpha-1)H+\frac{1}{q'}} \Big(\int_s^t \| h_r\|_{\cC^\alpha_x}^q \dd r\Big)^{\frac{1}{q'}} \sup_{r\in [s,t]} \| \psi_r\|_{L^m_\omega}\\
			& = |t-s|^\eps\, w_h(s,t)^{\frac{1}{q}} \sup_{r\in [s,t]} \| \psi\|_{L^m_\omega};
	\end{split}\end{equation}
	above we consecutively applied the LND property of the fBm $W$ and heat kernel estimates as in \eqref{eq:LND+heat}, the fact that $|r-s_-|=|r-s+(t-s)|\geq |t-s|$, H\"older's inequality and the definition of $\eps$ from \eqref{eq:defn_kappa_eps}.
	Next, we need to estimate $\EE_{s-(t-s)} \delta A_{s,u,t}$ for fixed $(s,u,t)\in \overline{[S,T]}^3_\leq $.
	Let us set $s_1=s-(t-s)$, $s_2=s-(u-s)$, $s_3=u-(t-u)$, $s_4=s$, $s_5=u$, $s_6=t$; these points are ordered according to their indices, except for $s_3$ and $s_4$, for which $s_4\leq s_3$ might happen, but this plays no role whatsoever.
	With this definition $\EE_{s-(t-s)} \delta A_{s,u,t}=\EE_{s_1} \delta A_{s_4,s_5,s_6}$; after an elementary rearrangement, one finds $\EE_{s_1} \delta A_{s_4,s_5,s_6} = I + J$
	for
	\begin{align*}
		& I:= \EE_{s_1} \EE_{s_2} \int_{s_4}^{s_5} \big[ h_r(\EE_{s_1} \varphi^1_r + W_r) - h_r(\EE_{s_1} \varphi^2_r + W_r) - h_r(\EE_{s_2} \varphi^1_r + W_r) + h_r(\EE_{s_2} \varphi^2_r + W_r) \big] \dd r,\\
		& J:= \EE_{s_1} \EE_{s_3} \int_{s_5}^{s_6} \big[ h_r(\EE_{s_1} \varphi^1_r + W_r) - h_r(\EE_{s_1} \varphi^2_r + W_r) - h_r(\EE_{s_3} \varphi^1_r + W_r) + h_r(\EE_{s_3} \varphi^2_r + W_r) \big] \dd r.
	\end{align*}
	We only show how to estimate $I$, since the procedure for $J$ is similar. By bringing the conditional expectation $\EE_{s_2}$ inside, using the LND property of $W$ as in \eqref{eq:LND+heat} along with Lemma~\ref{lem:taylor} from Appendix \ref{app:useful}, we obtain
	\begin{align*}
		|I| & \leq \EE_{s_1} \int_{s_4}^{s_5} \| D_x G_{\sigma^H_{s_2,r}}  h_r\|_{L^\infty_x}\, |\EE_{s_2} (\varphi^1_r-\varphi^2_r) - \EE_{s_1}(\varphi^1_r-\varphi^2_r)| \dd r\\
		& \quad + \EE_{s_1} \int_{s_4}^{s_5} \| D^2_x  G_{\sigma^H_{s_2,r}}   h_r\|_{L^\infty_x}\, |\EE_{s_2} \varphi^1_r-\EE_{s_1} \varphi^1_r|\, |\EE_{s_1} (\varphi^1_r-\varphi^2_r)| \dd r.
	\end{align*}
	Recall the notations for $\psi$ and $w_\psi$; taking the $L^m_\omega$ norm on both sides and applying heat kernel estimates, we find
	\begin{equation*}\begin{split}
			\|I\|_{L^m_\omega}
			& \leq \int_{s_4}^{s_5} |r-s_2|^{(\alpha-1)H} \| h_r\|_{\cC^\alpha_x} \|\EE_{s_2} \psi_r - \EE_{s_1}\psi_r\|_{L^m_\omega} \dd r\\
			& \quad + \int_{s_4}^{s_5} |r-s_2|^{(\alpha-2)H} \| h_r\|_{\cC^\alpha_x} \| \EE_{s_1}|\EE_{s_2} \varphi^1_r-\EE_{s_1} \varphi^1_r| \|_{L^\infty_\omega} \|\psi_r\|_{L^m_\omega} \dd r\\
			& =: I_1 + I_2.
	\end{split}\end{equation*}
	Next, using various basic properties of conditional expectations, observe that one has
	\begin{equation}\begin{split}\label{eq:proof-stability-2}
			\|\EE_{s_2} \psi_r - \EE_{s_1}\psi_r\|_{L^m_\omega}
			\leq \| \psi_r - \EE_{s_1}\psi_r\|_{L^m_\omega}
			\lesssim \| \psi_r-\psi_{s_1}\|_{L^m_\omega}
			\leq w_\psi(s_1,r)^{\frac{1}{\kappa}}.
	\end{split}\end{equation}
	By \eqref{eq:proof-stability-2} and H\"older's inequality, we can estimate $I_1$ by
	\begin{align*}
		I_1 \lesssim \bigg( \int_{s_4}^{s_5} |r-s_2|^{(\alpha-1)H {q^\prime}} \dd r\bigg)^{\frac{1}{{q^\prime}}}\, w_h (s_4,s_5)^{\frac{1}{q}}\, w_\psi(s_2,t)^{\frac{1}{\kappa}}
		\lesssim |t-s|^\eps w_h(s_1,t)^{\frac{1}{q}} w_\psi(s_1,t)^{\frac{1}{\kappa}}.
	\end{align*}
	Again by basic properties of conditional expectation, it holds
	\begin{equation}\label{eq:proof-stability-3}\begin{split}
		\EE_{s_1}|\EE_{s_2} \varphi^1_r-\EE_{s_1} \varphi^1_r|
		& = \EE_{s_1}|\EE_{s_2} ( \varphi^1_r- \EE_{s_1}\varphi^1_r)|\\
		& \leq \EE_{s_1} |\varphi^1_r- \EE_{s_1}\varphi^1_r|
		\leq w(s_1,r)^{\frac{1}{q}} |r-s_1|^{\alpha H + \frac{1}{{q^\prime}}},
	\end{split}\end{equation}
	where in the last passage we used the assumption \eqref{eq:structure_assumption_processes}.
	By \eqref{eq:proof-stability-3} and H\"older's inequality, we can estimate $I_2$ by 
	\begin{align*}
		I_2 & \lesssim \bigg( \int_{s_4}^{s_5} |r-s_2|^{(\alpha-2)H {q^\prime}} \dd r\bigg)^{\frac{1}{{q^\prime}}}\, w_h(s_4,s_5)^{\frac{1}{q}}\, w(s_1,s_5)^{\frac{1}{q}} |s_5-s_1|^{\alpha H + \frac{1}{{q^\prime}}} \sup_{r\in [s_4,s_5]} \| \psi_r\|_{L^m_\omega}\\
		& \lesssim |s_4-s_2|^{(\alpha-2)H+\frac{1}{{q^\prime}}}\, |s_5-s_4|\, |s_5-s_1|^{\alpha H + 1/{q^\prime}}\, w_h(s_4,s_5)^{\frac{1}{q}}\, w(s_1,s_5)^{\frac{1}{q}} \sup_{r\in [s_4,s_5]} \| \psi_r\|_{L^m_\omega}\\
		& \lesssim |t-s|^{2\eps} w_h(s_{1},t)^{\frac{1}{q}}\, w(s_{1},t)^{\frac{1}{q}} \sup_{r\in [s,t]} \| \psi_r\|_{L^m_\omega}.
	\end{align*}
	In the above passages, we exploited crucially the fact that $s_4$ is sufficiently distant from $s_2$, so that no infinite integrals appear, as well as the fact that
	\begin{align*}
		|s_5-s_4|\sim |s_4-s_2|\sim |s_5-s_1|\sim |t-s|.
	\end{align*}
	Let us stress that this is the main reasons for the introduction of germ $A_{s,t}$ based on the shifted conditional expectation $\EE_{s-(t-s)}$, instead of a simpler one based on $\EE_s$.
	
	Bringing together the above estimates, and performing similar ones for $J$, overall yields
	\begin{equation}\label{eq:proof-stability-4}\begin{split}
			\| \EE_{s_1} \delta A_{s,u,t,}\|_{L^m_\omega}
			& \lesssim |t-s|^\eps w_h(s_1,t)^{\frac{1}{q}} \\
			& \quad \times \bigg( w_\psi(s_{1},t)^{\frac{1}{\kappa}} + |t-s|^{\eps}\, w(s_1,t)^{\frac{1}{q}} \sup_{r\in [s,t]} \| \psi_r\|_{L^m_\omega}\bigg)
	\end{split}\end{equation}
	where the exponents appearing on the controls in \eqref{eq:proof-stability-4} satisfy $1/q+1/\kappa> 2/q\geq 1$.
	
	Estimates \eqref{eq:proof-stability-0} and \eqref{eq:proof-stability-4} therefore verify that $A$ satisfies the assumptions of Lemma~\ref{lem:SSL-shift} for $m=n$, which together with the knowledge of the associated stochastic sewing $\mathcal{A}_{\,\cdot\,} = \int_0^\cdot [h_r(\varphi^1_r+W_r)-h_r(\varphi^2_r+W_r)] \dd r$ finally yields (cf. eq.~\eqref{eq:stoch_sewing_bnd_3})
	\begin{equation}\label{eq:proof-stability-5}\begin{split}
			\bigg\| \int_s^t&  [h_r(\varphi^1_r+W_r) -h_r(\varphi^2_r+W_r)] \dd r \bigg\|_{L^m_\omega}\\
			& \lesssim |t-s|^\eps\, w_h(s,t)^{\frac{1}{q}} \Big( \sup_{r\in [s,t]} \| \psi_r \|_{L^m_\omega} +  \llbracket \psi\rrbracket_{C^{\kappa-\var}_{[s,t]} L^m_\omega} + |t-s|^{\eps}\, w(s,t)^{\frac{1}{q}} \sup_{r\in [s,t]} \| \psi_r \|_{L^m_\omega} \Big)
	\end{split}\end{equation}
	where the estimate holds uniformly over $[s,t]\subset [0,T]$.
	Finally, the upgrade from \eqref{eq:proof-stability-5} to~\eqref{eq:comparison_integrals} follows from the definition of $\| \psi\|_{C^{\kappa-\var}_{[s,t]} L^m_\omega}$ and repeated use of Remarks~\ref{rem:properties_controls} and~\ref{rem:link_control_variation}.
\end{proof}
The next result, in the spirit of Lemma~\ref{lem:estimate_integrals} and Proposition~\ref{prop:comparison_integrals}, gives similar statements in the case of measure dependent maps; we recall the notation $b(x,\mu)$ from \eqref{eq:measure_functional}.
Here, the statement is for $x\in\RR^d$, the state space of the particle dynamics \eqref{eq:intro_IPS}.

\begin{lem}\label{lem:estimate_measure_flow_integrals}
	Let $(\alpha,q,H)$ satisfy Assumption \eqref{ass:drift_assumption}, $(\kappa,\eps)$ defined by \eqref{eq:defn_kappa_eps}, and $b:[0,T]\times\RR^{2d}\to\RR$ be a smooth function.
	Further, assume we are given $\varphi$, $\tilde\varphi$ satisfying \eqref{eq:structure_assumption_processes} for the same control $w$ and define a flow of measures $t\mapsto \mu_t \in \cP(\RR^d)$ by $$\mu_t :=\cL(\tilde{\varphi}_t +W_t).$$
	Then, for any $m\in [1,\infty)$ there exists a constant $C\coloneqq C(m,d,\alpha,q,H)>0$ such that, for any $(s,t)\in [0,T]^2_{\leq}$,
	\begin{equation}\label{eq:estimate_measure_integrals}
		\bigg\llbracket \int_0^\cdot b_r(\varphi_r +W_r,
		\mu_r) \dd r \bigg\rrbracket_{C^{\kappa-\var}_{[s,t]} L^m_\omega}
		 \leq C \bigg( \int_s^t \| b_r\|_{\cC^{\alpha-1}_x}^q\, \dd r\bigg)^\frac{1}{q} |t-s|^\eps \big[ 1 + w(s,t)^{\frac{1}{q}} |t-s|^\eps\big].
	\end{equation}
	In addition, given $\varphi^1$, $\tilde{\varphi}^1$ both satisfying \eqref{eq:structure_assumption_processes} and $\varphi^2$, $\tilde{\varphi}^2$ such that $\varphi^i-\tilde\varphi^i \in C^{\kappa-\var}_t L^m_\omega$, let us define for $i=1,2$
	\begin{align*}
		\mu^i:= \cL(\tilde{\varphi}^i_t+W_t), \quad
		Z^i_\cdot:= \int_0^\cdot b_r(\varphi^i_r+W_r, \mu^i_r)\dd r,\quad
		w_{b,\alpha}=\int_s^t \| b_r\|_{\cC^\alpha_x}^q \dd r.
	\end{align*}
	Then for any $m\in [2,\infty)$, there exists a constant $C\coloneqq C(m,d,\alpha,q,H)>0$ such that for any $(s,t)\in [0,T]^2_\leq$ it holds
	\begin{equation}\label{eq:comparison_measure_integrals}\begin{split}
			\llbracket Z^1-Z^2 \rrbracket_{C^{\kappa-\var}_{[s,t]} L^m_\omega}
			& \leq C \left( \| \varphi^1-\varphi^2\|_{C^{\kappa-\var}_{[s,t]} L^m_\omega}+ \| \tilde{\varphi}^1-\tilde{\varphi}^2\|_{C^{\kappa-\var}_{[s,t]} L^m_\omega}\right)
			\\
			&\qquad\quad  \times w_{b,\alpha}(s,t)^{\frac{1}{q}} |t-s|^{\eps}\big(1 + w(s,t)^{\frac{1}{q}} |t-s|^{\eps} \big).
	\end{split}\end{equation}
	As before, by linearity and density, the inequality \eqref{eq:comparison_measure_integrals} extends to any $b\in L^q_T \cC^\alpha_x$.
\end{lem}
\begin{proof}
	The proofs of \eqref{eq:estimate_measure_integrals} and \eqref{eq:comparison_measure_integrals} follow respectively from Lemma~\ref{lem:estimate_integrals} and Proposition~\ref{prop:comparison_integrals}, up to applying them in doubled dimension and employing a conditional expectation trick. For this reason, we will only present the proof of \eqref{eq:estimate_measure_integrals}, \eqref{eq:comparison_measure_integrals} being similar.
	
	Up to enlarging the probability space, we can construct a pair $(\varphi',W')$, distributed as $(\tilde\varphi, W)$ and independent from $(\varphi,W)$, so that $\mu_t=\cL(\varphi'_t+W'_t)$; hence, letting $\mbG \coloneqq   \{\cG_t\}_{t\in [0,T]}$ denote the natural filtration generated by $(\varphi,W)$, appealing to independence it holds that
	\begin{equation*}
		\int_s^{t} b_r (\varphi_r + W_r ,\mu_r) \dd r = \EE\left[ \int_s^t b_r(\varphi_r +W_r,\varphi'_r + W'_r)  \dd r \,\bigg|\,\cG_t \right].
	\end{equation*}
	For all $m\in [1,\infty)$, conditional Jensen's inequality yields
	\begin{equation}\label{eq:conditional_bound}\begin{split}
			\left\| \int_s^t b_r(\varphi_r+W_r,\mu_r) \dd r \,\right\|_{L^m_\omega}
			& = \left\| \EE\left[\int_s^t b_r(\varphi_r +W_r,\varphi'_r + W'_r) \dd r \,\bigg|\,\cG_t\right]\right\|_{L^m_\omega}\\
			& \leq \left\| \int_s^t b_r(\varphi_r+W_r,\varphi'_r+W'_r) \dd r\,\right\|_{L^m_\omega}.
	\end{split}\end{equation}
	Define the $\RR^{2d}$-valued processes $\tilde\varphi=(\varphi,\varphi')$, $\tilde W=(W,W')$; it is easy to check that by construction  $\tilde{W}$ is an $\RR^{2d}$-valued fBm and $\tilde{\varphi}$ satisfies \eqref{eq:structure_assumption_processes}. Therefore we can apply Lemma~\ref{lem:estimate_integrals} (in dimension $\ell=2d$, for $h=b$) to deduce that
	\begin{equation}\label{eq:conditional_bound_2}
		\bigg\| \int_s^t b_r(\varphi_r +W_r,\varphi'_r+W'_r) \dd r\,\bigg\|_{L^m_\omega}
		\lesssim \bigg( \int_s^t \| b_r\|_{\cC^{\alpha-1}_x}^q\, \dd r\bigg)^\frac{1}{q} |t-s|^\eps \big[ 1 + w(s,t)^{\frac{1}{q}} |t-s|^\eps\big].
	\end{equation}
	Claim~\eqref{eq:estimate_measure_integrals} follows by combining~\eqref{eq:conditional_bound} and~\eqref{eq:conditional_bound_2} and applying Remark~\ref{rem:link_control_variation} as usual.	
\end{proof}

\section{The equations with regular drifts}\label{sec:reg.drift}

The overall aim of this paper is to prove quantitative propagation of chaos estimates for non-Lipschitz, possibly distributional interactions $b\in L^q_T \cC^\alpha_x$.
To this end, we will however start by considering smooth drifts $b$, where the result is classical and can be easily established by Sznitman's direct comparison argument, cf. \cite{sznitman1991topics} and Appendix \ref{app:lipschitz_p_system}.
The goal of this section is to study properties of solutions in the regular case, and in particular develop a priori estimates which only depend on $\| b\|_{L^q_T \cC^\alpha_x}$.
These will culminate in suitable stability estimates (cf. Proposition~\ref{prop:generic_pairwise_stabillity} and Corollary~\ref{cor:p_system_drift_stable}), which will eventually allow us to consider distributional drifts by passing to the limit along regular approximations; the latter will be performed in Section~\ref{sec:main_proofs}.

We fix some notation. For $N\geq 2$, we will consider pairwise interacting particle systems $X^{(N)}= \{X^{i;N}\}_{i=1}^N$ solving
\begin{equation}\label{eq:p_system_pairwise}
	X^{i;N}_t = X_0^i + \frac{1}{N-1}\sum_{\substack{j=1\\ j\neq i}}^N\int_0^t b_s(X^{i;N}_s,X^{j;N}_s) \dd s + W^i_t,\qquad \text{for all } i=1,\ldots,N,
\end{equation}
where $\{W^i\}_{i=1}^\infty$ is a family of i.i.d fBms with $H\in (0,\infty)\setminus \NN$ and $\{X^i_0\}_{i=1}^\infty$ is another family of $\mbR^d$ random variables independent of $\{W^i\}_{i=1}^\infty$.
We will denote by $\FF$ (the standard augmentation of) the $\sigma$-algebra $\cF_t=\sigma(X^i_0,\, W^i_r: r\in [0,t],\, i\in\NN)$.
In this section, we will always consider $b\in C^\infty_b(\RR_+\times\RR^{2d};\RR^d)$, which we will abbreviate as $b\in C^\infty_b$.
For notational simplicity, whenever the value $N$ is clear, we will write just $\sum_{j\neq i}$ in place of $\sum_{j=1,j\neq i}^N$ in expressions like \eqref{eq:p_system_pairwise}.

\subsection{A Priori and Stability Estimates for Interacting Particle Systems}\label{subsec:pairwise_a_priori}
We plan to establish a priori bounds which are uniform in $N$ and only depend on $\| b\|_{L^q_T \cC^\alpha_x}$, for $(\alpha,q,H)$ satisfying  \ref{ass:drift_assumption}.
Without further specification, we define the control
\begin{equation}\label{def.wbCa}
	w_b(s,t) := \int_s^t \|b_r\|^q_{\cC^\alpha_x}\dd r = \|b\|^q_{L^q_{[s,t]}\cC^\alpha_x}.
\end{equation}
Given a solution $X^{(N)}$ to \eqref{eq:p_system_pairwise}, we define the remainder process $\theta^{(N)}:= \{\theta^{i;N}\}_{i=1}^N$ by
\begin{equation}\label{eq:gen_remainder_def}
	\theta^{i;N}_t := X^{i;N}_t-W^i_t = X^i_0 + \frac{1}{N-1}\sum_{j\neq i}\int_0^t b_s(\theta^{i;N}_s+W^i_s,\theta^{j;N}_s+W^j_s) \dd s.
\end{equation}
The following two lemmas demonstrate a priori bounds on solutions to \eqref{eq:p_system_pairwise}, uniformly in system size $N$.
Lemma~\ref{lem:reg_apri} treats the regular case $\alpha>0$, while Lemma~\ref{lem:remainder_singular_apriori} treats the singular case $\alpha<0$ (and the borderline value $\alpha=0$).
Their proofs closely follow those of \cite[Lem.~2.1-2.4]{galeati2022subcritical} respectively.
\begin{lem}\label{lem:reg_apri}
	Let $N\geq 2$, $b\in C^\infty_b$ and $(\alpha,q,H)$ satisfy Assumption \ref{ass:drift_assumption}; additionally assume that $\alpha>0$.
	Let $X^{(N)}$ be a solution to \eqref{eq:p_system_pairwise}, $\theta^{(N)}$ be the associated remainder defined by \eqref{eq:gen_remainder_def}. For given $m\in [1,\infty)$, define $\llbracket \theta^{(N)}\rrbracket_m$ as the smallest deterministic constant such that
	\begin{equation}\label{eq:remainder_reg_apriori}
		\sup_{i=1,\ldots,N} \left\|\|\theta^{i;N}_t-\EE_s\theta^{i;N}_t\|_{L^m|\cF_s}\right\|_{L^\infty_\omega} \leq \llbracket \theta^{(N)}\rrbracket_m |t-s|^{\alpha H + \frac{1}{{q^\prime}}} w_b(s,t)^{\frac{1}{q}} \quad \forall\, (s,t)\in [0,T]^2_{\leq}.
	\end{equation}
	Then there exists a constant $C\coloneqq C(\alpha,q,H,d,T)>0$ such that
	\begin{equation}\label{eq:priori_estim_particle}
		\sup_{N\geq 2}\, \llbracket \theta^{(N)} \rrbracket_m \leq C\Big( 1+ \| b\|_{L^q_T \cC^\alpha_x}^{\frac{\alpha}{1-\alpha}} \Big). 
	\end{equation}
\end{lem}

\begin{proof}	
	First define, for any $\beta>0$, the quantity $\llbracket \theta^{(N)} \rrbracket_{m;\beta}$ by replacing $|t-s|^{\alpha H + 1/{q^\prime}}$ in \eqref{eq:remainder_reg_apriori} with $|t-s|^\beta$.
	By H\"older's inequality, for any $i$, $N$ and any $[0,T]^2_\leq$, by \eqref{eq:gen_remainder_def} we have the pointwise estimate
	\begin{equation}\label{eq:pointwise_holder}
		|\theta^{i;N}_{s,t}|
		\leq \frac{1}{N-1}\sum_{j\neq i} \int_s^t |b_r(X^{i;N}_r,X^{j;N}_r)| \dd r
		\leq \int_{s}^{t} \|b_r\|_{C^0_x}\dd r.
	\end{equation}
	Since $\alpha>0$, we have the embedding $\cC^\alpha_x\hookrightarrow \cC^0_x=BUC_x$, the latter being endowed with $C^0_x$-norm;
	applying this fact together with \eqref{eq:pseudo_projection}, \eqref{eq:pointwise_holder} and H\"older inequality, one sees that
	\begin{align*}
		\left\|\|\theta^{i;N}_t-\EE_s\theta^{i;N}_t\|_{L^m|\cF_s}\right\|_{L^\infty_\omega}
		\lesssim \left\|\|\theta^{i;N}_{s,t}\|_{L^m|\cF_s}\right\|_{L^\infty_\omega}
		\leq \int_{s}^{t} \|b_r\|_{\cC^\alpha_x}\dd r
		\lesssim (t-s)^{\frac1{q'}} w_b(s,t)^{\frac1q},
	\end{align*}
	which shows that  $\sup_{N\ge2}\llbracket \theta^{(N)} \rrbracket_{m,1/q'}$ is finite. 
	
	We now follow a bootstrapping argument from \cite{galeati2022subcritical} which will show that $\llbracket \theta^{(N)} \rrbracket_{m}$ is finite.
	Using \eqref{eq:pseudo_projection}, the definition of $\theta^{i;N}$ and the assumed H\"older regularity of $b$, for any $i=1,\ldots, N$, one has 
	\begin{align*}
		\| \theta^{i;N}_t &- \EE_s \theta^{i;N}_t\|_{L^m|\cF_s}
		\lesssim \bigg\| \theta^{i;N}_t - \theta^{i;N}_s - \frac{1}{N-1}\sum_{j\neq i} \int_s^t b_r(\EE_s[\theta^{i;N}_r+W^{i}_r,\theta^{j;N}_r+W^{j}_r]) \dd r\, \bigg\|_{L^m|\cF_s}\\
		& \lesssim \frac{1}{N-1}\sum_{j\neq i}  \int_s^t \|b_r(\theta^{i;N}_r+ W^{i}_r,\theta^{j;N}_r +W^{j}_r)-b_r(\EE_s[\theta^{i;N}_r+W^{i}_r,\theta^{j;N}_r+W^{j}_r])\|_{L^m|\cF_s} \dd r\\
		& \lesssim \frac{1}{N-1}\sum_{j\neq i} \int_s^t \| b_r\|_{\cC^\alpha_x} \left( \| \theta^{i;N}_r-\EE_s\theta^{i;N}_r\|^\alpha_{L^m|\cF_s} + \| \theta^{j;N}_r-\EE_s\theta^{j;N}_r\|^\alpha_{L^m|\cF_s}\right) \dd r\\
		& \quad +  \frac{1}{N-1}\sum_{j\neq i} \int_s^t \| b_r\|_{\cC^\alpha_x} \left( \| W^i_r-\EE_sW^i_r\|^\alpha_{L^m|\cF_s} + \| W^j_r-\EE_sW^j_r\|^\alpha_{L^m|\cF_s}\right) \dd r,
	\end{align*}
	where we applied both conditional Minkowski's and conditional Jensen's inequalities, using that $\alpha\in (0,1)$.
	By  Gaussianity and the LND property \eqref{eq:LND_fBm}, we see that
	\begin{align*}
		\| W^H_r-\EE_s W^H_r\|_{L^m|\cF_s} = \| W^H_r - \EE_s W^H_r\|_{L^m_\omega} \sim |r-s|^H.
	\end{align*}
	Thus, under the assumption that $\llbracket \theta^{(N)} \rrbracket_{m;\beta}<\infty$ for some $\beta>0$, we arrive at the almost sure bound
	\begin{align*}
		\| \theta^{i;N}_t - \EE_s \theta^{i;N}_t \|_{L^m\vert \cF_s}
		& \lesssim \left(\llbracket \theta^{(N)}\rrbracket_{m;\beta}^\alpha |t-s|^{\alpha \beta}w_b(s,t)^{\alpha/q} + |t-s|^{\alpha H}\right) \int_s^t \| b_r\|_{\cC^\alpha_x}  \dd r\\
		&\lesssim  \left(\llbracket \theta^{(N)}\rrbracket_{m;\beta}^\alpha |t-s|^{\alpha \beta}w_b(s,t)^{\alpha/q} + |t-s|^{\alpha H}\right) |t-s|^{1/{q^\prime}} w_b(s,t)^{1/q} \\
		&  \lesssim |t-s|^{\alpha (\beta\wedge H) + 1/{q^\prime}} w_b(s,t)^{1/q} (1+ \llbracket \theta^{(N)}\rrbracket_{m;\beta}^\alpha\, w_b(s,t)^{\alpha/q}),
	\end{align*}
	where we applied H\"older's inequality in the second line. Observing that the previous estimates are uniform over $i=1,\ldots,N$, we have shown that if $\llbracket \theta^{(N)}\rrbracket_{m;\beta}$ is finite, then so is $\llbracket \theta^{(N)}\rrbracket_{m,\beta'}$ for $$\beta':= \alpha (\beta\wedge H)+\frac{1}{q^\prime}.$$
	
	If $1/q'\ge H$, then this already implies that $\bra{\theta^{(N)}}_m$ is finite. When $1/q'< H$, we can iterate this argument, starting from $\beta_0=1/q'$. After $n$ iterations, we see that $\bra{\theta^{(N)}}_{m;\beta_n}$ is finite, for $\beta_n$ defined recursively by the relation
	\begin{align*}
		\beta_n=\alpha(\beta_{n-1}\wedge H)+\frac1{q'}, \quad \beta_0=\frac1{q'}.
	\end{align*}
	Using condition \eqref{eq:drift_assumption}, it is elementary to see that for any $n$ sufficiently large, $\beta_{n-1}>H$ and hence $\beta_n=\alpha H+1/q'$. Thus, we have shown that $\bra{\theta^{(N)}}_m$ is finite.
	
	We now refine the previous argument in order to obtain the quantitative estimate \eqref{eq:priori_estim_particle}. Going through the same computations as above,  putting $\beta:=\alpha H+1/{q^\prime}$ (so that $\llbracket \theta^{(N)}\rrbracket_{m;\beta}= \llbracket \theta^{(N)}\rrbracket_{m}$), one finds that
	\begin{align*}
		\left\|\| \theta^{i;N}_t - \EE_s \theta^{i;N}_t \|_{L^m|\cF_s}\right\|_{L^\infty_\omega} \lesssim |t-s|^{\alpha \beta +1/{q^\prime}} w_b(s,t)^{1/q+\alpha/q} \llbracket \theta^{(N)} \rrbracket_{m}^\alpha + |t-s|^{\alpha H + 1/{q^\prime}} w_b(s,t)^{1/q}.
	\end{align*}
	By assumption \eqref{eq:drift_assumption}, we have that $\beta>H$. Dividing both sides by $|t-s|^{\alpha H + 1/{q^\prime}}w_b(s,t)^{1/q}$ and taking supremum over $i,s,t$, one can find a constant $C\coloneqq C(\alpha,q,d)>0$ such that
	\begin{align*}
		\llbracket \theta^{(N)} \rrbracket_{m}
		& \leq C T^{\alpha (\beta-H)} \| b\|_{L^q_T \cC^\alpha_x}^\alpha \llbracket \theta^{(N)} \rrbracket_{m}^\alpha + C\\
		& \leq  \tilde C T^{\frac{\alpha}{1-\alpha} (\beta-H)+1/{q^\prime}} \| b\|_{L^q_T \cC^\alpha_x}^\frac{\alpha}{1-\alpha} + \frac{1}{2} \llbracket \theta^{(N)} \rrbracket_{m} + C,
	\end{align*}
	where in the second passage we used Young's inequality and $\tilde{C}$ is another constant with the same dependencies as $C$. Since $\bra{\theta^{(N)}}_m$ is finite, this readily implies the bound \eqref{eq:priori_estim_particle}.
\end{proof}
We have a similar in spirit, but technically different, a priori bound in the case $\alpha\leq 0$, where we restrict to $H<1/2$.
\begin{lem}\label{lem:remainder_singular_apriori}
	Let $N\geq 2$, $b\in C^\infty_b$  and $(\alpha, q,H)$ satisfy Assumption \ref{ass:drift_assumption}; additionally assume that $\alpha \leq 0$ (so that necessarily $H<1/2$).
	Let $X^{(N)}$ be a solution to \eqref{eq:p_system_pairwise}, $\theta^{(N)}$ be the associated remainder defined by \eqref{eq:gen_remainder_def}.
	Then for any $m\in [1,\infty)$, there exists a constant $C\coloneqq C(\alpha,q,H,d,T,m)>0$ such that 
	\begin{equation}\label{eq:singular_rem_apriori}
		\sup_{N\geq 2}\max_{i=1,\ldots,N}\left\|\|\theta^{i;N}_{s,t}\|_{L^m|\cF_s}\right\|_{L^\infty_\omega} 
		\leq C\Big(1+ \|b\|^{\frac{-\alpha H}{(\alpha-1)H+1}}_{L^q_T\cC^\alpha_x}\Big)w_b(s,t)^{\frac{1}{q}}|t-s|^{\alpha H+\frac{1}{q^\prime}},
	\end{equation}
	for all $(s,t)\in [0,T]^2_{\leq}$.
\end{lem}
\begin{rem}\label{rem:a_priori_norm_relations}
	Observe that, by properties of conditional norms, one has the control
	\begin{equation}\label{eq:a_priori_norm_relations}
		\|\theta^{i;N}_{s,t}\|_{L^m_\omega}
		+ \big\|\|\theta^{i;N}_t - \EE_s\theta^{i;N}_t\|_{L^m|\cF_s}\big\|_{L^\infty_\omega}
		\lesssim \big\|\|\theta^{i;N}_{s,t}\|_{L^m|\cF_s}\big\|_{L^\infty_\omega} 
	\end{equation}
	uniformly over $i,\,N$. In particular, estimate \eqref{eq:singular_rem_apriori} is stronger than the analogue of \eqref{eq:remainder_reg_apriori} for $\alpha\leq 0$.
\end{rem}

\begin{proof}
	In view of Jensen's inequality, we can assume without loss of generality that $m\in [2,\infty)$.
	
	Consider first $\alpha=0$; here by \eqref{eq:pointwise_holder} and H\"older's inequality, for any $i$, $N$ and any $[0,T]^2_\leq$, we have the pointwise estimate
	\begin{equation}\label{eq:pointwise_holder_bis}
		|\theta^{i;N}_{s,t}|
		\leq \int_{s}^{t} \|b_r\|_{C^0_x}\dd r 
		\leq |t-s|^{1/{q^\prime}} w_{b}(s,t)^{1/q};
	\end{equation}
	taking the $\| \| \cdot\|_{L^m|\cF_s} \|_{L^\infty_\omega}$-norm on both sides, we obtain \eqref{eq:singular_rem_apriori} for $\alpha=0$.
	
	Assume now $\alpha<0$ and let us define $$w_{b,0}(s,t):=\int_s^t\|b_r\|_{C^0_x}\dd r;$$ since the exponent $\beta:= \alpha H+1/{q^\prime}$ is strictly less than $1/{q^\prime}$, arguing as in \eqref{eq:pointwise_holder_bis} we have 
	\begin{equation*}
		|\theta^{i;N}_{s',t'}|  \leq |t'-s'|^{1/{q^\prime}} w_{b,0}(s',t')^{1/q} \lesssim |t'-s'|^{\beta}w_{b,0}(s',t')^{1/q}.
	\end{equation*}
	Taking the $\| \| \cdot\|_{L^m|\cF_s} \|_{L^\infty_\omega}$-norm on both sides, it follows that, for any $(s,t) \in [0,T]^2_{\leq}$, we have
	\begin{equation*}
		\llbracket \theta^{i;N} \rrbracket_{\beta,w_{b,0},[s,t]} := \sup_{(s',t') \in I^2_{\leq}} \frac{\left\|\|\theta^{i;N}_{s',t'}\|_{L^m|\cF_s}\right\|_{L^\infty_\omega} }{|t'-s'|^\beta w_{b,0}^{1/q}(s',t')} <\infty
	\end{equation*}
	uniformly in $i,N$.
	We plan to use this initial bound and Lemma~\ref{lem:SSL} to control $\llbracket \theta^{i;N} \rrbracket_{\beta,w_b,[s,t]}$, which will conclude the proof.
	We fix $(s,t) \in [0,T]^{2}_{\leq}$ and $(s',u',t')\in [s,t]^3_{\leq}$. We define the germ
	\begin{equation*}
		\begin{aligned}
			A^{i;N}_{s',t'} &= \frac{1}{N-1}\sum_{j\neq i} \EE_{s'}\left[ \int_{s'}^{t'} b_r(\theta^{i;N}_{s'}+ W^i_r , \theta^{j;N}_{s'} +W^j_r)\dd r\right]
			\\
			&= \frac{1}{N-1}\sum_{j\neq i} \left[ \int_{s'}^{t'} G_{\sigma^H_{s',r}} b_r(\theta^{i;N}_{s'}+ \EE_{s'}W^i_r , \theta^{j;N}_{s'} +\EE_{s'}W^j_r)\dd r\right],
		\end{aligned}
	\end{equation*}
	where the second identity is due to the independence between $W^{i;H}$, $W^{j;H}$ as well as conditional Fubini theorem. 
	Set
	\begin{align*}
		\cA^{i;N}_{s',t'} := \frac{1}{N-1}\sum_{j\neq i} \int_{s'}^{t'} b_r(\theta^{i;N}_r+ W^i_r , \theta^{j;N}_{r} +W^j_r)\dd r;
	\end{align*}
	arguing as in the proof of Lemma~\ref{lem:estimate_integrals_v0}, it is easy to verify that $\cA^{i;N}$ is the stochastic sewing of $A^{i;N}$. Our task is reduced to establishing suitable bounds on $A^{i;N}$, $\delta A^{i;N}$.
	
	Using \eqref{eq:drift_assumption}, we have the almost sure bound
	\begin{equation}\label{eq:sing_apriori_bnd_1}\begin{split}
		|A^{i;N}_{s',t'}| & \leq  \frac{N-1}{N} \int_{s'}^{t'} \|G_{\sigma^H_{s',r}} b_r\|_{L^\infty_x}\dd r\\
		& \lesssim \int_{s'}^{t'} |r-s'|^{\alpha H} \|b_r\|_{\cC^\alpha_x}\dd r 
		\lesssim |t'-s'|^{\beta} w_b(s',t')^{1/q}.
	\end{split}\end{equation}
	To estimate $\E_{s'}\delta A_{s',u',t'}$, we note that
	\begin{align*}
		\left|\EE_{s'} [\delta A^{i;N}_{s',u',t'}] \right|
		&=  \frac{1}{N-1}\sum_{j\neq i} \left| \EE_{s'}\EE_{u'}\int_{u'}^{t'} b_r(\theta^{i;N}_{s'}+ W^i_{r},\theta^{j;N}_{s'}+W^j_r) - b_r(\theta^{i;N}_{u'}+ W^i_r,\theta^{j;N}_{u'}+W^j_r)\dd r \right|.
	\end{align*}
	For each $i,j$, applying \eqref{eq:LND+heat}, H\"older inequality and condition \eqref{eq:drift_assumption}, we have 
	\begin{align*}
		\Big|\EE_{u'}\int_{u'}^{t'} b_r(\theta^{i;N}_{s'}&+ W^i_{r},\theta^{j;N}_{s'}+W^j_r) - b_r(\theta^{i;N}_{u'}+ W^i_r,\theta^{j;N}_{u'}+W^j_r)\dd r\Big|
		\\&\le \int_{u'}^{t'} |r-u'|^{H(\alpha-1)} \|b_r\|_{\cC^\alpha_x}\dd r \left(\EE_s |\theta^{i;N}_{s',u'}| + \EE_s|\theta^{j;N}_{s',u'}|\right)
		\\&\lesssim |t'-u'|^{2\beta-H}w_b(s',t')^{1/q} w_{b,0}(s',t')^{1/q} \sup_{k=1,\ldots,N} \llbracket \theta^{k;N}\rrbracket_{\beta,w_{b,0},[s,t]}.
	\end{align*}
	From here, we obtain the almost sure bound
	\begin{align}
		\left|\EE_{s'} [\delta A^{i;N}_{s',u',t'}] \right|
		\lesssim |t'-s'|^{2\beta -H}w_b(s',t')^{1/q}w_{b,0}(s',t')^{1/q} \sup_{i=1,\ldots,N} \llbracket \theta^{i;N}\rrbracket_{\beta,w_{b,0},[s,t]}. \label{tmp.esda}
	\end{align}
	Since we further have $2 \beta-H>0$ and $2\beta-H + \nicefrac{2}{q} >1$, for any $m\in [2,\infty)$, we can apply Lemma~\ref{lem:SSL} with $n=\infty$;
	since estimates \eqref{eq:sing_apriori_bnd_1}-\eqref{tmp.esda} are uniform in $i,\,N$, it follows from \eqref{eq:base:SSL_bnd_1} that there exists a constant $C_0\coloneqq C_0(\alpha,q,H,d,m)>0$ such that
	\begin{equation*}
		\begin{aligned}
			\sup_{i=1,\ldots,N}\Big\| \big\| \theta^{i;N}_{s',t'} \big\|_{L^m|\cF_s}\Big\|_{L^\infty_\omega} 
			&\leq C_0|t'-s'|^{\beta} w_b(s',t')^{1/q}
			\\
			&\quad \times \bigg(1+|t'-s'|^{\beta-H}w_{b,0}(s',t')^{1/q}\sup_{N\geq 2} 	\llbracket\theta^{(N)}\rrbracket_{\beta,w_{b,0},[s',t']} \bigg).
		\end{aligned}
	\end{equation*}
	Dividing both sides by $|t'-s'|^{\beta} w_b(s',t')^{1/q}$, taking suprema over $(s',t')\in [s,t]^2_{\leq }$ we obtain that
	\begin{equation*}
		\begin{aligned}
			\sup_{N\geq 2} 	\llbracket\theta^{(N)}\rrbracket_{\beta,w_b,[s,t]}  \leq &C_0 \bigg(1+|t-s|^{\beta-H}w_{b,0}(s,t)^{1/q}\sup_{N\geq 2} 	\llbracket\theta^{(N)}\rrbracket_{\beta,w_{b,0},[s,t]} \bigg).
		\end{aligned}
	\end{equation*}
	This shows that $	\sup_{N\geq 2} 	\llbracket\theta^{(N)}\rrbracket_{\beta,w_b,[s,t]}$ is finite. Hence, we may upgrade \eqref{tmp.esda} to the following almost sure estimate
	\begin{align*}
		\left|\EE_{s'} [\delta A^{i;N}_{s',u',t'}] \right|
		& \lesssim \int_{u'}^{t'} |r-u|^{H(\alpha-1)} \|b_r\|_{\cC^\alpha_x}\dd r \bigg(\EE_s |\theta^{i;N}_{s',u'}| + \frac{1}{N-1}\sum_{j\neq i}\EE_s|\theta^{j;N}_{s',u'}|\bigg) \notag \\
		&\leq |t'-s'|^{2\beta -H}w_b(s',t')^{2/q} \llbracket \theta^{(N)}\rrbracket_{\beta,w_b,[s,t]}. 
	\end{align*}
	Replacing \eqref{tmp.esda} by the above inequality in the previous argument yields that
	\begin{equation}\label{eq:sing_apriori_bnd_3}
		\begin{aligned}
			\llbracket\theta^{(N)}\rrbracket_{\beta,w_b,[s,t]}   &\leq C_1 \left(1+|t-s|^{\beta-H}w_b(s,t)^{1/q} 	\llbracket\theta^{(N)}\rrbracket_{\beta,w_b,[s,t]} \right),
		\end{aligned}
	\end{equation}
	for a finite constant $C_1\coloneqq C_1(\alpha,q,H,d,m)>0$ which is again independent from $N$.
	
	We now define a new control $w_\ast$ by the identity $$w_\ast (s,t)^{1/q+\beta-H}:= w_b(s,t)^{1/q} |t-s|^{\beta-H},$$ and an increasing sequence of times $(t_k)_{k=0}^n$ by requiring that $t_0=0$, $t_n=T$ and
	\begin{align*}
		\quad w_\ast(t_k,t_{k+1})^{1/q+\beta-H}=(2C_1)^{-1} \quad \text{for } k=0,\ldots,n-2, \quad w_\ast(t_{n-1},t_n)^{1/q+\beta-H}\leq(2C_1)^{-1}.
	\end{align*}
	From \eqref{eq:sing_apriori_bnd_3}, we get that $\llbracket\theta^{(N)}\rrbracket_{\beta,w_b,[t_k,t_{k+1}]}   \leq 2C_1$ for every $k=0,\ldots,n-1$. 
	For each $(s,t)\in [0,T]^2_\leq$, there exist $j,\,l \in \NN$ such that $t_{j-1}<s\leq t_j \leq \cdots \leq t_l \leq t <t_{l+1}$. Set $\tau_{j-1} =s$, $\tau_i = t_i$ for $i=j,\ldots,l$ and $\tau_{l+1} = t$. Then it holds that
	\begin{align*}
		\Big\| \big\| \theta^{i;N}_{s,t} \big\|_{L^m|\cF_s}\Big\|_{L^\infty_\omega} 
		&\leq   \sum_{i=j-1}^l	\Big\| \big\| \theta^{i;N}_{\tau_i,\tau_{i+1}} \big\|_{L^m|\cF_s}\Big\|_{L^\infty_\omega} \\
		&\leq  \sum_{i=j-1}^l 	\Big\| \big\| \theta^{i;N}_{\tau_i,\tau_{i+1}} \big\|_{L^m|\cF_{t_i}}\Big\|_{L^\infty_\omega}\\
		&\le 2C_1  \sum_{i=j-1}^l w_b(\tau_i,\tau_{i+1})^{1/q}|\tau_i-\tau_{i+1}|^{\beta}\\
		&\leq (l+1-j)^{-\alpha H} \bigg(\sum_{i=j-1}^l \left(w_b(\tau_i,\tau_{i+1})^{1/q}|\tau_i-\tau_{i+1}|^{\beta}\right)^{\frac{1}{1+\alpha H}} \bigg)^{1+\alpha H}\\
		&\leq (l+1-j)^{-\alpha H} w_b(s,t)^{1/q}|t-s|^{\beta}.
	\end{align*}
	In the last two passages we used that $1+\alpha H = \beta+1/q  \in (0,1)$, Jensen's inequality and the super-additivity of the control $(s,t)\mapsto (w_b(s,t)^{1/q}|t-s|^{\beta})^{\frac{1}{1+\alpha H}}$. The length $l+1-m$ of the above sum is controlled by the number of intervals $[t_k,t_{k+1}]$ required to cover $[0,T]$; however, this in turn is controlled by $w_\ast(0,T) = \|b\|^{\frac{1}{(\alpha-1)H+1}}_{L^q_T\cC^\alpha_x} T^{\beta-H}$. 
	It follows that 
	\begin{align*}
		\Big\| \big\| \theta^{i;N}_{s,t} \big\|_{L^m|\cF_s}\Big\|_{L^\infty_\omega} \le C \|b\|^{\frac{-\alpha H}{(\alpha-1)H+1}}_{L^q_T\cC^\alpha_x} w_b(s,t)^{1/q}|t-s|^{\beta}
	\end{align*}
	for some finite constant $C$ which depends only on $\alpha,q,H,d,T,m$. This concludes the proof.	
\end{proof}

With these preparations in hand, we can establish an abstract stability result, comparing the evolution of $N$ interacting particles $X^{(N)}$ to a nearby perturbed system $Y^{(N)}$.
Proposition~\ref{prop:generic_pairwise_stabillity} may be regarded as one of the fundamental achievements of this paper and the cornerstone of our main results. 
Recall that, for parameters $(H,\alpha,q)$ satisfying Assumption~\ref{ass:drift_assumption}, we define $\kappa,\,\eps$ by \eqref{eq:defn_kappa_eps}, which satisfy \eqref{eq:kappa_eps_basic}.
\begin{prop}\label{prop:generic_pairwise_stabillity}
	Let $N\geq 2$, $m\in [2,\infty)$, $b\in C^\infty_b$, $(\alpha, q,H)$ satisfy Assumption \ref{ass:drift_assumption} and $\kappa,\,\eps$ be defined by \eqref{eq:defn_kappa_eps}.
	Also, let $\cZ^{(N)}=\{\cZ^{i}\}_{i=1}^N$ be a family of processes in $C^{\kappa-\var}_{T}L^m_\omega$, $X^{(N)}$ be the solution to \eqref{eq:p_system_pairwise} and $Y^{(N)}=\{Y^{i;N}\}_{i=1}^N$ be the solution to the perturbed system
	\begin{equation}\label{eq:pairwise_stab_remainder}\begin{split}
			Y^{i;N}_t &= X_0^i + \frac{1}{N-1}\sum_{j\neq i} \int_0^t b_s(Y^{i;N}_s,Y^{j;N}_s) \dd s + W^i_t + \cZ^i_t,
		\quad  \text{for all } i=1,\ldots,N.
	\end{split}\end{equation}
	Then there exists a constant $C\coloneqq C(\alpha,q,H,d,T,m,\|b\|_{L^q_T\cC^\alpha_x})>0$ such that
	\begin{align}\label{est.sup.stab}
		\sup_{i=1,\ldots,N} \sup_{t\in[0,T]} \| X^{i;N}_t-Y^{i;N}_t\|_{L^m_\omega} \le C \sup_{i=1,\ldots,N} \| \cZ^i\|_{C^{\kappa-\var}_T L^m_\omega}.
	\end{align}
	Assume additionally that there is a control $w_\czz$ such that
	\begin{align}\label{con.zunif}
		\sup_{i=1,\ldots,N}\bra{\czz^i}_{C^{\kappa-\var}_{[s,t]}L^m_\omega}\le w_\czz(s,t)^{\frac1 \kappa} \quad \forall(s,t)\in[0,T]^2_\le.
	\end{align}
	Then for every $\tilde \kappa\in[1,\infty)$ satisfying
	\begin{align*}
		\frac1{\tilde \kappa}<\frac 1\kappa-\frac1m,
	\end{align*}
	there exists a constant $\tilde C$, depending on $\tilde \kappa$ and the aforementioned parameters,
	such that
	\begin{equation}\label{est.aspbpvar}
		\sup_{i=1,\ldots,N} \left\| \|X^{i;N}-Y^{i;N}\|_{C^{\tilde\kappa-\var}_{T}}\right\|_{L^m_\omega}  \le \tilde C \left(\sup_{i=1,\ldots,N}\|\czz^{i;N}_0\|_{L^m_\omega}+w_\czz(0,T)^{\frac 1 \kappa}\right).
	\end{equation}
	The constants  $C,\,\tilde C$ in \eqref{est.sup.stab}-\eqref{est.aspbpvar} are monotone in $\|b\|_{L^q_T\cC^\alpha_x}, T$, and independent from $N$.
\end{prop}
\begin{proof}
	Recall $\theta^{(N)}$ as defined in \eqref{eq:gen_remainder_def}; similarly, we set $\tilde\theta^{(N)}=\{\tilde \theta^{i;N}\}_{i=1}^N$ by
	\begin{equation}\label{eq:pairwise_remainders}
		\tilde\theta^{(N)} = Y^{(N)} -W^{(N)}.
	\end{equation}
	Let $\bar w $ be the control defined by $$\bar w(s,t)^{\frac{1}{\kappa}} \coloneqq  w_b(s,t)^{\frac{1}{q}} |t-s|^{\eps}.$$
	For clarity, we omit $N$ in the notation $\eta^{i;N}$ for $\eta\in\{X,Y,Z,\theta,\tilde \theta\}$ in the remainder of the proof. We emphasize that all constants below are independent from $N,i,j$. For distinct $i,j\in\{1,\ldots,N\}$, define the process
	\begin{align*}
		\caa^{i,j}_t=\int_0^t \left[b_r(\theta^i_r +W^i_r,\theta^j_r+W^j_r) - b_r(\tilde{\theta}^i_r+W^i_r,\tilde{\theta}^j_r+W^j_r)\right]\dd r.
	\end{align*}
	An application of Proposition~\ref{prop:comparison_integrals} with $\varphi^1=(\theta^i,\theta^j)$, $\varphi^2=(\tilde\theta^i, \tilde\theta^j)$ and $W=(W^i,W^j)$ gives
	\begin{align*}
		\bra{\cA^{i,j}}_{C^{\kappa-\var}_{[s,t]}L^m_\omega} &\le C_1\bar w(s,t)^{\frac{1}{\kappa}} \|\varphi^1- \varphi^2\|_{C^{\kappa-\var}_{[s,t]}L^m_\omega} \quad \forall(s,t)\in[0,T]^2_\le,
	\end{align*}
	for some finite constant $C_1$ depending on $w_b(0,T),T,m,\alpha,q,d,H$. 
	Here we used that due to Lemmas~\ref{lem:reg_apri}-\ref{lem:remainder_singular_apriori} (in particular making use of \eqref{eq:a_priori_norm_relations} in the case $\alpha<0$) it holds that  $(\theta^i,\theta^j)$ satisfies \eqref{eq:structure_assumption_processes} with $w =w_b$ and that it is unnecessary for $(\tilde \theta^i,\tilde \theta^j)$ to satisfy~\eqref{eq:structure_assumption_processes} (see Remark~\ref{rem:asymmetric}). The fact that $\czz$ belongs to $C^{\kappa-\var}_TL^m_\omega$ is used to verify that $\varphi^1- \varphi^2$ belongs to $C^{\kappa-\var}_TL^m_\omega$.
	Note that by triangle inequality,
	\begin{align*}
		\|\varphi^1- \varphi^2\|_{C^{\kappa-\var}_{[s,t]}L^m_\omega}\le\sum_{k\in\{i,j\}} \left(\|\theta^k_s-\tilde \theta^k_s\|_{L^m_\omega}+\bra{\theta^k-\tilde \theta^k}_{C^{\kappa-\var}_{[s,t]}L^m_\omega}\right).
	\end{align*}
	Hence, we have
	\begin{align*}
		\bra{\cA^{i,j}}_{C^{\kappa-\var}_{[s,t]}L^m_\omega}
		&\le C \bar w(s,t)^{\frac{1}{\kappa}}\sum_{k\in\{i,j\}}\left(\| \theta^k_s-\tilde{\theta}^k_s\|_{L^m_\omega}+\bra{\theta^k-\tilde{\theta}^k}_{C^{\kappa-\var}_{[s,t]} L^m_\omega} \right).
	\end{align*}

	Fix $i\in\{1,\ldots,N\}$.	For any $(s,t)\in [S,T]^{2}_{\leq}$, we obtain from the equations of $X$ and $Y$ that
	\begin{align*}
		\theta^i_{s,t}- \tilde{\theta}^i_{s,t} =X^i_{s,t} -Y^i_{s,t}  
		= \frac{1}{N-1}\sum_{j\neq i} \caa^{i,j}_{s,t}+ \cZ^{i}_{s,t}.
	\end{align*}
	It follows that
	\begin{align*}
		\bra{\theta^i -\tilde{\theta}^i}_{C^{\kappa-\var}_{[s,t]}L^m_\omega} 
		&\le  \sup_{j \neq i} \bra{\cA^{i,j}} _{C^{\kappa-\var}_{[s,t]}L^m_\omega}+ \bra{\cZ^{i}} _{C^{\kappa-\var}_{[s,t]}L^m_\omega}
		\\&\le C_1   \bar w(s,t)^{\frac{1}{\kappa}}\sup_{j}\left(\| \theta^j_s-\tilde{\theta}^j_s\|_{L^m_\omega}+\bra{\theta^j-\tilde{\theta}^j}_{C^{\kappa-\var}_{[s,t]} L^m_\omega} \right)
		+ \bra{\czz^i}_{C^{\kappa-\var}_{[s,t]}L^m_\omega}.
	\end{align*}
	Now choosing $(s,t)\in [0,T]^{2}_{\leq}$ such that 
	\begin{align}\label{tmp.stsmall}
		\bar w(s,t)^{\frac{1}{\kappa}} \leq \frac{1}{2C_1},
	\end{align}
	we find,
	\begin{align}\label{eq:pairwise_stab_increment_1}
		\sup_{j}\bra{\theta^j -\tilde{\theta}^j}_{C^{\kappa-\var}_{[s,t]}L^m_\omega}  
		\le 2C_1 \bar w(s,t)^{\frac{1}{\kappa}} \sup_{j}\| \theta^j_s-\tilde{\theta}^j_s\|_{L^m_\omega} + 2\sup_{j}\bra{\czz^j}_{C^{\kappa-\var}_{[s,t]}L^m_\omega}.
	\end{align}
	Set $$G_t := \sup_{j}\sup_{r\in [0,t]}\| \theta^j_{r}-\tilde{\theta}^j_{r}\|_{L^m_\omega}.$$ By triangle inequality, we see that 
	\begin{align*}
		G_t-G_s
		&\leq \sup_{j}\sup_{r\in [s,t]}\| \theta^j_{s,r}-\tilde{\theta}^j_{s,r}\|_{L^m_\omega} \leq \sup_{j} \bra{\theta^j -\tilde{\theta}^j}_{C^{\kappa-\var}_{[s,t]}L^m_\omega}.
	\end{align*}
	In view of \eqref{eq:pairwise_stab_increment_1}, $G$ satisfies
	\begin{align*}
		G_t-G_s\le 2 C_1 G_s \bar w(s,t)^{\frac1 \kappa}+\gamma(s,t)
	\end{align*}
	for every $(s,t)$ satisfying \eqref{tmp.stsmall},
	where $$\gamma(s,t) :=2\sup_{j}\bra{\czz^j}_{C^{\kappa-\var}_{[s,t]}L^m_\omega}.$$
	Applying Lemma~\ref{lem:rough-gronwall},  there exists a constant $C_2$ depending on $C_1, w_b(0,T),T, \varepsilon, \kappa$  such that
	\begin{equation}\label{tmp.supjt}
		\sup_{j}\sup_{t\in [0,T]} \| \theta^j_{t}-\tilde{\theta}^j_{t}\|_{L^m_\omega}=G_T \leq C_2\Big(G_0+\sup_{j}\bra{\czz^j}_{C^{\kappa-\var}_{[0,T]}L^m_\omega} \Big).
	\end{equation}
	Noticing that $G_0=\sup_{j}\|\czz^{j;N}_0\|_{L^m_\omega}$, we obtain \eqref{est.sup.stab} from the above estimate.
	Under the additional condition \eqref{con.zunif}, we substitute \eqref{tmp.supjt} into \eqref{eq:pairwise_stab_increment_1} to get
	\begin{align*}
		\sup_{j} \bra{\theta^j -\tilde{\theta}^j}_{C^{\kappa-\var}_{[s,t]}L^m_\omega} 
		&\lesssim\left(G_0+w_\czz(0,T)^{\frac 1 \kappa}\right)\bar  w(s,t)^{\frac{1}\kappa}   +w_\czz(s,t)^{\frac 1 \kappa},
	\end{align*}
	for any $(s,t)\in [0,T]^{2}_{\leq}$ satisfying \eqref{tmp.stsmall}. 
	Define the control $\bar w_\czz=w_\czz/w_\czz(0,T)$ so that
	\begin{align*}
		\left(G_0+w_\czz(0,T)^{\frac 1 \kappa}\right)\bar  w^{\frac{1}\kappa}   +w_\czz^{\frac 1 \kappa}\le \left(G_0+w_\czz(0,T)^{\frac 1 \kappa}\right)\left(\bar  w+\bar w_\czz\right)^{\frac{1}\kappa}.
	\end{align*}
	It follows from the previous two estimates that
	\begin{align*}
		\sup_{j} \bra{\theta^j -\tilde{\theta}^j}_{C^{\kappa-\var}_{[s,t]}L^m_\omega} 
		&\lesssim \left(G_0+w_\czz(0,T)^{\frac 1 \kappa}\right)\left(\bar  w(s,t)+\bar w_\czz(s,t)\right)^{\frac{1}\kappa}
	\end{align*}
	for any $(s,t)\in [0,T]^{2}_{\leq}$ satisfying \eqref{tmp.stsmall}. 
	To remove the restriction on $(s,t)$, we apply \cite[Prop.~5.10]{friz2010multidimensional} to obtain,
	\begin{align}\label{tmp.b4Kol}
		\sup_{j} \bra{\theta^j -\tilde{\theta}^j}_{C^{\kappa-\var}_{[s,t]}L^m_\omega}  
		\lesssim\left(G_0+w_\czz(0,T)^{\frac 1 \kappa}\right) w(s,t)^{\frac1 \kappa} \quad \forall (s,t)\in[0,T]^2_\le,
	\end{align}
	where $w$ is the control defined by
	\begin{align*}
		w= \bar w+\bar w_\czz + (\bar w+\bar w_\czz)^\kappa
	\end{align*}
	cf. Remark~\ref{rem:properties_controls}. We finally apply Lemma~\ref{lem:kolmogorov_control} to obtain \eqref{est.aspbpvar}.
\end{proof}
\begin{rem}
	The above proof actually yields a stronger almost sure estimate under condition \eqref{con.zunif}. Indeed, from \eqref{tmp.b4Kol}, applying Lemma~\ref{lem:kolmogorov_control}, we can find random variables $\{K^{i;N}\}_i$ such that $\PP$-a.s. for every $i,N$, every $(s,t)\in [0,T]^2_\le$,
	\begin{align*}
		|X^{i;N}_{s,t}- Y^{i;N} _{s,t}|\le K^{i;N}\left(\sup_{j=1,\ldots,N}\|\czz_0^j\|_{L^m_\omega} +w_\czz(0,T)^{\frac 1 \kappa}\right)\big(\bar w(s,t)+\bar w_\czz(s,t)\big)^{\frac1{\tilde \kappa}}
	\end{align*}
	and
	\begin{align*}
		\sup_{i=1,\ldots,N}\|K^{i;N}\|_{L^m_\omega}\le C,
	\end{align*}
	where $C$ is the same constant in \eqref{est.aspbpvar}.	
\end{rem}
We can now apply the general stability result of Proposition~\ref{prop:generic_pairwise_stabillity} to compare particle systems driven by different drifts.
\begin{cor}\label{cor:p_system_drift_stable}
	Let $b^1$, $b^2\in C^\infty_b$ and $(\alpha, q,H)$ satisfy Assumption \ref{ass:drift_assumption}. For each integer $N\geq 2$, let $X^{(N)}$, $Y^{(N)}$ denote respectively the solutions to
	\begin{equation}\label{eq:pairwise_b1_b2}
		\begin{aligned}
			X^{i;N}_t &= X_0^i + \frac{1}{N-1}\sum_{j\neq i}\int_0^t b^1_s(X^{i;N}_s,X^{j;N}_s) \dd s + W^i_t,\\
			Y^{i;N}_t &= Y_0^i + \frac{1}{N-1}\sum_{j\neq i}\int_0^t b^2_s(Y^{i;N}_s,Y^{j;N}_s) \dd s + W^i_t,
		\end{aligned}
		\quad\quad  \text{for all } i=1,\ldots,N.
	\end{equation}
	Then, for any $m\in [1,\infty)$ and $\tilde\kappa>\kappa$, there exists a constant 
	$C>0$ (depending on $\alpha$, $H$, $q$, $d$, $T$, $m$, $\|b^1\|_{L^q_T\cC^\alpha_x}$, $\|b^2\|_{L^q_T\cC^\alpha_x}$, $\tilde\kappa$) such that
	\begin{equation}\label{eq:stability-particle-by-particle}\begin{split}
		\sup_{j=1,\ldots,N} &
		\Big\|  \| X^{j;N}-Y^{j;N}\|_{C^{\tilde \kappa-\var}_{T}} \Big\|_{L^m_\omega}\\
		& \le C \left(\sup_{j=1,\ldots,N}\|X^{j;N}_0-Y^{j;N}_0\|_{L^m_\omega}+\| b^1-b^2\|_{L^q_T \cC_x^{\alpha-1}}\right).
	\end{split}\end{equation}
\end{cor}
\begin{proof}
	Define  $\tilde{\theta}^{i;N}=Y^{i;N}-W^i$. Applying Lemmas~\ref{lem:reg_apri} and \ref{lem:remainder_singular_apriori}, we have
	\begin{align*}
		\sup_{N\ge2}\sup_{i=1,\ldots,N}\|\E_s|\tilde\theta^{i;N}_t-\E_s\tilde \theta^{i;N}_t|\|_{L^\infty_\omega}\lesssim w_{b^2}(s,t)^{\frac1q}|t-s|^{\alpha H+\frac1{q'}} \quad \forall(s,t)\in [0,T]^2_\le,
	\end{align*}
	where $w_{b^2}$ is defined according to \eqref{def.wbCa}.
	We define the process
	\begin{equation*}
		\cZ^{i;N}_{t} =Y^{i;N}_0-X^{i;N}_0+ \frac{1}{N-1}\sum_{j\neq i} \int_0^t (b^1-b^2)_r(\tilde{\theta}^{i;N}_r+W^i_r,\tilde{\theta}^{j;N}_r+W^j_r)\dd r
	\end{equation*}
	so that the system for $Y^{(N)}$ may be rewritten as
	\begin{align*}
		Y^{i;N}_t &= X_0^i + \frac{1}{N-1}\sum_{j\neq i}\int_0^t b^1_s(Y^{i;N}_s,Y^{j;N}_s) \dd s + W^i_t+\czz^{i;N}_t.
	\end{align*}
	Taking into account the above estimate for $\tilde \theta$, we can apply Lemma~\ref{lem:estimate_integrals} to get that
	\begin{align*}
		\sup_{N\ge2}\sup_{i=1,\ldots,N}\bra{\czz^{i;N}}_{C^{\kappa-\var}_{[s,t]}L^m_\omega}\le \tilde C \left(\int_s^t\|b^1_r-b^2_r\|_{\C^{\alpha-1}_x}^q\dd r\right)^{\frac1q}|t-s|^\varepsilon,
	\end{align*}
	for some constant $\tilde C\coloneqq \tilde C(\alpha,q,H,d,T,m,\|b^2\|_{L^q_T\cC^\alpha_x})$. 
	This shows that $\czz$ verifies condition \eqref{con.zunif} with the control $w_\czz$ defined by
	\begin{align*}
		w_\czz(s,t)^{\frac 1 \kappa}=\tilde C\left(\int_s^t\|b^1_r-b^2_r\|_{\cC^{\alpha-1}_x}^q\dd r\right)^{\frac1q}|t-s|^\varepsilon
	\end{align*}
	and so that we are in the setting of Proposition~\ref{prop:generic_pairwise_stabillity}, with $b=b^1$.
	Hence, we obtain \eqref{eq:stability-particle-by-particle} from \eqref{est.aspbpvar}, noting that $w_\czz(0,T)^{\frac 1 \kappa}\lesssim \|b^1-b^2\|_{L^q_T\cC^{\alpha-1}_x}$.
\end{proof}
\subsection{A Priori and Stability Estimates for McKean--Vlasov Equations} 
\label{sub:mkv}
Having established a priori estimates and stability estimates for system \eqref{eq:p_system_pairwise}, we can now also infer similar bounds for solutions to the McKean--Vlasov equation
\begin{equation}\label{eq:regular_mkv}
	\bar X_t = X_0 + \int_0^t b_s(\bar X_s,\bar{\mu}_s)\dd s + W_t,\qquad \bar{\mu}_t = \cL(\bar X_t);
\end{equation}
as before, $X_0$ is an $\cF_0$-measurable random variable and $W$ an fBm of parameter $H$.

Next, we create a coupling between \eqref{eq:regular_mkv} and the interacting particle systems as follows.
Let $\{X_0^i\}_{i=1}^\infty$ and  $\{W^i\}_{i=1}^\infty$ be i.i.d. copies of $X_0$ and $W$ respectively.
This gives corresponding i.i.d. copies $\{\bar{X}^{i}\}_{i=1}$ of $\bar{X}$, so that each $\bar X^i$ is the solution to
\begin{align}\label{eq.mkvi}
	\bar{X}^{i} = X^i_0 +  \int_0^t b_s (\bar{X}^i_s,\bar{\mu}_s)\dd s + W^i_t.
\end{align}
Together with the particle system \eqref{eq:p_system_pairwise}, \eqref{eq.mkvi} forms Sznitman's classical direct coupling \cite{sznitman1991topics}. 
As usual, we denote 
\begin{align*}
	\bar\theta=\bar X-W, \quad \theta^{i;N}=X^{i;N}-W^i, \quad \bar \theta^i=\bar X^i-W^i.
\end{align*}
We can now use the smoothness of $b$ and the qualitative convergence from Lemma~\ref{lem:app_lip_MKV_propagation}, ensuring that $\theta^{i,N}\to \bar\theta$ as $N\to\infty$, to similarly obtain a priori bounds and stability estimates for $\bar\theta$ which only depend on $\| b\|_{L^q_T \cC^\alpha_x}$; these estimate will then be extended to rough drifts in Section \ref{sec:main_proofs}.

\begin{lem}\label{lem:pairwise_mkv_apriori}
	Let 
	$b\in C^\infty_b$ and $(\alpha, q,H)$ satisfy Assumption \ref{ass:drift_assumption}; let $X_0$ be $\cF_0$-measurable, $\bar X$ be the associated solution to \eqref{eq:regular_mkv}.
	If $\alpha>0$, then there exists a constant $C\coloneqq C(\alpha,q,H,d,T)>0$ such that
	\begin{equation}\label{eq:eq:mkv_apriori_alpha>0}
		\left\| \|\bar \theta_t - \EE_s\bar \theta_t\|_{L^m|\cF_s} \right\|_{L^\infty_\omega} 
		\leq C \Big( 1+ \| b\|_{L^q_T \cC^\alpha_x}^{\frac{\alpha}{1-\alpha}} \Big) w_b(s,t)^{\frac{1}{q}} |t-s|^{\alpha H+\frac{1}{q^\prime}}
		 \ \ \forall\, (s,t)\in [S,T]^2_\leq.
	\end{equation}
	If $\alpha\leq 0$, then there exists a constant $\tilde C\coloneqq \tilde C(\alpha,q,H,d,T,m)>0$ such that
	\begin{equation}\label{eq:mkv_apriori_alpha<0}
		\left\| \|\bar\theta_{s,t}\|_{L^m|\cF_s} \right\|_{L^\infty_\omega} 
		\leq \tilde C \Big(1+ \|b\|^{\frac{-\alpha H}{(\alpha-1)H+1}}_{L^q_T\cC^\alpha_x}\Big) w_b(s,t)^{\frac{1}{q}} |t-s|^{\alpha H+\frac{1}{q^\prime}}
		\quad\forall\, (s,t)\in [S,T]^2_\leq.
	\end{equation}
\end{lem}
\begin{proof}
	We first note that in either case, $\alpha>0$ or $\alpha\leq 0$, it follows from the results of Lemma~\ref{lem:reg_apri} or Lemma~\ref{lem:remainder_singular_apriori} that for all $(s,t)\in [S,T]^2_{\leq}$,
	\begin{equation*}
		\sup_{N\geq 2}\left\| \|\theta^{1;N}_t-\E_s \theta^{1;N}_t\|_{L^m|\cff_s} \right\|_{L^\infty_\omega} \leq C w_b(s,t)^{\frac{1}{q}} |t-s|^{\alpha H + \frac{1}{{q^\prime}}}.
	\end{equation*}
	Furthermore,  it follows from Lemma~\ref{lem:app_lip_MKV_propagation} that, for any $m\in [1,\infty)$
	\begin{align*}
		\lim_{N\to\infty}\Big(\theta^{1;N}_t-\E_s \theta^{1;N}_t\Big)=\bar\theta^{1}_t-\E_s \bar\theta^{1}_t \quad\text{in}\quad L^m_\omega.
	\end{align*}
	Using lower semi-continuity of the $L^\infty_\omega$ norm and the conditional version of Fatou's lemma, for all $[s,t]\in [S,T]^{2}_{\leq}$ it holds that
	\begin{equation*}
		\left\| \|\bar \theta^{1}_t-\E_s \bar \theta^{1}_t\|_{L^m|\cff_s} \right\|_{L^\infty_\omega} \leq C w_b(s,t)^{\frac{1}{q}} |t-s|^{\alpha H + \frac{1}{{q^\prime}}}.
	\end{equation*}
	The claim then follows from the fact that $\bar{\theta}^{1}$ has the same law as $\bar\theta$.\end{proof}
\begin{prop}\label{prop:pairwise_mkv_stability}
	Let $N\geq 2$, $b^1$, $b^2\in C^\infty_b$ and $(\alpha, q,H)$ satisfy Assumption \ref{ass:drift_assumption} and $\kappa,\,\eps$ be defined by \eqref{eq:kappa_eps_basic}; let $\bar X,\,\bar Y$ be the solutions to \eqref{eq:regular_mkv} with drifts $b^1,\,b^2$ and initial data $X_0,Y_0$ respectively. Then for any $m \in [1,\infty)$ and $\tilde\kappa>\kappa$ there exists $C\coloneqq C(\alpha,q,H,d,T,m,\|b^1\|_{L^q_T\cC^\alpha_x},\|b^2\|_{L^q_T\cC^\alpha_x},\tilde\kappa)>0$ such that
	\begin{equation}\label{eq:pairwise_mkv_stabillity}
		\Big\| \|\bar X-\bar Y\|_{C^{\tilde \kappa-\var}_{T}} \Big\|_{L^m_\omega} \leq C\left(\|X_0-Y_0\|_{L^m_\omega}+ \| b^1-b^2\|_{L^q_T \cC^{\alpha-1}_x}\right).
	\end{equation}
\end{prop}
\begin{proof}
	We define $Y^{i;N},\bar Y^{i}$ in a similar way as $X^{i;N},\bar X^{i}$. From Corollary~\ref{cor:p_system_drift_stable}, we have
	\begin{equation}\label{eq:pairwise_mkv_stabillity_proof}
		\Big\| \| X^{1;N}-Y^{1;N}\|_{C^{\tilde \kappa-\var}_T} \Big\|_{L^m_\omega} \leq C\left(\|X_0-Y_0\|_{L^m_\omega}+ \| b^1-b^2\|_{L^q_T \cC^{\alpha-1}_x}\right).
	\end{equation}
	From Lemma~\ref{lem:app_lip_MKV_propagation}, for $\eta\in\{X,Y\}$ we have that
	\begin{align*}
		\lim_{N\to\infty}\Big\|\sup_{t\in[0,T]}\big|\eta^{1;N}_t-\bar \eta^1_t\big|\Big\|_{L^m_\omega}=0;
	\end{align*}
	therefore we can extract a subsequence $\{N_k\}_k$ such that $\sup_{t\in[0,T]}|\eta^{1;N}_t-\bar \eta^1_t|\to 0$ $\PP$-a.s. as $k\to\infty$. 
	By the lower semi-continuity of the $C^{\tilde \kappa-\var}_T$-norm (see \cite[Lem.~5.12]{friz2010multidimensional}), we find the $\PP$-a.s. estimate
	\begin{align*}
		\|\bar X-\bar Y\|_{C^{\tilde \kappa-\var}_{T}}	
		\leq \liminf_{k\to\infty} \| X^{1;N_k}-Y^{1;N_k}\|_{C^{\tilde \kappa-\var}_T}.
	\end{align*}
	Taking the $L^m_\omega$-norm on both sides, using Fatou's lemma and \eqref{eq:pairwise_mkv_stabillity_proof}, we obtain
	\begin{align*}
		\big\|  \|\bar X-\bar Y\|_{C^{\tilde \kappa-\var}_{T}} \big\|_{L^m_\omega}
		& \leq \liminf_{k\to\infty} \big\|  \|X^{1;N_k}-Y^{1;N_k}\|_{C^{\tilde \kappa-\var}_{T}} \big\|_{L^m_\omega}\\
		& \leq C\left(\|X_0-Y_0\|_{L^m_\omega}+ \| b^1-b^2\|_{L^q_T \cC^{\alpha-1}_x}\right)
	\end{align*}		
	which yields \eqref{eq:pairwise_mkv_stabillity} as desired.	
\end{proof}
\subsection{Propagation of chaos estimates}
Another application of Proposition~\ref{prop:generic_pairwise_stabillity} gives a rate of convergence between the interacting particle system and the $N$ independent copies of the McKean--Vlasov equation. 
Let us stress that this is the first time in the paper we truly exploit the fact that the $\{X^i_0\}_i$ are i.i.d.
\begin{cor}\label{cor:estim_particle_to_mkv}
	Let $N\geq 2$, $b\in C^\infty_b$, $(\alpha, q,H)$ satisfy Assumption \ref{ass:drift_assumption} and $\kappa,\,\eps$ be defined by \eqref{eq:kappa_eps_basic}. Let $\{X^i_0\}_i$ be i.i.d. $\cF_0$-measurable random variables, $X^{(N)}$ be the associated particle system satisfying \eqref{eq:p_system_pairwise} and $\{\bar{X}^{i}\}_{i\ge1}$ be i.i.d. processes satisfying \eqref{eq.mkvi}.
	Then, for any $m\in [1,\infty)$ and any $\tilde\kappa>\kappa$, there exists a constant $C\coloneqq C(\alpha,q,H,d,T,m,\|b\|_{L^q_T\cC^\alpha_x},\tilde\kappa)>0$, independent of $N$, such that
	\begin{equation}\label{eq:estim-particle-to-mkv}
		\sup_{i=1,\ldots,N} \Big\| \|X^{i;N}-\bar{X}^{i}\|_{C^{\tilde\kappa-\var}_T} \Big\|_{L^m_\omega} \leq C N^{-1/2}.
	\end{equation}
\end{cor}

\begin{proof}
	For each $i$, we write
	\begin{align*}
		\bar X^i_t=X^i_0+\frac1{N-1}\sum_{j\neq i}\int_0^t b_r(\bar X^i_r,\bar X^j_r)\dd r+W^i_t+\czz^i_t,
	\end{align*}
	where 
	\begin{equation*}
		\czz^i_t=\frac1{N-1}\sum_{j\neq i} Z^{i,j}_t 
		\quad\text{and}\quad
		Z^{i,j}_{t} :=\int_0^t\left[b_r(\bar X^i_r,\bar{\mu}_r)-b_r(\bar X^i_r,\bar X^j_r)\right]\dd r.
	\end{equation*}
	We are going to apply Proposition~\ref{prop:generic_pairwise_stabillity}.
	Since $b$ is bounded, it is obvious that each $\czz^i$ belongs to ${C^{\kappa-\var}_{[0,T]}L^m_\omega}$.
	By Lemma~\ref{lem:pairwise_mkv_apriori}, the processes $\bar{\theta}^i=X^i-W^i$ satisfy condition \eqref{eq:structure_assumption_processes}, therefore we can invoke Lemmas~\ref{lem:estimate_integrals}-\ref{lem:estimate_measure_flow_integrals} to find that, for every $(s,t)\in[0,T]^2_\le$, we have
	\begin{align}\label{tmp.estZij}
		\sup_{i,j}\|Z^{i,j}_{s,t}\|_{L^m_\omega}\lesssim\|b\|_{L^q_{[s,t]}\cC^{\alpha-1}_x}|t-s|^\eps.
	\end{align}	
	In addition, conditional on $\bar{X}^{i}$, $\{Z^{i,j}\}_{j\neq i}$
	is a family of i.i.d random variables. 
	Since $\law(\bar{X}^j_r)= \mu_r$, we have 
	\begin{align*}
		\EE\left[\EE\left[ b_r(\bar{X}^{i;N}_r,\bar{X}^{j;N}_r) - b_r(\bar{X}^{i;N}_r,\mu_r)|\bar{X}^{i;N}_r\right]\right]
		= b_r(\bar{X}^{i;N}_r, \mu_r ) - b_r(\bar{X}^{i;N}_r,\mu_r)
		=0.
	\end{align*}
	For $m\geq 2$, we can therefore apply Lemma~\ref{lem:moments_martingale_type2} and estimate \eqref{tmp.estZij} to obtain
	\begin{align*}
		\| \cZ^i_{s,t}\|_{L^m_\omega}
		\lesssim \frac{1}{N-1} \bigg( \sum_{j\neq i} \| Z^{i,j}_{s,t}\|_{L^m_\omega}^2\bigg)^{1/2}
		\lesssim N^{-1/2} \| b\|_{L^q_{[s,t]}\cC^{\alpha-1}_x} |t-s|^\eps
	\end{align*}
	where the estimate is uniform in $i$.
	By Remark~\ref{rem:link_control_variation}, it follows that 
	\begin{equation*}
		\sup_{i=1,\ldots,N} \bra{\czz^i}_{C^{\kappa-\var}_{[s,t]}L^m_\omega}\lesssim N^{-1/2} w_\czz(s,t)^{\frac 1 \kappa} \quad \forall(s,t)\in[0,T]^2_\le
	\end{equation*}
	where $ w_\czz$ is the control defined by $ w_\czz(s,t)^{\frac1 \kappa}=\| b\|_{L^q_{[s,t]}\cC^{\alpha-1}_x} |t-s|^\eps$.
	Applying Proposition~\ref{prop:generic_pairwise_stabillity}, we obtain \eqref{eq:estim-particle-to-mkv}.
\end{proof}

\section{Well-Posedness and Propagation of Chaos}\label{sec:main_proofs}
With the results of Section~\ref{sec:reg.drift} in hand, we are ready to prove our main results, by passing to the limit with respect to regular approximations of the singular drift.
We will first show in Sections~\ref{sec:sing_drifts_sols}-\ref{sec:reg_drifts_sols} that all equations of interest are well-posed in this limit, up to introducing suitable solution concepts.
We then prove quantitative mean field convergence and propagation of chaos in Theorem~\ref{th:pairwise_mean_field}, which as a consequence yields Theorem~\ref{thm:main}, as shown in Section~\ref{sec:poc_proofs}.

We maintain most of the notations and conventions from the previous sections.
In particular, whenever clear, we write $W^i$ in place of $W^{i,H}$ and when given a collection of $N$ random variables $\{Y^i\}_{i=1}^N$, for concision we write it as $Y^{(N)}$; this applies to objects like $X^{(N)}_0$, $X^{(N)}$, $\bar X^{(N)}$, $W^{(N)}$ and so on. Given a particle system $X^{(N)}$, we write $\theta^{(N)}=X^{(N)}-W^{(N)}$ for its remainder process, which is defined $i$-wise by $\theta^{i;N}_t=X^{i;N}_t-W^i_t$, for $i=1,\ldots, N$; similarly for $\bar\theta^{(N)}$.
Whenever $N$ is fixed and clear, especially in the proofs, we will just write $X^i$ (resp. $\theta^i$) instead of $X^{i;N}$ (resp. $\theta^{i;N}$), so as to lighten notation.

Throughout the section, unless specified otherwise, $\{(X^i_0,W^i)\}_{i=1}^\infty$ are always taken i.i.d., with $X^i_0$ independent of $W^i$.
Let us also stress that we continue working with the spaces $L^q_T \cC^\alpha_x$, which have the advantage of being well-approximated by elements of $C^\infty_b$; we will only show at the very end, in the proof of Theorem~\ref{thm:main}, how the case of drifts $b\in L^q_T B^\alpha_{\infty,\infty}$ can be reduced to this one.

\subsection{Solution Concept and Well-Posedness for Distributional Drifts}\label{sec:sing_drifts_sols}

We begin by defining a suitable notion of solutions to the interacting particle system  in the case $\alpha<0$, so that the interaction $b$ is a genuine distribution and cannot be naively evaluated pointwise.
The following definition is rather standard in this context, after the work \cite{BasChe2003}.

\begin{defn}\label{def:pairwise_singular_sol}
	Let $\alpha<0$, $b\in L^q_T\cC^\alpha_x$. We say that a tuple $(\Omega,\FF,\PP;X^{(N)}_0, W^{(N)}, X^{(N)})$ is a weak solution on $[0,T]$ to the particle system
	\begin{equation}\label{eq:sing_pairwise_system}
		X^{i;N}_t = X_0^i + \frac{1}{N-1}\sum_{\substack{j=1\\ j\neq i}}^N\int_0^t b_s(X^{i;N}_s,X^{j;N}_s) \dd s + W^i_t,\qquad \text{for all } i=1,\ldots,N,
	\end{equation} 
	if the following hold:
	\begin{enumerate}[label=\roman*)]
		\item $(\Omega,\FF,\PP)$ is a filtered probability space, $W^{(N)}$ is an $\RR^{Nd}$-valued $\FF$-fBm and $X^{(N)}_0$ is $\cF_0$-adapted (in particular, this implies that $W^i$ and $W^j$ are independent for $i\neq j$, and that $W^{(N)}$ and $X^{(N)}_0$ are independent);
		\item $X^{(N)}$ is $\FF$-adapted;
		\item there exists a sequence $\{b^n\}_{n\geq 1}\subset C^\infty_b$ such that $b^n\to b$ in $L^q_T \cC^{\alpha}_x$, and for all $i=1,\ldots, N$ it holds that
		\begin{equation}\label{eq:defn_pairwise_singular_sol}
			X^{i;N}_{\,\cdot\,} -X_0^{i;N} - W^{i}_{\,\cdot\,}  = \lim_{n\to \infty}\frac{1}{N-1}\sum_{j\neq i} \int_{0}^{_{\,\cdot\,} } b^n_s(X^{i;N}_s,X^{j;N}_s)\dd s,
		\end{equation}
		where the limit is taken in probability with respect to $\PP$, uniformly on $[0,T]$.
	\end{enumerate}
\end{defn}

In the setting of Definition \ref{def:pairwise_singular_sol}, for simplicity we will often refer to $X^{(N)}$ as a weak solution, without making the tuple $(\Omega,\FF,\PP;X^{(N)}_0, W^{(N)}, X^{(N)})$ explicit. We will say that $X^{(N)}$ is a strong solution if it is adapted to the filtration generated by $(X^{(N)}_0, W^{(N)})$.

We will consider a more regular class of solutions.

\begin{defn}\label{def:kappa_m_sol}
	In the setting of Definition~\ref{def:pairwise_singular_sol}, given a pair $\kappa \in [1,\infty)$, $m\in [1,\infty]$, a weak solution $X^{(N)}$
	to \eqref{eq:sing_pairwise_system} 
	is said to be $(\kappa,m)$-regular if $X^{(N)}-X^{(N)}_0-W^{(N)} \in C^{\kappa-\var}_{T}~L^m_\omega$.
\end{defn}
\begin{rem}\label{rem:kappa_m_sol}
	By Jensen's inequality, if $X^{(N)}$ 
	is $(\kappa,m)$-regular, then necessarily it is also $(\kappa,\tilde m)$-regular for any $\tilde m\in [1,m]$.
	Similarly, by virtue of \eqref{eq:inequality_variations}, if $X^{(N)}$	is $(\kappa,m)$-regular, then it is also $(\tilde\kappa,m)$-regular for any $\tilde\kappa>\kappa$.
\end{rem}

The next statement provides already a partial justification for the use of $(\kappa,m)$-regular solutions; it shows shows that, under our regularity assumptions on $b$, for $(\kappa,2)$-regular solutions,
the limit in \eqref{eq:defn_pairwise_singular_sol} holds in a stronger sense.
\begin{lem}\label{lem:k_m_convergence}
	Let $N\geq 2$, $(\alpha, q,H)$ satisfy Assumption \ref{ass:drift_assumption} with $\alpha <0$, $\kappa$, $\eps$  defined by \eqref{eq:defn_kappa_eps} and $b\in L^q_T \mcC^{\alpha}_x$.
	Let $m\in [2,\infty)$ and let $X^{(N)}$ be $(\kappa,m)$-regular weak solution to \eqref{eq:sing_pairwise_system}.
	Then, for any sequence ${\tilde b^n}_n\subset C^\infty_b$ such that $\tilde b^n \to b \in L^q_T \cC^\alpha_x$, it holds that 
	\begin{equation}\label{eq:k_m_p_system_limit}
		\lim_{n\to \infty} \max_{i=1,\ldots N} \bigg\| \theta^{i;N}_{\,\cdot\,} - X^i_0 - \frac{1}{N-1}\sum_{j\neq i} \int_0^{\,\cdot\,} \tilde b^n_s(X^{i;N}_s,X^{j;N}_s) \dd s\, \bigg\|_{C^{\kappa-\var}_T L^m_\omega} = 0.
	\end{equation}
\end{lem}
\begin{proof}
	Let $i\in \{1,\ldots,N\}$ be fixed. For any $j\neq i$, by assumption $(\theta^{i;N},\theta^{j;N})\in \dot C^{\kappa-\var}_T L^m_\omega$; 
	thus we can apply Lemma~\ref{lem:estimate_integrals_v0} (for $\ell=2d$, $\varphi=(\theta^{i;N},\theta^{j;N})$, $W=(W^i,W^j)$) to deduce that the process formally given by 
	\begin{align*}
		Z^{i,j;N} = \int_0^\cdot b_s(X^{i;N}_s,X^{j;N}_s) \dd s
	\end{align*}
	is a well-defined element of $C^{\kappa-\var}_T L^m_\omega$, and the limit of the corresponding processes associated to any sequence $\{\tilde b^n\}_n\subset C^\infty_b$ such that $\tilde b^n\to b$ in $L^q_T \cC^\alpha_x$.
	This applies in particular to the sequence $\{b^n\}_n$ associated to $X^{(N)}$ being a weak solution, namely the one for which \eqref{eq:defn_pairwise_singular_sol} holds; we deduce that
	\begin{align*}
		\theta^{i;N} = X^i_0 + \frac{1}{N-1} \sum_{j\neq i} Z^{i,j;N} \quad \forall\, i=1,\ldots, N.
	\end{align*}
	Finally, by applying Lemma~\ref{lem:estimate_integrals_v0} with $h = b-b^n$ to each $Z^{i,j;N}$, we deduce that
	\begin{align*}
		\bigg\| \theta^{i;N}_{\,\cdot\,} &- X^i_0 - \frac{1}{N-1}\sum_{j\neq i} \int_0^{\,\cdot\,} \tilde b^n_s(X^{i;N}_s,X^{j;N}_s) \dd s\, \bigg\|_{C^{\kappa-\var}_T L^m_\omega}\\
		& \leq \frac{1}{N-1}\sum_{j\neq i} \bigg\| Z^{i,j;N} - \int_0^{\,\cdot\,} \tilde b^n_s(X^{i;N}_s,X^{j;N}_s) \dd s\, \bigg\|_{C^{\kappa-\var}_T L^m_\omega}
		\lesssim \| b-\tilde b^n\|_{L^q_T \cC^\alpha_x}
	\end{align*}
	which yields the desired limit \eqref{eq:k_m_p_system_limit}.
\end{proof}
We are now in a position to state a well-posedness result for the particle system \eqref{eq:sing_pairwise_system}, under our standing assumptions in the case of distributional drift.
Pathwise uniqueness is shown in the restricted class of $(\kappa,2)$-regular solutions; this is in a similar vein to previous results on singular SDEs driven by fBm, cf. \cite[Thm.~2.10]{AnRiTa2023} or \cite[Thm.~2.14]{butkovsky2023stochastic}.

\begin{thm}\label{thm:unique_pwise_sing_system}
	Let $N\geq 2$, $(\alpha, q,H)$ satisfy Assumption \ref{ass:drift_assumption} with $\alpha <0$, $b \in L^q_T \cC^\alpha_x$ and $\kappa$, $\eps$ be defined by \eqref{eq:defn_kappa_eps}.
	Then the following hold:
	\begin{enumerate}[label=\roman*)]
		\item\label{it:unique_pwise_sing_it1} There exists a strong solution $X^{(N)}$ to \eqref{eq:sing_pairwise_system}, which is $(\kappa,m)$-regular for any $m\in [2,\infty)$.
		Furthermore, $X^{(N)}$ is the unique limit, in the $L^m_\omega C_T$-topology, of the solutions $X^{(N);n}$ to \eqref{eq:p_system_pairwise} associated to the same data $X^{(N)}_0$ and smooth interaction $b^n$, for any sequence $\{b^n\}_n\subset C^\infty_b$ such that $b^n\to b$ in $L^q_T \cC^\alpha_x$.
		\item\label{it:unique_pwise_sing_it2} Pathwise uniqueness holds in the class of $(\kappa,2)$-solutions. Namely, given two weak solutions $X^{(N)}$ and $\tilde X^{(N)}$, defined on the same filtered space $(\Omega,\FF,\PP)$, driven by the same $\FF$-fBm $W^{(N)}$ and with same initial data $X^{(N)}_0$, such that $X^{(N)}$ and $\tilde X^{(N)}$ are $(\kappa,2)$-regular, then necessarily $X^{(N)}$ and $\tilde X^{(N)}$ are indistinguishable.
	\end{enumerate}
	Further, let $X^{(N)}$, $\tilde X^{(N)}$ be strong solutions to \eqref{eq:sing_pairwise_system} as constructed in Point~\ref{it:unique_pwise_sing_it1}, associated to distinct data $(X^{(N)}_0,b)$ and $(\tilde X^{(N)}_0,\tilde b)$, but driven by the same $W^{(N)}$; then 	for any $m\in [1,\infty)$ and any $\tilde \kappa>\kappa$, there exists $C\coloneqq C(\alpha,H,q,d,T,m,\|b\|_{L^q_T\cC^\alpha_x},\|\tilde b\|_{L^q_T\cC^\alpha_x},\tilde\kappa)>0$ such that
	\begin{equation}\label{eq:stability_sing_pairwise}
		\sup_{j=1,\ldots,N}
		\Big\|  \| X^{j;N}-\tilde X^{j;N}\|_{C^{\tilde\kappa}_T} \Big\|_{L^m_\omega}
		\le C \bigg(\sup_{j=1,\ldots,N}\|X^{j;N}_0-\tilde X^{j;N}_0\|_{L^m_\omega}+\| b-\tilde{b}\|_{L^q_T \cC_x^{\alpha-1}}\bigg).
	\end{equation}
\end{thm}

We split the proof of Theorem~\ref{thm:unique_pwise_sing_system} in several steps, which are accomplished in Lemmas~\ref{lem:unique_pwise_sing_system_A}, \ref{lem:unique_pwise_sing_system_B} and \ref{lem:unique_pwise_sing_system_C} below.

\begin{lem}\label{lem:unique_pwise_sing_system_A}
	Let $N\geq 2$, $(\alpha, q,H)$ satisfy Assumption \ref{ass:drift_assumption} with $\alpha <0$, $b \in L^q_T \cC^\alpha_x$ and $\kappa$, $\eps$ be defined by \eqref{eq:defn_kappa_eps}.
	Then there exists a strong solution $X^{(N)}$ to \eqref{eq:sing_pairwise_system}, which is $(\kappa,m)$-regular for any $m\in [2,\infty)$. Moreover $X^{(N)}$ is the limit of any sequence $X^{(N);n}$ of solutions associated to regular interactions $b^n$, in the sense described in Point~\ref{it:unique_pwise_sing_it1} of Theorem~\ref{thm:unique_pwise_sing_system}.
\end{lem}

\begin{proof}
	As usual, we drop the superscript $N$ whenever clear throughout the proof.
	
	Given $b\in L^q_T \mcC^{\alpha}_x$, let $\{b^n\}_{n\geq 1}\subset C^{\infty}_b$ be a fixed sequence such that $b^n\to b$ in $L^q_T \mcC^{\alpha}_x$ (we can always find at least one);
	for each $n$, denote by $X^{(N);n}$ the strong solution to \eqref{eq:p_system_pairwise} with $b$ replaced by $b^n$ (whose existence is classical, cf. also Lemma~\ref{lem:app_lip_MKV_propagation}).
	For each $n\geq 1$ and $i=1,\ldots,N$, we define $\theta^{(N);n}= X^{(N);n}-W^{(N)}$.

	By Lemma~\ref{lem:remainder_singular_apriori} and Remark~\ref{rem:a_priori_norm_relations}, for any $m\in [1,\infty)$, it holds that
	\begin{equation}\label{eq:sing_theta_a_priori}
		\max_{i=1,\ldots,N}\|\theta^{i;n}_{s,t}\|_{L^m_\omega}
		\lesssim  \bigg( \int_s^t \|b^n_r\|^q_{\mcC^{\alpha}_x}\dd r \bigg)^{\frac{1}{q}} |t-s|^{\alpha H+1/q'} =: w_{b^n}(s,t)^{\frac{1}{q}} |t-s|^{\alpha H+1/q'}
	\end{equation}
	and moreover $\theta^{(N);n}$ satisfies \eqref{eq:structure_assumption_processes}, with control $w_{b^n}$.
	By Corollary~\ref{cor:p_system_drift_stable}, for any $n'$, $n$, we have that
	\begin{equation}\label{eq:sing_p_system_cauchy}
		\max_{i=1,\ldots,N}
		\Big\| \sup_{t\in [0,T]} \left|X^{i;n'}_t-X^{i;n}_t\right|\Big\|_{L^m_\omega}
		\lesssim \| b^{n'}-b^{n}\|_{L^q_T \cC_x^{\alpha-1}},
	\end{equation}
	where we used the fact that $\sup_n \| b^n\|_{L^q_T \cC^\alpha_x}$ is finite since $b^n\to b$ in $L^q_T \cC^\alpha_x$. 
	
	By construction, the sequence $\{b^n\}_{n\geq 1}$  is Cauchy in $L^q_T \mcC^{\alpha}_x \hookrightarrow L^q_T \cC^{\alpha-1}_x$,
	therefore by \eqref{eq:sing_p_system_cauchy} the sequence $\{X^{(N);n}\}_{n\geq 1}$ is Cauchy in $L^m_\omega C_T$; 
	it must have a unique limit in $L^m_\omega C_T$, let us call it $X^{(N)}$.
	We claim that, for any $m\in [1,\infty)$, $X^{(N)}$ is a $(\kappa,m)$-regular strong solution to \eqref{eq:sing_pairwise_system}, in the sense of Definitions~\ref{def:pairwise_singular_sol}-\ref{def:kappa_m_sol}.
	
	Firstly, since $X^{(N);n}$ is adapted to the filtration generated by $X^{(N)}_0,W^{(N)}$ for each $n$, the same must hold for the limit $X^{(N)}$.
	Next, since $X^{(N);n}\to X^{(N)}$ in $L^m_\omega C_T$, $\theta^{(N);n}\to \theta^{(N)}$ in $L^m_\omega C_T$ as well, for $\theta^{(N)}:= X^{(N)}-W^{(N)}$; passing to the limit as $n\to\infty$ in \eqref{eq:sing_theta_a_priori}, we deduce that
	\begin{equation}\label{eq:sing_p_system_remainder}
		\max_{i=1,\ldots,N}\|\theta^{i}_{s,t}\|_{L^m_\omega}
		\lesssim  \bigg( \int_s^t \|b_r\|^q_{\mcC^{\alpha}_x}\dd r \bigg)^{\frac{1}{q}} |t-s|^{\alpha H+1/q'} =: w_b(s,t)^{\frac{1}{q}} |t-s|^{\alpha H+1/q'};
	\end{equation}
	it follows that $\theta^{(N)}\in \dot C^{\kappa-\var}_T L^m_\omega$ and it satisfies \eqref{eq:structure_assumption_processes} with $w=w_b$.
	It remains to check that $X^{(N)}$ is a solution, in the sense that \eqref{eq:defn_pairwise_singular_sol} holds; we show this exactly for the sequence $\{b^n\}_n$ used to construct $X^{(N)}$.
	
	Similarly to the proof of Lemma~\ref{lem:k_m_convergence}, it is convenient to introduce the processes
	\begin{align*}
		Z^{i,j} \coloneqq  \int_0^\cdot b_s(X^i_s,X^j_s)\dd s
		= \lim_{n\to\infty} \int_0^\cdot b^n_s(X^i_s,X^j_s)\dd s;
	\end{align*}
	the first identity is formal, but the second is rigorous, with limit holding in the $C^{\kappa-\var}_T L^m_\omega$-topology by an application of Lemma~\ref{lem:estimate_integrals_v0}.
	Verifying \eqref{eq:defn_pairwise_singular_sol} amounts to showing that $\theta^i = X^i_0 + \frac{1}{N-1}\sum_{j\neq i} Z^{i,j}$; namely, to complete the proof, it suffices to show that
	\begin{equation}\label{eq:sing_p_system_goal}
		\lim_{n\to\infty} \Big\| Z^{i,j} - \int_0^\cdot b^n_s(X^{i;n}_s,X^{j;n}_s) \dd s \Big\|_{C^{\kappa-\var}_T L^m_\omega} = 0\quad \forall\, i\neq j.
	\end{equation}
	In order to check \eqref{eq:sing_p_system_goal}, we fix $i\neq j$ and take any $\bar{b}\in C^\infty_b$; it holds
	\begin{align*}
		Z^{i,j} - & \int_0^\cdot b^n_s(X^{i;n}_s,X^{j;n}_s) \dd s\\
		& = \Big( Z^{i,j} - \int_0^\cdot b^n_s(X^{i}_s,X^{j}_s) \dd s\Big) + \int_0^\cdot [\bar b_s- b^n_s](X^{i}_s,X^{j}_s) \dd s\\
		& \quad + \int_0^\cdot [\bar b_s(X^{i}_s,X^{j}_s) - \bar b_s(X^{i;n}_s,X^{j;n}_s)] \dd s
		+ \int_0^\cdot [\bar b_s- b^n_s](X^{i;n}_s,X^{j;n}_s) \dd s\\
		& \eqqcolon\, I^{1;n} + I^{2;n} + I^{3;n} + I^{4;n}.
	\end{align*}
	$I^{1;n}\to 0$ in $C^{\kappa-\var}_T L^m_\omega$ by Lemma~\ref{lem:estimate_integrals_v0}, while $I^{3;n}\to 0$ in the same topology because $\bar{b}$ is smooth and $X^{(N);n}\to X^{(N)}$ in $L^m_\omega C_T$.
	Since $\theta^{(N);n}$ and $\theta^{(N)}$ satisfy the bounds \eqref{eq:sing_theta_a_priori} and \eqref{eq:sing_p_system_remainder} respectively, we can apply Lemma~\ref{lem:estimate_integrals_v0} (with $h=b^n-\bar{b}$) to deduce that
	\begin{align*}
		\| I^{2;n} \|_{C^{\kappa-\var}_T L^m_\omega} + \| I^{4;n} \|_{C^{\kappa-\var}_T L^m_\omega} \lesssim \| b^n-\bar{b}\|_{L^q_T \cC^\alpha_x}
	\end{align*}		
	with implicit constant independent of $n$; combining all these facts and passing to the limit as $n\to\infty$ (using that $b^n\to b$ in $L^q_T \cC^\alpha_x$), we arrive at
	\begin{align*}
		\lim_{n\to\infty} \bigg\| \Big\| Z^{i,j} - \int_0^\cdot b^n_s(X^{i;n}_s,X^{j;n}_s) \dd s \Big\|_{C^{\kappa-\var}_T} \bigg\|_{L^m_\omega} \lesssim \| b-\bar{b}\|_{L^q_T \cC^\alpha_x}.
	\end{align*}
	Since $\bar{b}\in C^\infty_b$ was arbitrary and $C^\infty_b$ is dense in $L^q_T \cC^\alpha_x$, we conclude that \eqref{eq:sing_p_system_goal} holds, showing that $X^{(N)}$ is a weak solution.
	Combined with the aforementioned facts, we conclude that $X^{(N)}$ is actually a $(\kappa,m)$-regular, strong solution, for any $m\in [1,\infty)$.

	Finally, let $\{\tilde b^n\}_n\subset C^\infty_b$ be any other sequence such that $\tilde b^n\to b$ in $L^q_T \cC^\alpha_x$, and let $\tilde X^{(N);n}$ be the strong solution to \eqref{eq:p_system_pairwise} with $b$ replaced by $\tilde b^n$.
	Then $\tilde \theta^{(N)}$ will still satisfy uniform bounds of the form \eqref{eq:sing_theta_a_priori}, which in turn allows to apply Corollary~\ref{cor:p_system_drift_stable} to find that
	\begin{equation*}
		\max_{i=1,\ldots,N}
		\Big\| \sup_{t\in [0,T]} \left|X^{i;n}_t-\tilde X^{i;n}_t\right|\Big\|_{L^m_\omega}
		\lesssim \| b^{n}-\tilde b^{n}\|_{L^q_T \cC_x^{\alpha-1}},
	\end{equation*}
	with implicit constant independent of $n$.
	Since $b^n,\tilde b^n$ both converge to $b$ in $L^q_T \cC^\alpha_x$, $b^n-\tilde b^n\to 0$ in $L^q_T \cC_x^{\alpha-1}$, this implies that $X^{(N);n}$ and $\tilde X^{(N);n}$ converge to the same limit. In particular, $\tilde X^{(N);n}\to X^{(N)}$ in $L^m_\omega C_T$.
\end{proof}
Having constructed a strong solution, in order to prove pathwise uniqueness, we require the following intermediate lemma, strengthening the available bounds on $X^{(N)}$, in particular with respect to the reference filtration.
\begin{lem}\label{lem:unique_pwise_sing_system_B}
	Let $X^{(N)}$ be the strong solution from Lemma~\ref{lem:unique_pwise_sing_system_A}, and let $\FF$ be any filtration on $(\Omega,\FF,\PP)$ such that $W^{(N)}$ is an $\FF$-fBm. Then $\theta^{(N)}$ satisfies condition \eqref{eq:structure_assumption_processes} with respect to $\FF$.
\end{lem}

\begin{proof}
	Since by construction $X^{(N)}$ is the limit of solutions $X^{(N);n}$ associated to regular interactions $b^n$, passing to the limit as $n\to\infty$ in \eqref{eq:sing_apriori_bnd_1}, combined with \eqref{eq:a_priori_norm_relations}, we find that
	\begin{align*}
		\sup_{i=1,\ldots,N} \Big\| |\EE\big[\big| \theta^{i;N}_t - \EE[\theta^{i;N}_t | \cG_s]|\, \big|\cG_s\big]\Big\|_{L^\infty_\omega}
		\lesssim |t-s|^{\alpha H + \frac{1}{q'}} \Big( \int_s^t \| b_r\|_{\cC^\alpha_x}^q \dd r\Big)^{\frac{1}{q}},
	\end{align*}
	where we used the lower semicontinuity of the $L^\infty_\omega$-norm as in Lemma~\ref{lem:pairwise_mkv_apriori}.
	Here $\GG=\{\cG_t\}_{t\geq 0}$ denotes the natural filtration generated by $(X^{(N)}_0,W^{(N)})$, to which all the solutions  $X^{(N);n}$ are adapted.
	
	This estimate does not immediately yield the conclusion, since $\FF$ may be larger than $\GG$.
	Observe, however, that when proving the a priori bounds on $X^{(N);n}$ in particular within the proof of Lemma~\ref{lem:remainder_singular_apriori}, we never used the exact definition of the underlying filtration $\GG$.
	Instead, the proof only required that: i) $X^{(N)}$ and $W^{(N)}$ are adapted to the filtration; ii) $W^{(N)}$ is an $\GG$-fBm. This allows us to apply the LND property in conjunction with heat kernel estimates (cf. \eqref{eq:LND+heat}) to perform stochastic sewing arguments.
	Since by assumption both properties hold for $\FF$ as well, one can repeat the same argument to obtain estimates of the form \eqref{eq:structure_assumption_processes} for $\theta^{(N);n}$ and $\FF$; then we can pass to the limit $n\to\infty$ as above to conclude.
\end{proof}

\begin{lem}\label{lem:unique_pwise_sing_system_C}
	Under the assumptions of Theorem~\ref{thm:unique_pwise_sing_system}, pathwise uniqueness holds in the class of $(\kappa,2)$-regular solutions.
\end{lem}
\begin{proof}
	For simplicity, we drop the index $N$ whenever clear.
	
	In order to show pathwise uniqueness, we will actually prove the following: given any $(\kappa,2)$-regular weak solution $(\Omega,\FF,\PP;X^{(N)}_0,W^{(N)},Y^{(N)})$, on the same probability space there exists a strong solution $X^{(N)}$, such that $Y^{(N)}$ and $X^{(N)}$ are indistinguishable from $X^{(N)}$.
	
	In fact, in Lemma~\ref{lem:unique_pwise_sing_system_A} we constructed a strong solution (defined e.g. on the canonical space), which is a function of the initial data and driving noises (denote this function by $F$); we can therefore construct a copy of this strong solution on $(\Omega,\FF,\PP;X^{(N)}_0,W^{(N)},Y^{(N)})$, by taking $$X^{(N)}\coloneqq  F(X^{(N)}_0,W^{(N)}).$$
	By Lemma~\ref{lem:unique_pwise_sing_system_A}, $X^{(N)}$ is a $(\kappa,m)$-regular solution, for any $m\in [1,\infty)$, adapted to the natural filtration $\GG$ generated by $(X^{(N)}_0,W^{(N)})$; by Definition \ref{def:pairwise_singular_sol}, $\GG$ is a subfiltration of $\FF$ and $W^{(N)}$ is an $\FF$-fBm.
	We can therefore apply Lemma~\ref{lem:unique_pwise_sing_system_B} to deduce that $\theta^{(N)}\coloneqq  X^{(N)}-W^{(N)}$ satisfies \eqref{eq:structure_assumption_processes} with respect to $\FF$. 
	
	Since $Y^{(N)}$ is $(\kappa,2)$-regular, $\tilde{\theta}^{(N)}= Y^{(N)}-W^{(N)}\in \dot C^{\kappa-\var}_T L^2_m$.
	Comparing $\theta^{(N)}$ and $\tilde\theta^{(N)}$ on any interval $[s,t]\subset [0,T]$ we see that, for any $n\geq 1$ and $i=1,\ldots,N$,
	\begin{align*}
		\llbracket & \theta^{i} -  \tilde{\theta}^{i} \rrbracket_{C^{\kappa-\var}_{[s,t]}L^2_\omega}\\
		\leq &\,\bigg\llbracket\,\theta^{i}_{\variable} -X^i_0 -\frac{1}{N-1}\sum_{j\neq i}^N\int_{0}^{\,\cdot\,}  b^n_r (\theta^{i}_r+W^{i}_r,\theta^{j}_r+W^{j}_r) \dd r\,\bigg\rrbracket_{C^{\kappa-\var}_{[s,t]}L^2_\omega}\\
		&+\sup_{i\neq j \in \{1,\ldots,N\}}\left\llbracket\int_{0}^{\,\cdot\,} \left[ b^n_r \left(\theta^{i}_r+W^{i}_r,\theta^{j}_r+W^{j}_r\right)  -  b^n_r\left(\tilde{\theta}^{i}_r+W^{i}_r,\tilde{\theta}^{j}_r + W^{j}_r\right)\right]\dd r \right\rrbracket_{C^{\kappa-\var}_{[s,t]}L^2_\omega}\\
		&+ \bigg\llbracket\,\tilde{\theta}^{i}_{\variable}-X^i_0-\frac{1}{N-1}\sum_{j\neq i}^N\int_0^{\,\cdot\,} b^n_r(\tilde{\theta}^{i}_r+W^{i}_r,\tilde{\theta}^{j}_r + W^{j}_r)\dd r  \,\bigg\rrbracket_{C^{\kappa-\var}_{[s,t]}L^2_\omega}\\
		\eqqcolon&\,  \RN{1}_{s,t}^{i;n} + \RN{2}_{s,t}^{i;n} + \RN{3}_{s,t}^{i;n}.
	\end{align*}
	Since $X^{(N)}$ and $Y^{(N)}$ are both $(\kappa,2)$-regular, by Lemma~\ref{lem:k_m_convergence} we know that $\RN{1}^{i;n}_{s,t} + \RN{3}^{i;n}_{s,t} \to 0$ as $n\to\infty$, for all $i=1,\ldots,N$. 
	Concerning $\RN{2}^{i;n}$, we can go through similar arguments as in the proof of Proposition~\ref{prop:generic_pairwise_stabillity}.
	Indeed, since $\theta^{(N)}$ satisfies \eqref{eq:structure_assumption_processes} with $w=w_b$ and with respect to the filtration $\FF$ under which $\tilde \theta^{(N)}$ is adapted, we may apply Proposition~\ref{prop:comparison_integrals}, with $m=2$, $\varphi^1=(\theta^{i},\theta^{j})$ and $\varphi^2 = (\tilde{\theta}^{i},\tilde{\theta}^{j})$ to obtain that
	\begin{align*}
		\max_{i=1,\ldots,N}\RN{2}^{i;n}_{s,t} \lesssim w_b(s,t)^{\frac{1}{q}}|t-s|^\eps\max_{i=1,\ldots,N} \|\theta^{i}-\tilde{\theta}^{i}\|_{C^{\kappa-\var}_{[s,t]}L^2_\omega}\left(1+w_{b}(s,t)^{\frac{1}{q}}|t-s|^\eps\right),
	\end{align*}
	where the implicit constant is independent of $n$. Hence, taking the limit $n \to \infty$ on both sides, we find
	\begin{equation*}
		\max_{i =1,\ldots,N} \left\llbracket\theta^{i} - \tilde{\theta}^{i} \right \rrbracket_{C^{\kappa-\var}_{[s,t]}L^2_\omega}
		\lesssim w_b(s,t)^{\frac{1}{q}}|t-s|^\eps\max_{i=1,\ldots,N} \|\theta^{i}-\tilde{\theta}^{i}\|_{C^{\kappa-\var}_{[s,t]}L^2_\omega}\left(1+w_b(s,t)^{\frac{1}{q}}|t-s|^\eps\right).
	\end{equation*}
	A standard iteration argument on sufficiently small intervals $[s,t] \subset [0,T]$ then shows that 
	\begin{equation*}
		\max_{i =1,\ldots,N} \llbracket X^{i} - Y^{i} \rrbracket_{C^{\kappa-\var}_T L^2_\omega} 
		= \max_{i =1,\ldots,N} \llbracket\theta^{i} - \tilde{\theta}^{i} \rrbracket_{C^{\kappa-\var}_T L^2_\omega} =0
	\end{equation*}
	which yields the conclusion.
\end{proof}

\begin{proof}[Proof of Theorem~\ref{thm:unique_pwise_sing_system}]
	Point~\ref{it:unique_pwise_sing_it1} comes from Lemma~\ref{lem:unique_pwise_sing_system_A}, while Point~\ref{it:unique_pwise_sing_it2} from Lemma~\ref{lem:unique_pwise_sing_system_C}; thus it only remains to show the stability estimate~\eqref{eq:stability_sing_pairwise}.
	Given $b$, $\tilde b\in L^q_T \cC^\alpha_x$, we can find sequences $\{b^n\}_n$, $\{\tilde b^n\}_n$ such that $b^n\to b$ in $L^q_T \cC^\alpha_x$, similarly for $\tilde b$; by Point~\ref{it:unique_pwise_sing_it1}, the associated solutions $X^{(N);n}$ (resp. $\tilde X^{(N);n}$) converge to $X^{(N)}$ (resp. $\tilde X^{(N)}$).
	Estimate~\eqref{eq:stability_sing_pairwise} then follows by passing to the limit as $n\to\infty$ in the estimate~\eqref{eq:stability-particle-by-particle}, applied for $b^1=b^n$ and $b^2=\tilde b^n$.
\end{proof}

Having established a well-posedness result for the particle system \eqref{eq:sing_pairwise_system}, we can carry out a similar analysis for the associated singular McKean--Vlasov equation. We begin by defining a similar notion of solution for the McKean--Vlasov equation.

\begin{defn}\label{def:mkv_singular_sol}
	Let $\alpha<0$, $b\in L^q_T\cC^\alpha_x$. We say that a tuple $(\Omega,\FF,\PP;X_0, W, \bar X)$ is a weak solution on $[0,T]$ to the McKean--Vlasov equation
	\begin{equation}\label{eq:singular_mkv}
		\bar X_t = X_0 + \int_0^t b_s(\bar X_s,\bar{\mu}_s)\dd s + W_t,\qquad \mu_t = \cL(\bar X_t);
	\end{equation} 
	if the following hold:
	\begin{enumerate}[label=\roman*)]
		\item $(\Omega,\FF,\PP)$ is a filtered probability space, $W$ is an $\RR^d$-valued $\FF$-fBm and $X_0$ is $\cF_0$-adapted (in particular, this implies that $W$ and $X_0$ are independent);
		\item $\bar X$ is $\FF$-adapted;
		\item there exists a sequence $\{b^n\}_{n\geq 1}\subset C^\infty_b$ such that $b^n\to b$ in $L^q_T \cC^{\alpha}_x$ and
		\begin{equation}\label{eq:defn_mkv_sol}
			\bar{X}_{\,\cdot\,}-X_0-W_{\,\cdot\,}
			= \lim_{n\to \infty} \int_0^{\,\cdot\,} b^n_s(\bar{X}_s,\bar{\mu}_s)\dd s,\quad
			\bar{\mu}_t =\mcL(\bar{X}_t),
		\end{equation}
		where the limit is taken in probability with respect to $\PP$, uniformly on $[0,T]$.
	\end{enumerate}
\end{defn}
Clearly, we can readapt the concept of $(\kappa,m)$-regular solutions from Definition~\ref{def:kappa_m_sol} to the setting of equation~\eqref{eq:singular_mkv}; natural analogues of Remark~\ref{rem:kappa_m_sol} and Lemma~\ref{lem:k_m_convergence} then follow with identical proofs. The next statement is a direct parallel of Theorem~\ref{thm:unique_pwise_sing_system}.
\begin{lem}\label{lem:unique_pairwise_mkv}
	Let $N\geq 2$, $(\alpha, q,H)$ satisfy Assumption \ref{ass:drift_assumption} with $\alpha <0$, $b \in L^q_T \cC^\alpha_x$ and $\kappa$, $\eps$ be defined by \eqref{eq:defn_kappa_eps}.
	Then for any $X_0$ there exists a strong solution $\bar{X}$ to the McKean--Vlasov equation \eqref{eq:singular_mkv}, which is a $(\kappa,m)$-solution for any $m\in [2,\infty)$, and can be recovered as the unique limit of solutions $\bar{X}^n$ associated to regular approximations $\{b^n\}_n\subset C^\infty_b$ such that $b^n\to b$ in $L^q_T \cC^\alpha_x$.
	Furthermore, pathwise uniqueness holds in the class of $(\kappa,2)$-solutions.
	
	Finally, given two solutions $\bar X^1$, $\bar X^2$, associated to distinct data $(X^i_0,b^i)$ but driven by the same $W$, for any $m \in [1,\infty)$ and $\tilde \kappa>\kappa$ there exists a constant $$C\coloneqq C(\alpha,q,H,d,T,m,\|b^1\|_{L^q_T\cC^\alpha_x},\|b^2\|_{L^q_T\cC^\alpha_x},\tilde\kappa)>0$$ such that
	\begin{equation}\label{eq:singular_mkv_stability}
		\Big\| \|\bar X-\bar Y\|_{C^{\tilde \kappa-\var}_{T}} \Big\|_{L^m_\omega} \leq C\left(\|X_0-Y_0\|_{L^m_\omega}+ \| b^1-b^2\|_{L^q_T \cC^{\alpha-1}_x}\right).
	\end{equation}
\end{lem}
\begin{proof}
	The proof is almost identical to that of Theorem~\ref{thm:unique_pwise_sing_system}, so let us only shortly describe the main differences.
	
	To establish strong existence of a solution, which corresponds to the unique limit of any regular approximations, one goes through the same argument as in Lemma~\ref{lem:unique_pwise_sing_system_A}, up to replacing the stability estimates from Corollary~\ref{cor:p_system_drift_stable} with those of Proposition~\ref{prop:pairwise_mkv_stability}.
	
	As a byproduct of the construction by an approximation procedure, the strong solution $\bar{X}$ satisfies an estimate of the form \eqref{eq:mkv_apriori_alpha<0}, with respect to the filtration $\GG$ generated by $(X^{(N)}_0,W^{(N)})$.
	
	Arguing as in Lemma~\ref{lem:unique_pwise_sing_system_B}, this property can be upgraded to any other filtration $\FF$ such that $W$ is an $\FF$-fBm.
	Pathwise uniqueness then follows by comparing any other weak solution $\bar Y$ to $\bar X$; the argument is the same as in Lemma~\ref{lem:unique_pwise_sing_system_C}, up to applying the estimates on integrals coming from Lemma~\ref{lem:estimate_measure_flow_integrals} instead of Proposition~\ref{prop:comparison_integrals}.
	Finally, the stability estimate \eqref{eq:singular_mkv_stability} comes from the one for regular drifts of Proposition~\ref{prop:pairwise_mkv_stability} by passing to the limit with respect to regular approximations.
\end{proof}
\subsection{Solutions Concept and Well-Posedness for Functional Drifts}\label{sec:reg_drifts_sols}
In the case of functional drifts, i.e. $b\in L^1_T \mcC^\alpha_x$ with $\alpha \geq 0$, the notion of solution for both the particle system \eqref{eq:sing_pairwise_system} and the McKean--Vlasov equation \eqref{eq:singular_mkv} are unambiguous.
Indeed, integrals of the form
\begin{align*}
	\int_0^\cdot b_s(X^{i;N}_s,X^{j;N}_s) \dd s, \quad \int_0^\cdot b_s(\bar X_s,\bar\mu_s) \dd s,
\end{align*}
are always well-defined in the Lebesgue sense, and give rise to absolutely continuous paths.
Since they coincide with the usual ones, ee therefore refrain from giving complete definitions of weak solutions $(\Omega,\FF,\PP; X^{(N)}_0,W^{(N)},X^{N})$  (resp. $(\Omega,\FF,\PP; X_0,W,\bar X)$) in this context.
The same goes in this setting for the notions of strong solutions and pathwise uniqueness.

With these considerations in mind, we can give results analogous to (in fact, partially stronger than) Theorem~\ref{thm:unique_pwise_sing_system} and Lemma~\ref{lem:unique_pairwise_mkv}, guaranteeing well-posedness of the interacting particle system and McKean--Vlasov equations in the case of functional drifts.

\begin{prop}\label{prop:unique_cts}
	Let $N\geq 2$, $(\alpha, q,H)$ satisfy Assumption \ref{ass:drift_assumption} with $\alpha \geq 0$, $b \in L^q_T \cC^\alpha_x$. Then, for any $X^{(N)}_0$ (resp. $X_0$),
	strong existence and pathwise uniqueness hold for solutions $X^{(N)}$ (resp. $\bar{X}$) to the particle system \eqref{eq:sing_pairwise_system} (resp. the McKean--Vlasov equation \eqref{eq:singular_mkv}).
	Furthermore, such solutions are the unique limit in $L^m_\omega C_T$ of those associated to regular approximations $\{b^n\}_n\subset C^\infty_b$ such that $b^n\to b$ in $L^q_T \cC^\alpha_x$.
	Finally, stability estimates of the form \eqref{eq:stability_sing_pairwise} (resp. \eqref{eq:singular_mkv_stability}) are available also in this case.
\end{prop}
\begin{proof}
	The proof is almost identical to those of Theorem~\ref{thm:unique_pwise_sing_system} and Lemma~\ref{lem:unique_pairwise_mkv}, so let us only describe the key differences.
	Contrary to those statements, here we state uniqueness in the larger class of all weak solutions, without any need to restrict to $(\kappa,2)$-regular solutions; this is because, as we will show, any solution is automatically $(1,\infty)$-regular.
	
	We prove this claim for $X^{(N)}$, the argument for $\bar X$ being similar.
	Since $b\in L^q_T \cC^\alpha_x$, it also belongs to $L^1_T C^0_x$. Recall that $\theta^{(N)}=X^{(N)}-W^{(N)}$; arguing as in \eqref{eq:pointwise_holder} we have the $\PP$-a.s. bound
	\begin{align*}
		\sup_{i=1,\ldots, N} |\theta^{i;N}_{s,t}| \leq \int_s^t \|b_r\|_{C^0_x} \dd r \quad \forall\, (s,t)\in [0,T]^2_\leq.
	\end{align*}
	Observing that the right hand side defines a deterministic control, taking the $L^\infty_\omega$-norm on both sides and in view of Remark~\ref{rem:link_control_variation}, we deduce that $\theta^i \in  \dot C^{1-\var}_{[0,T]} L^\infty_\omega$; by Remark~\ref{rem:kappa_m_sol}, it follows in particular that $\theta^{(N)}\in \dot C^{\kappa-\var}_{[0,T]} L^m_\omega$ with $\kappa$ as in \eqref{eq:defn_kappa_eps} (see~\eqref{eq:inequality_variations}) and any $m\in [1,\infty]$.
	
	The proof of strong existence and characterization as the unique limit with respect to regular approximations proceeds as in Lemma~\ref{lem:unique_pwise_sing_system_A}; similarly, Lemma~\ref{lem:unique_pwise_sing_system_A} also holds in this context.
	Having established that any weak solution is $(1,\infty)$-regular, the proof of pathwise uniqueness proceeds identically as in Lemma~\ref{lem:unique_pwise_sing_system_C}, but without the need to assume $(\kappa,2)$-regularity.
	Stability estimates of the form \eqref{eq:stability_sing_pairwise}-\eqref{eq:singular_mkv_stability} are again a byproduct of those respectively from Corollary~\ref{cor:p_system_drift_stable} and Proposition~\ref{prop:pairwise_mkv_stability}, and a limiting procedure $b^n\to b$.
\end{proof}

We conclude this section with a small regularity result for solutions to \eqref{eq:singular_mkv}, valid for all $\alpha\in (-\infty,1)$.

\begin{lem}\label{lem:k_m_sol_time_reg}
	Let $(\alpha, q,H)$ satisfy Assumption \ref{ass:drift_assumption} with $\alpha \in \mbR$, $b \in L^q_T \cC^\alpha_x$; given $(X_0,W)$, let $\bar{X}$ denote the unique associated solution to the McKean-Vlasov equation \eqref{eq:singular_mkv}, coming from either Lemma~\ref{lem:unique_pairwise_mkv} (for $\alpha< 0$) or Proposition~\ref{prop:unique_cts} (for $\alpha\geq 0$).
	Then $\bar{X}=\bar{\theta}+W$, where:
	\begin{itemize}
		\item for $\alpha\geq 0$, one has $\bar{\theta}\in W^{1,q}_T$ with the pathwise bound
		\begin{align*}
			\llbracket \bar{\theta}\rrbracket_{W^{1,q}_T} \lesssim \| b\|_{L^q_T \cC^0_x}\quad \PP\text{-a.s.};
		\end{align*}
		\item for $\alpha<0$, for any $\iota>0$ it holds that $\theta\in W^{1+\alpha H-\iota,q}_T$, with moment bound
		\begin{align*}
			\Big\|\,\llbracket \bar{\theta}\rrbracket^m_{W^{1+\alpha H - \iota,q}_T}\Big\|_{L^m_\omega} \lesssim_{m,T,\iota} \| b\|_{L^q_T \cC^\alpha_x}\quad \text{for all }\,\, m \in [1,\infty).
		\end{align*}
	\end{itemize}	 
\end{lem}

\begin{proof}
	For $\alpha\geq 0$, since $\cC^\alpha_x\hookrightarrow C^0_x$, the pathwise bound follows directly from the estimate
	\begin{equation*}
		\llbracket \bar\theta\rrbracket_{W^{1,q}_T}^q
		= \int_0^T |\dot{\bar \theta}_t|^q \dd t
		= \int_0^T |b_t (X_t, \mcL(\bar{X}_t))|^q \dd t
		\leq \int_0^T \|b_t\|_{C^0_x}^q \dd t.
	\end{equation*}

	For $\alpha<0$, since by construction $\bar{X}$ is the unique limit in $C^{\kappa-\var}_T L^m_\omega$ of a sequence $\{\bar X^n\}_n$ of solutions associated to regular drifts $b^n\to b$ in  $L^q_T \cC^\alpha_x$, we can pass to the limit as $n\to\infty$ in the bound \eqref{eq:mkv_apriori_alpha<0} (using lower-semicontinuity of the $L^\infty_\omega$-norm as in Lemma~\ref{lem:pairwise_mkv_apriori});
	it follows that
	\begin{equation*}
		\left\| \|\bar\theta_{s,t}\|_{L^m|\cF_s} \right\|_{L^\infty_\omega} 
		\lesssim w_b(s,t)^{\frac{1}{q}} |t-s|^{\alpha H+\frac{1}{q^\prime}}
		\quad\forall\, (s,t)\in [S,T]^2_\leq
	\end{equation*}
	for some hidden constant depending on $\| b\|_{L^q_T \cC^\alpha_x}$ and the usual parameters, where $w_b(s,t) = \int_s^t \|b_s\|^q_{\mcC^{\alpha}_x}\dd s$. We can therefore apply Proposition~\ref{prop:kolmogorov_sobolev} with $w=w_b$ and $\beta=1+\alpha H$ to conclude that for any $\iota>0$,
	\begin{align*}
		\Big\|\,\|\llbracket \bar{\theta}\rrbracket^m_{W^{1+\alpha H - \iota,q}_T}\Big\|_{L^m_\omega}
		\lesssim_{m,T,\iota} w_b(0,T)^{\frac{1}{q}}
		=\| b\|_{L^q_T \cC^\alpha_x}\quad \text{for all }\,\, m \in [1,\infty). \qquad \qquad \qedhere
	\end{align*}
\end{proof}
\subsection{Propagation of Chaos for Distributional and Functional Drifts}\label{sec:poc_proofs}
Having established existence and uniqueness of solutions to particle systems and McKean--Vlasov equations, both in the case of distributional and continuous drifts, we are in a position to present (a slightly more refined version of) our main result.
\begin{thm}\label{th:pairwise_mean_field}
	Let $(\alpha, q,H)$ satisfy Assumption \ref{ass:drift_assumption} with $\alpha \in \mbR$, $b \in L^q_T \cC^\alpha_x$.
	Let $\{X^i_0\}_{i=1}^\infty$ be i.i.d. random variables and let $\{W^i\}_{i=1}^\infty$ be a collection of i.i.d. fBms of parameter $H$, independent from $\{X^i_0\}_{i=1}^\infty$.
	Given $N\geq 2$, denote by $X^{(N)}=\{X^{i;N}\}_{i=1}^N$ the unique strong solution to the interacting particle system \eqref{eq:sing_pairwise_system} and set $$\mu^N_t := \frac{1}{N}\sum_{i=1}^N\delta_{X^{i;N}_t};$$ define a collection of i.i.d. random variables $\{\bar X^i\}_{i=1}^\infty$ as the unique strong solutions to the McKean--Vlasov equations
	\begin{equation}\label{eq:iid_copies_mkv}
		\bar{X}^i_t =  X_0^i + \int_0^t b_s (\bar{X}^i_s,\bar{\mu}_s) \dd s +  W^i_t,
		\quad \bar{\mu}_t =\cL(\bar{X}^i_t),
		\quad \forall\, i\in \NN
	\end{equation}
	(observe that $\bar\mu_t$ does not depend on $i$).
	Then, for all $m\in [1,\infty)$, there exists a constant $C\coloneqq C(\alpha,q,H,d,T,m,\|b\|_{L^q_T\cC^\alpha_x})>0$ independent of $N$ such that
	\begin{equation}\label{eq:estim_particle_to_mkv_rough}
		\sup_{i=1,\ldots, N} \EE\left[\sup_{t\in [0,T]}\left| X^{i;N}_t - \bar{X}^i_t\right|^m\right]^{\frac{1}{m}} \leq C N^{-1/2}
	\end{equation}
	and
	\begin{equation}\label{eq:empirical_to_mean_field_rough}
		\EE\left[\sup_{t\in [0,T]}\left| \langle \varphi, \mu_t^N - \bar{\mu}_t\rangle\right|^m\right]^{\frac{1}{m}} \leq C \|\varphi\|_{W^{1,\infty}_x} N^{-1/2} \quad \forall\, \varphi\in W^{1,\infty}_x
	\end{equation}
\end{thm}
\begin{proof}
	Given $b\in L^q_T \cC^\alpha_x$, considering any sequence $\{b^n\}_n\subset C^\infty_b$ such that $b^n\to b$ in $L^q_T \cC^\alpha_x$, we know that the associated solutions $X^{(N);n}$ to \eqref{eq:sing_pairwise_system} converge to $X^{(N)}$ in $L^m_T C_T$.
	We can therefore pass to the limit as $n\to\infty$ in estimate \eqref{eq:estim-particle-to-mkv} from Corollary~\ref{cor:estim_particle_to_mkv}, using the fact that $\sup_n \| b^n\|_{L^q_T \cC^\alpha_x}<\infty$ and the constant therein is non-decreasing in $\|\cdot \|_{L^q_T \mcC^{\alpha}_x}$, to deduce that \eqref{eq:estim_particle_to_mkv_rough} holds.
	
	To prove \eqref{eq:empirical_to_mean_field_rough}, recall that $\bar{X}^{(N)}\coloneqq \{\bar{X}^{i;N}\}_{i=1}^N$ are i.i.d. random variables and set
	\begin{equation*}
		\bar{\mu}^N_t \coloneqq \frac{1}{N}\sum_{i=1}^N \delta_{\bar{X}^i_t};
	\end{equation*}
	we have the decomposition,
	\begin{equation*}
		\langle \varphi, \mu^N_t - \mu^X_t\rangle
		= \langle \varphi,\mu^N_t - \bar{\mu}^N_t\rangle + \langle \varphi,\bar{\mu}^N_t -\mu^X_t\rangle
		\eqqcolon I^{1;N}_t + I^{2;N}_t.
	\end{equation*}
	Since $\varphi$ is Lipschitz, we may apply \eqref{eq:estim_particle_to_mkv_rough} to see that
	\begin{align*}
		\Big\| \sup_{t\in [0,T]} |I^{1;N}_t| \Big\|_{L^m_\omega}
		\leq \frac{1}{N} \sum_{i=1}^N \Big\|  \sup_{t\in [0,T]} |\varphi(X^{i;N}_t) -\varphi(\bar X^i_t)| \Big\|_{L^m_\omega}
		\lesssim \| \varphi\|_{\dot W^{1,\infty}_x} N^{-1/2}.
	\end{align*}
	To estimate $I^{2;N}$, noting that $\cL(\bar{X}^i_t)= \bar{\mu}_t$ for all $i$, we can write
	\begin{align*}
		I^{2;N}_t = \frac{1}{N} \sum_{i=1}^N \Big(\varphi(\bar X^i_t) - \EE[\varphi(\bar{X}^i_t)]\Big).
	\end{align*}
	To conclude the proof, we will actually provide a stronger estimate on $I^{2;N}$; we claim that there exist $\beta\in (0,1)$, $p\in [2,\infty)$ with $\beta p>1$ such that
	\begin{equation}\label{eq:main_thm_proof_claim}
		\Big\| \| I^{2;N} \|_{W^{\beta,p}_T} \Big\|_{L^m_\omega} \lesssim N^{-1/2}
	\end{equation}
	where the hidden constant depends only on the usual parameters. The conclusion will then follow from the fractional Sobolev embedding $W^{\beta,p}_T \hookrightarrow C_T$, which implies a similar bound in $L^m_\omega C_T$.
	
	Set $$\bar Z^i:=\varphi(\bar X^i_t) - \EE[\varphi(\bar{X}^i_t)],$$ which is a collection of i.i.d., mean-zero random variables. In order to prove \eqref{eq:main_thm_proof_claim}, it suffices to show that, for $(\beta,p)$ as above, it holds that $\bar Z^1\in L^m_\omega W^{\beta,p}_T$.
	Indeed in that case, we can use the fact that for $p\in [2,\infty)$, $W^{\beta,p}_T$ is a martingale type $2$ space (see Appendix \ref{app:useful} for more details), and that $\bar Z^i$ are mean-zero and i.i.d., to conclude that $$Y_k:=\frac{1}{N}\sum_{i=1}^k \bar{Z}^i$$ is a martingale sequence in $W^{\beta,p}_T$; applying Lemma~\ref{lem:moments_martingale_type2} (with $E=W^{\beta,p}_T$), we see that
	\begin{align*}
		\Big\| \| I^{2;N} \|_{W^{\beta,p}_T} \Big\|_{L^m_\omega}
		\lesssim \frac{1}{N} 	\bigg( \sum_{i=1}^N \| \bar Z^i \|_{L^m_\omega W^{\beta,p}_T}^2 \bigg)^{1/2}
		\lesssim \|\bar Z^1\|_{L^m_\omega W^{\beta,p}_T} N^{-1/2} 
	\end{align*}
	yielding the desired \eqref{eq:main_thm_proof_claim}. So it only remains to show that $Z^1\in L^m_\omega W^{\beta,p}_T$ and estimate its norm. We perform the estimates for $\varphi(\bar X^1_t)$, the ones for its expectation being similar.
	Firstly, since $\varphi\in W^{1,\infty}_x$, using the definition of $W^{\beta,p}_T$, we have the pathwise bounds
	\begin{align*}
		\| \varphi(\bar X^1)\|_{L^p_T}
		\leq T^{1/p} \| \varphi(\bar X^1)\|_{L^\infty_T}
		\leq T^{1/p} \| \varphi\|_{L^\infty_x},\quad
		\bra{\varphi(\bar X^1)}_{W^{\beta,p}_T}
		\leq \| \varphi\|_{\dot W^{1,\infty}_x} \bra{\bar X^1}_{W^{\beta,p}_T}.
	\end{align*}
	Furthermore, since $\bar X^1= \bar\theta^1 + W^1$, $\bra{\bar X^1}_{W^{\beta,p}_T} \leq \bra{\bar \theta^1}_{W^{\beta,p}_T} + \bra{W^1}_{W^{\beta,p}_T}$; since $W$ is an fBm of parameter $H$ and $T$ is finite, $\bra{W^1}_{W^{\beta,p}_T}\in L^m_\omega$ for all $p\in [2,\infty)$ and all $\beta<H$. To estimate $\bra{\bar \theta^1}_{W^{\beta,p}_T}$, we rely on Lemma~\ref{lem:k_m_sol_time_reg}, distinguishing two cases:
	\begin{itemize}
		\item[i)] For $\alpha\geq 0$, we know that $\bar\theta^1-X_0\in L^\infty_\omega W^{1,q}_T$; since $q\in (1,2]$, by fractional Sobolev embeddings $W^{1,q}_T\hookrightarrow W^{\frac{1}{q'}+\frac{1}{p},p}$. The conclusion then follows by taking $p\in [2,\infty)$ large enough and $\beta<\min\{H, 1/q'+1/p\}$ such that $\beta p>1$; observe that this is always possible since $q>1$.
		\item[ii)] For $\alpha<0$, we know that $\bar\theta^1-X_0\in L^m_\omega W^{1+\alpha H-\iota,q}_T$ for any $\iota>0$; by Assumption \ref{ass:drift_assumption}, we can take $\iota>0$ small enough so that $1+\alpha H-\iota>H+1/q$. By fractional Sobolev embeddings it then holds $\bar\theta^1-X_0\in L^m_\omega W^{H+1/p,p}_T$, for all $p\in [2,\infty)$; we can then choose $p$ large enough and $\beta<\min\{H, H+1/p\}=H$ such that $\beta p>1$ to reach the conclusion. \qedhere
	\end{itemize}
\end{proof}
\begin{proof}[Proof of Theorem~\ref{thm:main}]
	Let $b\in L^q_T B^\alpha_{\infty,\infty}$, with $(\alpha,q,H)$ satisfying \eqref{eq:intro_assumption}; since the inequality is strict, by virtue of \eqref{eq:holder_complete_embed}, we can find $\tilde \alpha<\alpha$ such that $b\in L^q_T \cC^{\tilde \alpha}_x$ and $\tilde\alpha>1-1/(Hq')$ still holds. We are therefore in the framework of the results from this section.
	
	Points~\ref{it:intro_thm_i}-\ref{it:intro_thm_ii} are a consequence of the results gathered in Sections~\ref{sec:sing_drifts_sols}-\ref{sec:reg_drifts_sols}, most notably Theorem~\ref{thm:unique_pwise_sing_system}, Lemma~\ref{lem:unique_pairwise_mkv} and Proposition~\ref{prop:unique_cts}.
	
	In particular, for $\tilde\alpha<0$, it is now clear that we are adopting Definitions~\ref{def:pairwise_singular_sol}-\ref{def:mkv_singular_sol} as solutions concepts; uniqueness here may be either interpreted as pathwise uniqueness in the class of $(\kappa,2)$-regular solutions (cf. Lemma~\ref{lem:unique_pwise_sing_system_C}) or as the fact that $X^{(N)}$ (resp. $\bar X$) is the unique limit of solutions associated to regular approximations $b^n\to b$ (cf. Lemma~\ref{lem:unique_pwise_sing_system_A}).
	Instead for $\tilde\alpha\geq 0$, solutions are defined classically and classical pathwise uniqueness holds, by virtue of Proposition~\ref{prop:unique_cts}.
	
	Points~\ref{it:intro_thm_iii}-\ref{it:intro_thm_iv} instead come from Theorem~\ref{th:pairwise_mean_field}.
\end{proof}

Estimate \eqref{eq:empirical_to_mean_field_rough} provides a mean field limit result in which we first fix a reference observable $\varphi$, and then derive estimates in expectation.
Using similar arguments however, one can provide more intrinsic estimates, e.g. by means of negative Sobolev norms.

In the next statement, we use the inhomogeneous Sobolev spaces $H^{-\lambda}_x=B^{-\lambda}_{2,2}$; they can be equivalently characterized by the norm $\| f\|_{H^{-\lambda}_x} = \| (1+|\xi|)^{-\lambda} \hat f\|_{L^2}$, where $\hat f$ denotes the Fourier transform of $f$.
Let us note that the next result is general and rather simple, as it applies to any particle system for which an estimate of the form \eqref{eq:empirical_to_mean_field_rough} is available.

\begin{cor}\label{cor:mean_field_negative_sobolev}
	Let $(\alpha,q,H)$, $b$, $\{X^i_0\}_i$, $\{W^i\}_i$, $\{\bar{X}^i\}$, $\bar\mu$ be as in Theorem~\ref{th:pairwise_mean_field}.
	For given $N\geq 2$, define $X^{(N)}$ and $\mu^N$ as therein.
	Then, for any $m\in [1,\infty)$ and $\delta>0$, there exists a constant $C$, depending on $\delta$ and the usual parameters, such that
	\begin{align*}
		\Big\| \sup_{t\in [0,T]} \|\mu^N_t - \bar \mu_t \|_{H^{-d/2-1-\delta}_x} \Big\|_{L^m_\omega} \leq C N^{-1/2}.
	\end{align*}
\end{cor}

\begin{proof}
	As usual, we may assume $m\geq 2$.
	Let us set $\nu^N_t=\mu^N_t-\bar{\mu}_t$.
	By applying estimate \eqref{eq:empirical_to_mean_field_rough} for the choice $\varphi=e^{i\xi\cdot x}$, we find
	\begin{equation}\label{eq:mean_field_negative_sobolev_basic}
		\Big\| \sup_{t\in [0,T]} |\hat{\nu}^N_t(\xi)| \Big\|_{L^m_\omega} \lesssim (1+|\xi|) N^{-1/2} \quad \forall\, \xi\in\RR^d
	\end{equation}
	where $\hat{\nu}^N_t$ denotes the Fourier transforms of $\nu^N$. Therefore
	\begin{align*}
		\Big\| \sup_{t\in [0,T]} \| \nu^N_t \|_{H^{-d/2-1-\delta}_x} \Big\|_{L^m_\omega}
		& \leq \bigg\| \bigg( \int_{\RR^d} (1+|\xi|)^{-2-d-2\delta} \sup_{t\in [0,T]} |\nu^N_t(\xi)|^2 \dd \xi \bigg)^{1/2} \bigg\|_{L^m_\omega}\\
		& \leq \bigg( \int_{\RR^d} (1+|\xi|)^{-2-d-2\delta}\, \Big\| \sup_{t\in [0,T]} |\nu^N_t(\xi)|\, \Big\|_{L^m_\omega}^2 \dd \xi \bigg)^{1/2}\\
		& \lesssim N^{-1/2} \bigg( \int_{\RR^d} (1+|\xi|)^{-d-2\delta} \dd \xi \bigg)^{1/2}
		\lesssim N^{-{1/2}}.
	\end{align*}
	In the intermediate passage above, we applied Minkowski's inequality to bring the norm $\| \cdot\|_{L^m_\omega}$ inside, which is allowed since $m\geq 2$ and $\ell(\dd \xi)\coloneqq  (1+|\xi|)^{-2-d-2\delta} \dd \xi$ is a measure; then we employed estimate \eqref{eq:mean_field_negative_sobolev_basic}, and finally we used $\delta>0$ to guarantee convergence of the integral.
\end{proof}

Let us finally point out that we can combine Theorem~\ref{th:pairwise_mean_field} with our stability estimates, to derive quantitative rates for \emph{moderate interactions}; for simplicity, we focus solely on variants of \eqref{eq:estim_particle_to_mkv_rough}, although similar considerations apply also to \eqref{eq:empirical_to_mean_field_rough}.

\begin{cor}\label{cor:moderate_interactions}
	Let $(\alpha,q,H)$, $b$, $\{X^i_0\}_i$ and $\{W^i\}_i$ be as in Theorem~\ref{th:pairwise_mean_field}.
	Consider a family $\{b^\delta\}_{\delta >0}\subset L^q_T \cC^\alpha_x$ such that $b^\delta \to b$ in $L^q_T \cC^\alpha_x$ as $\delta\to 0$. 
	For any $N\geq 2$, denote by $X^{(N);\delta}$ the solution to \eqref{eq:sing_pairwise_system} associated to $b^\delta$ and by $\bar X^{(N)}$ the i.i.d. copies of the McKean--Vlasov equation \eqref{eq:iid_copies_mkv} associated to $b$.
	Then it holds that
	\begin{equation}\label{eq:moderate_interactions}
		\sup_{i=1,\ldots,N} \Big\| \sup_{t\in [0,T]}\left| X^{i;N,\delta}_t - \bar{X}^i_t\right| \Big\|_{L^m_\omega} \lesssim N^{-1/2} + \| b^\delta-b\|_{L^q_T \cC^{\alpha-1}_x}.
	\end{equation}
\end{cor}

\begin{proof}
	Estimate \eqref{eq:moderate_interactions} follows from triangular inequality, upon relying on the stability estimates \eqref{eq:stability_sing_pairwise} in order to compare $X^{(N);\delta}$ to $X^{(N)}$, and Theorem~\ref{th:pairwise_mean_field} to compare $X^{(N)}$ to $\bar X^{(N)}$.
\end{proof}

\begin{rem}
	If $b^\delta=\rho^\delta\ast b$ for some standard mollifier $\{\rho^\delta\}_{\delta>0}$, one has
	\begin{align*}
		\| b^\delta-b\|_{L^q_T \cC^{\alpha-1}_x} \lesssim \delta \| b\|_{L^q_T \cC^{\alpha}_x};
	\end{align*}
	the proof is standard and can be established e.g. by going through the same argument as in \cite[Lem.~A.7]{GaHaMa2023}. Combined with \eqref{eq:moderate_interactions}, in this case we find
	\begin{equation*}
		\Big\| \sup_{t\in [0,T]}\left| X^{1;N,\delta}_t - \bar{X}^1_t\right| \Big\|_{L^m_\omega} \lesssim N^{-1/2} + \delta.
	\end{equation*}
	One can e.g. recover the rate $N^{-1/2}$ by taking $\delta=\delta(N)\sim N^{-1/2}$.
\end{rem} 

\begin{appendix}
	\section{Interacting Particle Systems with Lipschitz Interactions}\label{app:lipschitz_p_system}
	We present here a classical propagation of chaos result for regular vector fields $b$.
	
	Let $\{X_0^i\}_{i=1}^\infty$ be i.i.d. $\RR^d$-valued random variables and $\{W^i\}_{i=1}^\infty$ be i.i.d. $\RR^d$-valued fBm with Hurst parameter $H\in (0,+\infty)\setminus \NN$. For any $N\geq 2$, consider the particle system $X^{(N)}=\{X^{i;N}\}_{i=1}^N$ given by
	\begin{equation}\label{eq:app_IPS}
		X^{i;N}= X^i_0 + \frac{1}{N-1}\sum_{\substack{j=1\\ j\neq i}}^N \int_0^t b_s(X^{i;N}_s,X^{j;N}_s) \dd s + W^i_t \quad \forall\, i=1,\ldots,N,
	\end{equation}
	the nonlinear process $X$ solving the McKean--Vlasov SDE
	\begin{equation}\label{eq:app_MKV}
		\bar{X}_t= X_0 + \int_0^t b_s(\bar{X}_s,\bar{\mu}_s) \dd s + W_t, \quad \bar{\mu}_t = \cL(X_t),
	\end{equation}
	as well as the i.i.d. copies $\{\bar{X}^i\}_{i=1}^\infty$ of $\bar{X}$, coupled to $X^{(N)}$ by solving
	\begin{align}\label{eq:app_pairwise_MKV}
		\bar{X}^i = X^i_0 +  \int_0^t b_s (\bar{X}^i_s,\bar{\mu}_s)\dd s + W^i_t, \quad \forall\, i=1,\ldots,N.
	\end{align}
	The next propagation of chaos statement is classical. 
	We include a proof to stress that the absence of self-interaction $i=j$ in \eqref{eq:app_IPS}, and $W^i$ not being sampled as Brownian motion, do not play any relevant role; moreover, there is no need to impose any finite moment assumption on $X^1_0$.
	\begin{lem}\label{lem:app_lip_MKV_propagation}
		Let $b\in L^1_T C^1_b$ and $\{X_0^i\}_i$, $\{W^i\}_i$ as above.
		Then, for each $N\geq 2$, there exist unique strong solutions $X^{(N)}$, $\bar{X}$ and $\bar{X}^i$ to \eqref{eq:app_IPS}, \eqref{eq:app_MKV} and \eqref{eq:app_pairwise_MKV}.
		Moreover, for any $m\in [1,\infty)$, there exists a constant $C\coloneqq C(m,T,\| b\|_{L^1_T C^1_b})>0$ such that
		\begin{equation}\label{eq:app_lip_MKV_propagation}
			\bigg\| \sup_{t\in [0,T]} | X^{1;N}_t-\bar{X}^1_t |\, \bigg\|_{L^m_\omega} \leq C N^{-1/2}.
		\end{equation}
		for all $N\geq 2$.
	\end{lem}
	\begin{proof}
		The existence and uniqueness statement is standard, so we only explain how to obtain \eqref{eq:app_lip_MKV_propagation}; given the similarities to \cite[Thm.~1.4]{sznitman1991topics}, we mostly sketch the argument.
		
		We may assume $m\geq 2$; we fix $N$ and for simplicity write $X^i$ in place of $X^{i;N}$ in what follows. Define the quantities
		\begin{align*}
			Q_t := \frac{1}{N}\sum_{j=1}^N 1\wedge |X^j_t-\bar{X}^j_t|, \quad
			R^i_t = \Big| \frac{1}{N-1}\sum_{j\neq i} [b(\bar{X}^i_t,\bar{X}^j_t)-b(\bar{X}^i_t,\bar{\mu}_t)] \Big|.
		\end{align*}
		By our assumption on $b$, it holds
		\begin{align*}
			|b_t(x,y)-b_t(x',y')| \lesssim \| b_t\|_{C^1_b} (1\wedge |x-x'| + 1\wedge |y-y'|)
		\end{align*}
		so that by \eqref{eq:app_IPS} and \eqref{eq:app_pairwise_MKV}, adding and subtracting $b_s(\bar{X}^i_s,\bar{X}^j_s)$, for any $\tau \in (0,T]$ we find
		\begin{align}
			\sup_{t\leq \tau}|X^i_t - \bar{X}^i_t|
			& \lesssim \| b_s\|_{C^1_b}\bigg[ \int_0^\tau (1\wedge |X^i_s-\bar{X}^i_s|) \dd s + \int_0^\tau \frac{1}{N-1}\sum_{j\neq i} (1\wedge |X^j_s-\bar{X}^j_s|) \dd s \bigg]+ \int_0^\tau R^i_s \dd s \nonumber \\
			& \lesssim \| b_s\|_{C^1_b} \bigg[ \int_0^\tau 1\wedge |X^i_s-\bar{X}^i_s| \dd s + \int_0^\tau Q_s \dd s\bigg] + \int_0^\tau R^i_s \dd s.\label{eq:app_propagation_eq1}
		\end{align}
		Since $\bar{X}^j$ and $\bar{X}^\ell$ are independent whenever $j\neq \ell$ and $\cL(\bar{X}^j_t)=\bar{\mu}_t$, for any fixed $t$ the process
		\[ Y^k_t= \sum_{\alpha=1, \alpha\neq i}^k b(\bar{X}^i_t,\bar{X}^\alpha_t)-b(\bar{X}^i_t,\mu_t)\]
		is a martingale with respect to the filtration $\cG_k=\sigma(X_t^\alpha: \alpha\in \{1,\ldots k\}\cup \{i\})$.
		By Lemma~\ref{lem:moments_martingale_type2} below, it holds
		\begin{align*}
			\| R^i_t\|_{L^m_\omega}
			= \frac{1}{N-1} \| Y^N_t\|_{L^m_\omega}
			\lesssim \frac{1}{N-1} \bigg( \sum_{j\neq i} \| b_t(\bar{X}^i_t,\bar{X}^j_t)-b_t(\bar{X}^i_t,\bar{\mu}_t)\|_{L^m_\omega}^2\bigg)^{1/2}
			\lesssim N^{-1/2} \| b_t\|_{C^0_b}.
		\end{align*}
		Taking the $L^m_\omega$-norm on both sides of \eqref{eq:app_propagation_eq1}, we then find
		\begin{equation}\label{eq:app_propagation_eq2}
			\big\| \sup_{t\leq \tau}|X^i_t - \bar{X}^i_t| \big\|_{L^m_\omega}
			\lesssim  \int_0^\tau \| b_s\|_{C^1_b}\Big[ \| 1\wedge |X^i_s-\bar{X}^i_s|\|_{L^m_\omega} +  \| Q_s\|_{L^m_\omega} + N^{-1/2}\Big]\dd s.
		\end{equation}
		Starting from \eqref{eq:app_propagation_eq2}, summing over $i$ and applying Minkowski's inequality, one finds an inequality where $\| Q_t\|_{L^m_\omega}$ appears on both sides, which by Gr\"onwall's inequality then yields
		\begin{equation}\label{eq:app_propagation_eq3}
			\sup_{t\in [0,T]} \| Q_t\|_{L^m_\omega}
			\lesssim \exp(C \| b\|_{L^1_T C^1_b})\, \| b\|_{L^1_T C^1_b}\, N^{-1/2}.
		\end{equation}
		Reinserting bound \eqref{eq:app_propagation_eq3} in \eqref{eq:app_propagation_eq2} and applying Gr\"onwall again finally gives \eqref{eq:app_lip_MKV_propagation}.
	\end{proof}

	\section{Some useful lemmas}\label{app:useful}
	
	We collect here some basic facts which are often used throughout the paper.
	
	\begin{lem}\label{lem:taylor}
		For any regular $f:\RR^d \to \RR^m$ and any $x_i\in\RR^d$, $i=1,\ldots,4$, it holds
		\begin{equation*}
			|f(x_1)-f(x_2)-f(x_3)+f(x_4)| \leq \| D_xf \|_{L^\infty_x} |x_1-x_2-x_3+x_4| + \| D^2_x f \|_{L^\infty_x} |x_1-x_2| |x_2-x_4|  
		\end{equation*}
	\end{lem}
	
	\begin{proof}
		Set $\tilde x_3:= x_1-x_2+x_4$; then by mean value theorem, it holds
		\begin{align*}
			|f(x_1)-f(x_2)-f(\tilde x_3)+f(x_4)|
			& \leq \int_0^1 |D_x f(x_2+\lambda (x_1-x_2)) - D_xf(x_4 + \lambda (x_1-x_2))| |x_1-x_2| \dd \lambda\\
			& \leq \| D^2_x f\|_{L^\infty_x} |x_1-x_2| |x_2-x_4|.
		\end{align*}
		On the other hand, $|f(x_3)-f(\tilde x_3)| \leq \| D_xf\|_{L^\infty_x} |x_1-x_2-x_3+x_4|$. Combining the above estimates with triangular inequality gives the claim.
	\end{proof}

	The next statement, Lemma~\ref{lem:moments_martingale_type2}, is a martingale inequality which is often useful in establishing quantitative convergence rates for law of large number type results on Banach spaces $E$.
	The lemma requires a special class of spaces $E$, namely those of martingale type $2$ (M-type $2$ for short); we refer to \cite{Hytonen2016} for the precise definition and an overview of their properties.
	Here we will only need the following:
	\begin{itemize}
		\item The Euclidean space $\RR^\ell$ (in fact, any Hilbert space) is of M-type $2$.
		\item By \cite[Prop.~7.1.4]{Hytonen2016}, if $(S,\cA,\mu)$ is an arbitrary measure space and $E$ is of M-type $2$, then $L^q(S,\cA,\mu;E)$ is of M-type $2$ for any $q\in [2,\infty)$.
		\item For any $\beta\in (0,1)$ and $q\in [2,\infty)$, $W^{\beta,q}([0,T];\RR^\ell)$ is of M-type $2$; this is because $f\in W^{\beta,q}([0,T];\RR^\ell)$ if and only if
		\begin{align*}
			f\in L^q([0,T];\RR^\ell), \quad (s,t)\mapsto \frac{f_t-f_s}{|t-s|^{\frac{1}{q}+\beta}}\in L^q([0,T]^2;\RR^\ell),
		\end{align*}
		so that $W^{\beta,q}([0,T];\RR^\ell)$ can be identified with a closed subspace of $L^q([0,T]^3;\RR^\ell)$.
	\end{itemize}
	
	\begin{lem}\label{lem:moments_martingale_type2}
		Let $E$ be a Banach space of M-type $2$ and $\{Y_j\}_{j=0}^n$ be an $E$-valued martingale sequence with $Y_0=0$. Then for any $m\in [2,\infty)$ there exists a constant $C\coloneqq C(m,E)>0$ such that for any $n\in\NN$ it holds
		\begin{equation*}
			\| Y_n \|_{L^m_\omega E} \leq C \bigg( \sum_{j=0}^{n-1} \| Y_{j+1}-Y_j\|_{L^m_\omega E}^2\bigg)^{1/2}.
		\end{equation*}
	\end{lem}
	
	\begin{proof}
		Set $\delta y_j= \| Y_{j+1}-Y_j\|_E$; since $E$ is of martingale type $2$, by \cite[Prop.~3.5.27]{Hytonen2016}, for any $n\in\NN$ it holds
		\begin{align*}
			\| Y_n \|_{L^m_\omega E}
			\leq C \bigg\| \Big( \sum_{j=0}^{n-1} \| Y_{j+1}-Y_j \|_E^2\Big)^{1/2} \bigg\|_{L^m_\omega} = C \big\| \{\delta y_\cdot\}_{j=0}^{n-1} \big\|_{L^m_\omega \ell^2}.
		\end{align*}
		Since $m\geq 2$, by Minkowski's inequality $\| \cdot \|_{L^m_\omega \ell^2} \leq \| \cdot \|_{\ell^2 L^m_\omega}$, which yields the conclusion.
	\end{proof}
	
	\section{Results on Besov--H\"older Spaces}\label{app:holder_besov}
	We provide here some results related to the examples of singular interactions $b$ covered by our main result, Theorem~\ref{thm:main}. Although classical, in the absence of a concise and self-contained reference, we have included them here for the sake of completeness.
	
	\begin{lem}\label{lem:besov_dimension_reduction}
		Let $\alpha\in \RR$ and $f \in B^\alpha_{\infty,\infty}(\mbR^d;\mbR^\ell)$.
		Then $$b^1(x,y)\coloneqq  f(x),\quad b^2(x,y)\coloneqq  f(y)$$ are well defined elements of $B^\alpha_{\infty,\infty}(\mbR^{2d};\mbR^\ell)$ and $\| b^i\|_{B^\alpha_{\infty,\infty}} \lesssim \| f\|_{B^\alpha_{\infty,\infty}}$.
		Similarly, $$b^3(x,y)\coloneqq  f(x-y)$$ is a well defined element of $B^\alpha_{\infty,\infty}(\mbR^{2d};\mbR^\ell)$, with $\| b^3\|_{B^\alpha_{\infty,\infty}} \lesssim \| f\|_{B^\alpha_{\infty,\infty}}$.
	\end{lem}
	
	\begin{proof}
		We first prove the statement for $b^1$, $b^2$ being similar. Up to reasoning componentwise, we may assume $\ell=1$.
		Given $\varphi\in \cS(\RR^{2d})$, we may define the pairing with $b^1$ by
		\begin{align*}
			\langle b^1,\varphi \rangle
			= \int_{\mbR^{2d}} b^1(x,y) \varphi(x,y)\dd x \dd y
			= \int_{\mbR^{d}} f(x) \bigg( \int_{\RR^d} \varphi(x,y) \dd y\bigg) \dd x
			= \langle f,\tilde\varphi^1 \rangle
		\end{align*}
		where we set $$\tilde\varphi^1(x)\coloneqq  \int \varphi(x,y) \dd y.$$
		it is easy to check that the linear map $\varphi\mapsto \tilde\varphi^1$ is continuous from $\cS(\RR^{2d})$ to $\cS(\RR^d)$, which in turn shows that $b^1$ is a well-defined element of $\cS'(\RR^{2d})$.
		
		Next, observe that, since $b^1(x,y)=f(x)\cdot 1$, its Fourier transform is given by $\hat b^1(\xi,\eta) = \hat f(\xi) \delta_0(\dd \eta)$.
		Similarly, the Littlewood--Paley (LP) decomposition $f=\sum_{j={-1}}^{\infty} \Delta_j f$ induces a corresponding decomposition
		\begin{equation}\label{eq:besov_proof1}
			b^1=\sum_{j\geq -1}^\infty b_j, \quad \text{for } b_j(x,y)\coloneqq  \Delta_j f(x).
		\end{equation}
		By the previous observation, $\hat{b}_j(\xi,\eta)= \widehat{\Delta_j f}(x) \delta_0(\dd \eta)$;
		therefore ${\rm supp}\, \hat{b}_j= {\rm supp}(\widehat{\Delta_j f}) \times \{0\}$, which by the definition of LP blocks $\Delta_j f$ implies the existence of an annulus $\cA\subset \RR^{2d}$ such that
		\begin{equation}\label{eq:besov_proof2}
			{\rm supp}\, \hat{b}_j \subset 2^j \cA\quad \forall\, j\geq 0,
		\end{equation}
		as well as a similar result on the support being contained in a ball for $\hat{b}_{-1}$.
		On the other hand, by construction
		\begin{equation}\label{eq:besov_proof3}
			\| b_j\|_{L^\infty(\RR^{2d})} = \| \Delta_j f\|_{L^\infty(\RR^{d})} \lesssim 2^{-j\alpha} \| f\|_{B^\alpha_{\infty,\infty}(\RR^d)}
		\end{equation}
		Combining \eqref{eq:besov_proof1}, \eqref{eq:besov_proof2} and \eqref{eq:besov_proof3} with \cite[Lem.~2.69]{BahCheDan} readily yields the conclusion in this case. 
		
		Consider now $b^3(x,y)=f(x-y)$; the proof that $b^3\in\cS'(\RR^{2d})$ is similar to the one for $b^1$.
		Let us define a linear isomorphism of $\RR^{2d}$ by $A:(x,y)\mapsto (u,v)$ with $u=x-y$, $v=x+y$.
		Simple computations show that $A^{-1}=\frac{1}{2} A^T$, $\det A=2^d$, $\det A^{-1}=2^{-d}$.
		Observe that, setting $z=(x,y)$, by construction $b^3(z)=b^1(Az)$. We claim that the composition with the linear isomorphism $A$ leaves the $B^\alpha_{\infty,\infty}$ spaces invariant. Indeed, given $\psi\in\cS(\RR^{2d})$, the Fourier transform of $\psi\circ A$ can be computed explicitly by
		\begin{align*}
			\widehat{\psi\circ A}(\zeta)
			= \int_{\RR^{2d}} \psi(A z) e^{-i\zeta\cdot z} \dd z
			= \det(A^{-1}) \int_{\RR^{2d}} \psi(z) e^{-i (A^{-T}\zeta)\cdot z} \dd z
			= 2^{-d} \hat\psi \Big(A^{-T}\zeta \Big);
		\end{align*}
		by duality, the same relation then holds for $\psi\in \cS'(\RR^{2d})$ as well. In particular, if $\hat\psi$ was supported on an annulus $\cA$, then $\widehat{\psi\circ A}$ is supported on the new annulus $\tilde \cA=A^T\cA$.
		
		Going through the same argument as before, we can then relate the LP decomposition of $b^1$ to a decomposition for $b^3$ of the form
		\begin{align*}
			b^3=\sum_{j\geq -1} u_j, \quad
			u_j(z) \coloneqq  \Delta b^1_j(A z),
			\quad {\rm supp }\, \hat{u}_j \subset 2^j \tilde \cA\quad \forall\, j\geq 0,
		\end{align*}
		and use the fact that
		\begin{align*}
			\| u_j\|_{L^\infty(\RR^{2d})}
			= \| \Delta b^1_j\|_{L^\infty(\RR^{2d})}
			\leq 2^{-\alpha j} \| b^1\|_{B^\alpha_{\infty,\infty}(\RR^{2d})}
			\lesssim 2^{-\alpha j} \| f\|_{B^\alpha_{\infty,\infty}(\RR^{d})}
		\end{align*}		
		to again achieve the conclusion by invoking \cite[Lem.~2.69]{BahCheDan}.
	\end{proof}

	In support of Example~\ref{ex:singular_riesz}, we need to establish a result on the Besov regularity for candidate distributions of the form
	\begin{equation}\label{eq:candidate_distribution}
		K(x) = g\left(\frac{x}{|x|} \right) |x|^{-s}
	\end{equation}
	for some smooth function $g:\RR^d\to\RR$ and some $s\geq d$; observe that since $K$ is not integrable in $0$, a priori it is not obvious that it defines a tempered distribution.
	
	To this end, we need to introduce some terminology from \cite{hormander_03_analysis_I}.
	In the next statement, $\cS(\RR^d\setminus\{0\})$ denotes the set of smooth functions $\varphi$ such that $\varphi(x)$ and all its derivatives decay to zero faster than any polynomial, both for $|x|\to 0$ and  $|x|\to \infty$; $\cS'(\RR^d\setminus\{0\})$ denotes its dual.
	
	\begin{defn}
		We say that a distribution $K \in \mcS'(\mbR^d\setminus \{0\})$ is homogeneous of degree $\sigma \in \mbR$ if for all $\psi\in \cS(\RR^d\setminus\{0\})$ and all $\lambda>0$ it holds that
		\begin{equation*}
			\langle K,\psi^\lambda\rangle = \lambda^{-\sigma} \langle K,\psi\rangle, \quad \text{where }\psi^\lambda(x)\coloneqq  \lambda^d \psi(\lambda x).
		\end{equation*}
		A similar definition holds for $K \in \mcS'(\mbR^d)$.
	\end{defn}
	In the next statement, $\dot B^{-s}_{\infty,\infty}$ denote homogeneous Besov--H\"older spaces, see \cite{BahCheDan}.
	\begin{lem}\label{lem:non_int_homogeneous}
		Let $s\in (d,\infty)\setminus\mbN$, $g$ be smooth and $K\in \mcS'(\mbR^d\setminus \{0\})$ be defined by \eqref{eq:candidate_distribution}.
		Then there exists a unique $(-s)$-homogeneous extension $\tilde{K}\in \mcS'(\mbR^d)$ of $K$; furthermore, $\tilde K\in \dot B^{-s}_{\infty,\infty}\hookrightarrow B^{-s}_{\infty,\infty}$.
	\end{lem}

	\begin{proof}
		By the structural assumption \eqref{eq:candidate_distribution}, it is easy to check that $K$ is a well-defined, $(-s)$-homogeneous element of $\cS'(\RR^d\setminus\{0\})$.
		We can therefore invoke \cite[Thm.~3.2.3]{hormander_03_analysis_I}, which guarantees that $K$ has a unique $(-s)$-homogeneous extension $\tilde K\in \cS'(\RR^d)$.
		On the other hand, one can construct such an extension by hand; indeed, for $\psi \in \mcS(\mbR^d)$, let us set
		\begin{equation}\label{eq:candidate_extension}
			\langle \tilde{K} , \psi\rangle \coloneqq \int_{\RR^d} K(x) \bigg( \psi(x)- \sum_{k=0}^{\floor{s}-d} \frac{1}{k!} D^k_x \psi(0): x^{\otimes k} \bigg) \dd x.
		\end{equation}
		By Taylor expansion, it is clear that the map
		\begin{align*}
			\psi(x)- \sum_{k=0}^{\floor{s}-d} \frac{1}{k!} D^k_x \psi(0): x^{\otimes k}
		\end{align*}
		decays at least like $|x|^{\floor{s}-d+1}$ for $x\sim 0$, which combined with \eqref{eq:candidate_distribution} implies that the integral appearing in \eqref{eq:candidate_extension} is finite around $0$.
		On the other hand, since $|K|\lesssim |x|^{-s}$ and we are multiplying by polynomials of degree $k\leq \floor{s}-d$, we see that $|K(x)\, x^{\otimes k}| \lesssim |x|^{ -\{s\}-d}$, which implies integrability as $|x|\to +\infty$ since $s$ is not an integer. Overall, this shows that the integral in \eqref{eq:candidate_extension} is well-defined.
		
		Verifying that $\tilde K\in\cS'(\RR^d)$, in the sense that it is linear and continuous, follows the same lines.
		Moreover, if $\psi\in\cS(\RR^d\setminus\{0\})$, then $D^k_x\psi(0)=0$ for all $k$ and thus $\langle \tilde K,\psi\rangle= \langle K,\psi\rangle$, so $\tilde K$ is indeed an extension of $K$.
		To verify $(-s)$-homogeneity, first observe that for any $\psi\in \cS(\RR^d)$, $D^k_x \psi^\lambda(0)=\lambda^{d+k} D^k_\psi(0)$; using formula \eqref{eq:candidate_distribution}, we have that
		\begin{align*}
			\langle \tilde K, \psi^\lambda \rangle
			& = \int_{\RR^d} K(x) \bigg( \psi^\lambda(x)- \sum_{k=0}^{\floor{s}-d} \frac{1}{k!} D^k_x \psi^\lambda(0): x^{\otimes k} \bigg) \dd x\\
			& = \int_{\RR^d} g\Big( \frac{x}{|x|} \Big) |x|^{-s} \lambda^d \bigg( \psi(\lambda x)- \sum_{k=0}^{\floor{s}-d} \frac{1}{k!} D^k_x \psi(0): (\lambda x)^{\otimes k} \bigg) \dd x\\
			& = \lambda^s \int_{\RR^d} g\Big( \frac{y}{|y|} \Big) |y|^{-s} \bigg( \psi(y)- \sum_{k=0}^{\floor{s}-d} \frac{1}{k!} D^k_x \psi(0): y^{\otimes k} \bigg) \dd y = \lambda^s \langle \tilde K, \psi\rangle.
		\end{align*}
		It remains to show the desired Besov regularity; to this end, recall that the homogeneous LP blocks $\{\dot \Delta_j\}_{j\in\ZZ}$ are obtained by setting $\dot \Delta_j f = \varphi_j\ast f$, where $\varphi_j(x)= 2^{jd} \varphi_0(2^{j} x)$, and $\varphi_0=\varphi$ is a Schwartz function with Fourier transform $\hat \varphi$ supported on an annulus $\cA$.
		Combining the scaling of $\varphi_j$ with the $(-s)$-homogeneity of $\tilde K$, one can see that
		\begin{align*}
			\dot \Delta_j \tilde K(x)
			= \langle \tilde K,\varphi_j(x-\cdot)\rangle
			= 2^{js} \langle \tilde K,\varphi(2^{-j} x-\cdot)\rangle
			= 2^{js} \dot \Delta_0 K(2^{-j} x)\quad \forall\, j\in\ZZ;
		\end{align*}
		as a consequence,
		\begin{align*}
			\| \tilde K\|_{\dot B^{-s}_{\infty,\infty}}
			= \sup_{j\in\ZZ} 2^{-js} \| \dot\Delta_j K\|_{L^\infty_x}
			= \| \dot\Delta_0 K\|_{L^\infty_x}
			= \sup_{x\in \RR^d} |\langle \tilde K, \varphi(x-\cdot) \rangle|.
		\end{align*}
		In particular, we only need to show that the last quantity is finite.
		To this end, note that the Taylor expansion argument that led to the meaningfulness of definition \eqref{eq:candidate_extension}, actually works for any $\psi\in C^{\floor{s}}_b$; since the associated norm is translation invariant, we have that
		\begin{align*}
			\sup_{x\in \RR^d} |\langle \tilde K, \varphi(x-\cdot) \rangle|
			\lesssim \sup_{x\in \RR^d} \| \varphi(x-\cdot)\|_{C^{\floor{s}}_b}
			= \| \varphi\|_{C^{\floor{s}}_b} <\infty.
		\end{align*}
		Concerning the inhomogeneous Besov regularity, we recall that $\dot\Delta_j = \Delta_j$ for all $j\geq 0$, while $\Delta_0=\sum_{j\geq 1} \dot\Delta_{-j}$; therefore we only need to additionally bound $\Delta_0 \tilde K$. Since $s>0$, it holds that
		\begin{align*}
			\| \Delta_0 \tilde K\|_{L^\infty_x}
			\leq \sum_{j=1}^{\infty} \| \dot\Delta_{-j} \tilde K\|_{L^\infty_x}
			\leq \| \tilde K\|_{\dot B^{-s}_{\infty,\infty}} \sum_{j=1}^{\infty} 2^{-js}
			\lesssim_s \| \tilde K\|_{\dot B^{-s}_{\infty,\infty}}
		\end{align*}
		which gives the conclusion.
	\end{proof}
	\begin{rem}\label{rem:int_homogeneous}
		Lemma~\ref{lem:non_int_homogeneous} does not cover the case of integer $s \in [d,\infty)\cap \mbN$, for good reasons.
		On one hand, \cite[Thm.~3.2.4]{hormander_03_analysis_I} still guarantees that homogeneous distributions $K\in\cS'(\RR^d\setminus\{0\})$ can be extended to bona fide distributions, $\tilde K\in \cS'(\RR^d)$.
		However, such extension may not be unique, nor $(-s)$-homogeneous, the standard example being $1/x$ on $\RR$.
		We refer to \cite[Section~3.3.2]{hormander_03_analysis_I} for further discussion on the topic.
	\end{rem}
\end{appendix}


\section*{Acknowledgments}

LG was supported by the SNSF Grant 182565 and by the Swiss State Secretariat for Education, Research and lnnovation (SERI) under contract number MB22.00034 through the project TENSE. KL was
supported by the Engineering and Physical Sciences Research Council (EPSRC), grant No. EP/Y016955/1. AM was supported through the  Simons Foundation (Award ID 316017), EPSRC (grant number EP/R014604/1) and DFG research unit FOR2402.

\bibliographystyle{alpha} 
\bibliography{all.bib}       

\end{document}